\documentclass[a4paper,11pt]{amsart}
\usepackage{amscd}
\usepackage{amsmath,amsthm,amsfonts,graphicx,amssymb}
\usepackage{hyperref}
\usepackage{color}
\usepackage[alphabetic,initials]{amsrefs}
\usepackage{tikz}
\newcommand{\mapsfrom}{\mathrel{\reflectbox{\ensuremath{\mapsto}}}}

\title[Conformal dimension and splittings]{Conformal dimension of hyperbolic groups that split over elementary subgroups}
\author{Matias Carrasco}
\email{mcarrasco@fing.edu.uy}
\address{Instituto de Matem\'atica y Estad\'istica Rafael Laguardia\\ Universidad de la Rep\'ublica}
\author{John M. Mackay}\thanks{This research was supported in part by EPSRC grant EP/P010245/1 and the MathAMSUD project Geometry and Dynamics of Infinite Groups.}
\email{john.mackay@bristol.ac.uk}
\address{School of Mathematics \\ University of Bristol \\ Bristol, UK}
\date{\today}

\numberwithin{equation}{section}
\newtheorem{theorem}[equation]{Theorem}
\newtheorem{proposition}[equation]{Proposition}
\newtheorem{corollary}[equation]{Corollary}
\newtheorem{lemma}[equation]{Lemma}
\newtheorem{sublemma}[equation]{Sublemma}

\newtheorem{definition}[equation]{Definition}
\newtheorem{remark}[equation]{Remark}

\newtheoremstyle{citing}
  {3pt}
  {3pt}
  {\itshape}
  {}
  {\bfseries}
  {}
  {.5em}
  {\thmnote{#3}}

\theoremstyle{citing}

\DeclareMathOperator{\diam}{diam}

\DeclareMathOperator{\Conv}{Conv}

\DeclareMathOperator{\Mod}{Mod}
\DeclareMathOperator{\Vol}{Vol}

\newcommand{\bdry}{\partial_\infty}

\newcommand{\cL}{\mathcal{L}}

\newcommand{\cB}{\mathcal{B}}

\newcommand{\cC}{\mathcal{C}}

\newcommand{\cW}{\mathcal{W}}

\DeclareMathOperator{\Confdim}{Confdim}

\newcommand{\cG}{\mathcal{G}}

\newcommand{\cR}{\mathcal{R}}
\newcommand{\cS}{\mathcal{S}}

\newcommand{\ra}{\rightarrow}
\newcommand{\R}{\mathbb{R}}

\newcommand{\N}{\mathbb{N}}
\newcommand{\Z}{\mathbb{Z}}

\begin{document}
\begin{abstract}
	We study the (Ahlfors regular) conformal dimension of the boundary at infinity of Gromov hyperbolic groups which split over elementary subgroups.
	If such a group is not virtually free, we show that the conformal dimension is equal to the maximal value of the conformal dimension of the vertex groups, or $1$, whichever is greater, and we characterise when the conformal dimension is attained.
	As a consequence, we are able to characterise which Gromov hyperbolic groups (without $2$-torsion) have conformal dimension $1$, answering a question of Bonk and Kleiner.
\end{abstract}



\maketitle

\section{Introduction}
\label{sec:intro}

\subsection{Overview}
The conformal dimension of a metric space, introduced by Pansu, is the infimal Hausdorff dimension of all the quasisymmetrically equivalent metrics on the space.
It is a natural quasisymmetric invariant, and is connected to the uniformisation problem of finding an optimal (``flattest'') metric for a given space.
Since the boundary at infinity $\bdry G$ of a Gromov hyperbolic group $G$ carries a canonically defined family of metrics that are pairwise quasisymmetric, the conformal dimension of $\bdry G$ is a well-defined quasi-isometric invariant of $G$. The initial motivation for the introduction of this invariant by Pansu in \cite{pansu1989dimension} was in the study of the large scale geometry of negatively curved homogeneous spaces, for which the conformal dimension can be computed explicitly. However, in general it is an invariant that is very hard to compute. Despite this difficulty, it has found applications in other areas of geometric group theory and dynamical systems. These include the work of Bonk and Kleiner on the rigidity of quasi-{M}\"{o}bius group actions \cite{Bonk-Kleiner-02-rigidity-QM-actions}; the works of Bonk and Kleiner \cite{Bonk-Kleiner-05-Confdim-Cannon} and Ha{\"\i}ssinsky \cite{haissinsky2015hyperbolic} on Cannon's conjecture and the boundary characterisation of Kleinian groups; the works of Ha{\"\i}ssinsky and Pilgrim on the characterisation of rational maps among coarse expanding conformal dynamical systems on the 2-sphere \cite{haissinsky-pilgrim-2014-minimal}; the works of Bourdon and Kleiner focussing on the relations between the $\ell_p$-cohomology, the conformal dimension, combinatorial modulus, and the Combinatorial Loewner Property \cite{bou-kle13-CLP,bourdon-kleiner-2015-some}; and the works of the second author on conformal dimension bounds for small cancellation and random groups \cite{mackay2012random,mackay2016random}, as well as further connections to actions on $L_p$-spaces \cite{bourdon2016cohomologie, DM-16-FLp}. We refer the reader to the survey \cite{MacTys10-confdim} for the basic theory of conformal dimension and its first applications.

In this paper we compute the conformal dimension of a hyperbolic group that splits as a graph of groups with elementary edge groups  in terms of the conformal dimensions of the resulting vertex groups.
Throughout the paper, an elementary (sub)group is a group that is finite or $2$-ended, i.e., virtually $\Z$.
Unless otherwise indicated, by `conformal dimension' we mean the now more commonly studied Ahlfors regular conformal dimension, see Definition~\ref{def:confdim}.
\begin{theorem}\label{thm:main}
	Suppose $G$ is a hyperbolic group, and we are given a graphs of groups decomposition of $G$, with vertex groups $\{G_i\}$ and all edge groups are elementary.
	Then if $G$ is not virtually free,
	\[
		\Confdim \bdry G = \max\Big\{ \{ 1 \} \cup \{\Confdim \bdry G_i : G_i \text{ infinite} \} \Big\}.
	\]
\end{theorem}
This theorem enables us to resolve a question of Bonk and Kleiner~\cite[Question 6.1]{Bonk-Kleiner-05-Confdim-Cannon}, characterising those hyperbolic groups which have conformal dimension equal to one (under the mild assumption of no $2$-torsion).
The conformal dimension of the boundary of a hyperbolic group is either $0$ or one of a dense set of values in $[1,\infty)$, and the groups whose boundaries have conformal dimension $0$ are exactly the virtually free groups (by Stallings--Dunwoody, see e.g.~\cite[Theorem 3.4.6]{MacTys10-confdim}).
Bonk and Kleiner's question therefore asks for a classification of the next fundamental case: conformal dimension $1$.  
Additional motivation for their question comes from the problem of understanding which hyperbolic groups attain their conformal dimension (see Subsection~\ref{ssec:intro-attain}): since Bonk--Kleiner had previously classified those hyperbolic groups attaining conformal dimension $1$, one can view our answer to their question as characterising those hyperbolic groups which have conformal dimension $1$, but do not attain it.  
In a different direction, as we discuss below, our work here also gives new kinds of self-similar metric spaces having conformal dimension $1$.
\begin{corollary}\label{cor:confdim1}
	If $G$ is a hyperbolic group with no $2$-torsion and not virtually free,
	then 
	$\Confdim \bdry G = 1$
	if and only if
	$G$ has a hierarchical decomposition over elementary edge groups
	so that each vertex group is
	elementary or virtually Fuchsian.
\end{corollary}

Let us now consider these results in more detail.
The case of Theorem~\ref{thm:main} when all the edge groups are finite is well known in the field (a proof may be found in the first author's thesis~\cite[Theorem 6.2]{carrascopiaggio-thesis}).
\begin{theorem}
	\label{thm:confdim-finite-edges}
	If $G$ is an infinite hyperbolic group with a finite graph of groups decomposition where the vertex groups are $\{G_i\}$ and the edge groups are finite, then
	\[
		\Confdim \bdry G = \max\{\Confdim \bdry G_i : G_i \text{ infinite} \},
	\]
	where $\max \emptyset = 0$.
\end{theorem}
In light of this result, Theorem~\ref{thm:main} reduces to the following:
\begin{theorem}\label{thm:main-precise}
	Suppose $G$ is a hyperbolic group with a graph of groups decomposition of $G$ with vertex groups $\{G_i\}$ and all edge groups $2$-ended, then if $G$ is not virtually free,
	\[
		\Confdim \bdry G = \max\Big\{ \{ 1 \} \cup  \{\Confdim \bdry G_i \} \Big\}.
	\]
\end{theorem}
\begin{proof}[Proof of Theorem~\ref{thm:main}]
	The lower bound for $\Confdim \bdry G$ is immediate: if we have $\Confdim \bdry G < 1$ then $\Confdim \bdry G = 0$ and $G$ is virtually free (see e.g.\ \cite[Theorem 3.4.6]{MacTys10-confdim}), thus $\Confdim \bdry G \geq 1$.  In addition, each $G_i$ is a quasiconvex subgroup of $G$ so each infinite $G_i$ has $\bdry G_i$ is quasisymmetrically embedded in $\bdry G$, therefore $\Confdim \bdry G \geq \Confdim \bdry G_i$.

	For the upper bound,
	amalgamate all edges with infinite edge groups to get a less refined graph of groups decomposition $\cG'$, where the conformal dimension of the new vertex groups has the bound from Theorem~\ref{thm:main-precise}.  Then as all edge groups in $\cG'$ are finite, the upper bound follows from Theorem~\ref{thm:confdim-finite-edges}.
\end{proof}
Particular cases of Theorems~\ref{thm:main} and \ref{thm:main-precise} were known before.
Keith and Kleiner in unpublished work~\cite{KeithKleiner} and Carrasco~\cite{carr14-cdim-split} showed that if $\bdry G$ has \emph{well spread} local cut points (``WS'' for short), then $\bdry G$ has conformal dimension $1$.  By saying $\bdry G$ has WS we mean that for some (any) fixed metric in the family, for any $\delta>0$ one can delete finitely many points from $\bdry G$ so that all remaining connected components have diameter at most $\delta$.  

As Theorem~\ref{thm:main} applies whether WS holds or not, we can complete the ``if'' direction of Corollary~\ref{cor:confdim1} characterising which hyperbolic groups have conformal dimension one.
The ``only if'' direction of Corollary~\ref{cor:confdim1} follows from work of the second author showing that hyperbolic groups with, for example, Sierpi\'nski carpet or Menger sponge boundaries have conformal dimension greater than one, and an accessibility result of Louder--Touikan~\cite{lou-tou-17-accessibility}.
\begin{proof}[Proof of Corollary~\ref{cor:confdim1}]
	Suppose $G$ admits a finite hierarchy of graph of groups decompositions over finite and $2$-ended subgroups, ending with vertex groups that are elementary or virtually Fuchsian; such groups have conformal dimension at most $1$.  Since $G$ is not virtually free we have $\Confdim \bdry G \geq 1$, and by repeatedly applying Theorem~\ref{thm:main} we have that $\Confdim \bdry G \leq 1$.

	Now for the converse, suppose $\Confdim \bdry G = 1$.
	As $G$ has no $2$-torsion, \cite[Corollary 2.7]{lou-tou-17-accessibility} implies that we can find a finite hierarchy for $G$ as follows: by Stallings and Dunwoody we can split $G$ maximally over finite edge groups leaving finite or one-ended vertex groups, then take the JSJ decomposition of the one-ended (hyperbolic) vertex groups, maximally splitting over $2$-ended subgroups, then repeat the Stallings--Dunwoody splitting for any vertex group with more than one end, and so on, repeating finitely many times until all the vertex groups remaining are either elementary, virtually Fuchsian groups, or one-ended groups that do not split over a $2$-ended subgroup.

	Each vertex group is quasiconvex in the original group $G$ as all splittings were over elementary subgroups.  The third case of one-ended, non-virtually Fuchsian groups with no splittings over a virtually $\Z$ subgroup cannot arise, as such groups have conformal dimension $>1$ by \cite[Corollary 1.3]{Mac10-cdim1}.
\end{proof}

\begin{remark}
	Corollary~\ref{cor:confdim1} holds also with the definition of conformal dimension as the infimal Hausdorff dimension of (not necessarily Ahlfors regular) quasisymmetrically equivalent metrics; let us denote this by $\Confdim_H$.
	First, if $G$ admits such a hierarchy and is not virtually free, $1 \leq \Confdim_H \bdry G \leq \Confdim \bdry G = 1$.
	Second, as the lower bound $>1$ from \cite{Mac10-cdim1} works for $\Confdim_H$ also, if $\Confdim_H \bdry G =1$ then all vertex groups in the hierarchical decomposition must be elementary or virtually Fuchsian as desired.
\end{remark}

\begin{remark}
	The groups considered in Corollary~\ref{cor:confdim1}, when torsion free, are the groups Wise suggests might be the hyperbolic virtual limit groups~\cite[Section 1.4]{Wise-19-v-limit-grps}.
\end{remark}

\subsection{Attainment of conformal dimension}\label{ssec:intro-attain}
It is natural to ask when the conformal dimension of a hyperbolic group is attained, i.e. when $\bdry G$ is quasisymmetric to an Ahlfors $Q$-regular space with $Q=\Confdim \bdry G$.
When this is satisfied $G$ often has rigidity properties, see the results and discussion in~\cite{Kle-06-ICM}.

Under the hypothesis of Corollary~\ref{cor:confdim1}, Bonk and Kleiner have shown that if a hyperbolic group $G$ has $\Confdim \bdry G =1$ and this is attained, i.e.\ if $\bdry G$ is quasisymmetric to an Ahlfors $1$-regular space, then $\bdry G$ is a circle and $G$ is virtually Fuchsian~\cite[Theorem 1.1]{Bonk-Kleiner-02-rigidity-QM-actions}.  

When we have a graph of groups as in Theorem~\ref{thm:main}, we can show the following.  
\begin{theorem}
	\label{thm:main-attained}
	Suppose $G$ is a hyperbolic group, and we are given a graph of groups decomposition of $G$ with vertex groups $\{G_i\}$ and all edge groups elementary.  Then the conformal dimension of $\bdry G$ is attained if and only if either:
	\begin{itemize}
		\item $\Confdim \bdry G = 0$ and $G$ is $2$-ended, or
		\item $\Confdim \bdry G = 1$ and $G$ is virtually cocompact Fuchsian, or
		\item $G = G_i$ for some vertex group with $\bdry G_i$ attaining its conformal dimension $\Confdim \bdry G_i>1$.
	\end{itemize}
\end{theorem}

The main idea here is that if the conformal dimension $\Confdim \bdry G$ is attained, then any ``porous'' subset has strictly smaller conformal dimension.  Since, by Theorem~\ref{thm:main}, $\Confdim \bdry G = \Confdim \bdry G_i$ for some vertex group $G_i$, and $G_i$ is a quasiconvex subgroup of $G$, the limit set $\Lambda G_i$ cannot be porous in $\bdry G$ and one can conclude that $G_i$ must equal $G$.

\subsection{Idea of proof and toy example}
\label{ssec:toyexample}
By work of Keith--Kleiner, Bourdon--Kleiner and the first author \cite{KeithKleiner,bou-kle13-CLP,carr13-conf-gauge}, the (Ahlfors regular) conformal dimension of the boundary of a hyperbolic group $X = \bdry G$ is equal to the critical exponent of the combinatorial modulus of the family of all curves in $X$ of diameter at least $\delta$, for some fixed small $\delta$.  
Prior to the works just cited, other authors who have used combinatorial modulus to study conformal dimension include Pansu~\cite{pansu1989dimension} and Keith--Laakso~\cite{Keith-Laakso-04-confdim}; see \cite{carr13-conf-gauge} for further discussion.

These notions are formally defined in Section~\ref{sec:defns-modulus}, but we can illustrate the idea here with a toy example.
Consider the double $G= \pi_1(S) *_\Z \pi_1(S)$ where $S$ is a closed surface of genus $2$, and $\Z$ corresponds to a closed geodesic curve $\gamma$ in $S$.
The boundary $\bdry G$ is (speaking informally) the limit of a countable collection of circles, corresponding to $\bdry \pi_1(S)$, glued at pairs of points, corresponding to $\bdry \Z$, in a tree-like fashion given by the Bass--Serre tree of the splitting.

The topological properties of the boundary depend on the type of curve $\gamma$ chosen.
If $\gamma$ is a simple closed curve, then Pansu observed that $\Confdim \bdry G = 1$ by varying the hyperbolic structure on $S$ to find CAT$(-1)$ model spaces for $G$ with volume entropy arbitrarily close to that of the hyperbolic plane; see discussion in~\cite[Section 6]{Bonk-Kleiner-05-Confdim-Cannon} and \cite[Theorem 1.1]{buyalo2005volume}.

If $\gamma$ is not simple, but not filling, one cannot use this argument. Recall that a curve $\gamma$ is \emph{filling} if all connected components in $S \setminus \gamma$ are topological discs, see Figure~\ref{fig:lifts} for an example of a filling curve.
However, the second author observed that such boundaries still satisfy the WS property, with cut points arising from limit points corresponding to an essential curve in $S\setminus \gamma$, and so $\Confdim \bdry G = 1$ here also.
For a complete characterisation of when $\bdry G$ has WS, including this case, see the work of the first author~\cite[Theorem 1.3]{carr14-cdim-split}.

The case when $\gamma$ is filling remained unresolved, but now we can apply Theorem~\ref{thm:main} to find that $\Confdim \bdry G = 1$.

\begin{figure}
\includegraphics[width=0.5\textwidth]{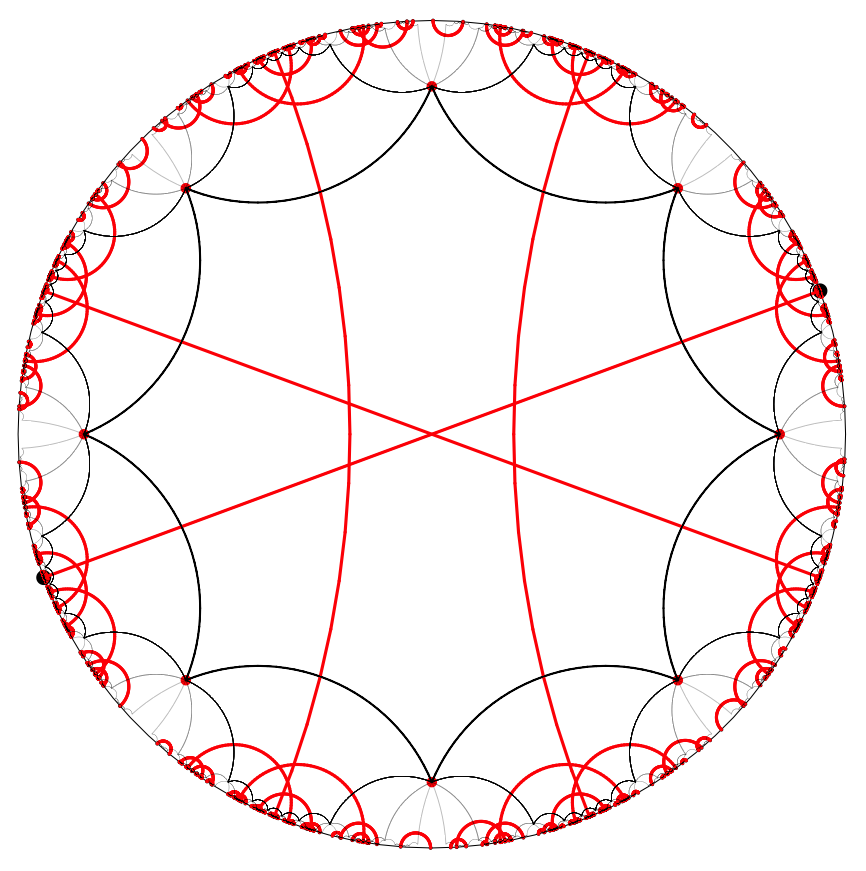}	
 \caption{\label{fig:lifts}Some lifts in the hyperbolic plane of the filling curve $abcd$ of the surface group $\langle a,b,c,d\ |\ [a,b][c,d]=1\rangle$.}
\end{figure}

To show how to prove this, we sketch the idea for a toy example which models $\bdry G$.
We build the space in stages, beginning by letting $X_0$ be a circle with length metric and of diameter $1$, and fix two antipodal points $x_-, x_+ \in X_0$.  
Define $X_1$ by taking $X_0$ and gluing on at pairs of points on say $12$ copies of $X_0$ scaled down by $1/3$ spaced around $X_0$ in an overlapping fashion.
For each $n=2,3,\ldots$, define $X_n$ in the following way: take a copy of $S^1$, and for each $j=1,\ldots,n$ glue on at pairs of points between $3^{j-1}$ and $12\cdot 3^{j-1}$ copies of $X_{n-j}$ scaled down by $1/3^j$, spaced around $S^1$.
We assume that these gluings are done in a self-similar way, so there is a natural limit space $X$ of this construction; see Figure~\ref{fig:toymodel} for a partial illustration of how $X_3$ is constructed. 
In the figure, circles are coloured black, blue, red, green. While the circles appear to overlap, a circle coloured blue, red or green meets no other circle of the same colour, and exactly one circle of some preceding colour at exactly one pair of points.
\begin{figure}
	\def\svgwidth{.7\textwidth}
   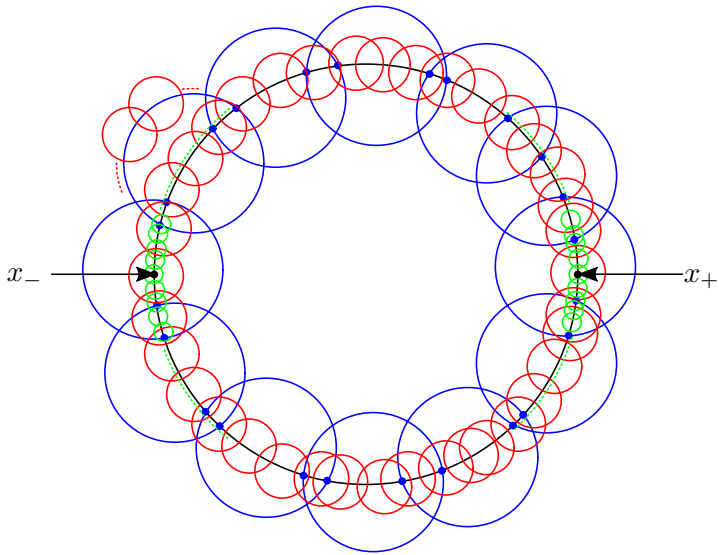
   \caption{\label{fig:toymodel} A toy model for the boundary of a surface group doubled along a filling curve}
\end{figure}

To show that $\Confdim X = 1$, since $\Confdim X \geq \dim_{top} X = 1$ is trivial, it suffices to show that $\Confdim X \leq p$ for an arbitrary $p>1$.
Using the machinery of Keith--Kleiner and Carrasco mentioned above, such a bound follows from a combinatorial modulus estimate on $X$.
Rather than considering all curves in $X$ of diameter $\geq \delta$, we simplify the argument here by considering the family of all paths in $X$ joining $x_-$ to $x_+$, which we call $\Gamma$.

For each $n \in \N$, let $\cS_n$ be the cover of $X$ by sets of size $3^{-n}$ corresponding to the copies of $X_0$ of size $3^{-n}$ in $X_n$.
A \emph{weight function} $\rho_n: \cS_n\to (0,\infty)$ is \emph{admissible} for $\Gamma$ if for any $\gamma \in \Gamma$, the \emph{$\rho_n$-length} $\ell_{\rho_n}(\gamma)$ satisfies
\[
  \ell_{\rho_n}(\gamma) := \sum_{A \in \cS_n : \gamma \cap A \neq \emptyset} \rho_n(A) \geq 1.
\]
Roughly speaking, a weight function describes a hoped for conformal deformation where the desired diameters of the images of $A \in \cS_n$ are the values of $\rho_n(A)$, and admissibility ensures that the image doesn't collapse down in size.
The \emph{$p$-volume} $\Vol_p(\rho_n)$ of $\rho_n$ is defined as
\[
	a_n := \Vol_p(\rho_n) := \sum_{A \in \cS_n} \rho_n(A)^p.
\]
To achieve the bound $\Confdim X \leq p$, we require a sequence $(\rho_n)$ of $\Gamma$-admissible weight functions so that $a_n = \Vol_p(\rho_n) \to 0$ as $n\to\infty$.

We now define $\rho_n: \cS_n \to (0,\infty)$ and estimate $a_n$ by induction.
The first step is easy: $\cS_0=\{A_0\}$ is a cover of $X_0$ by a single open set, and we let $\rho_0 : \cS_0 \to (0,\infty)$, $\rho_0(A_0)=1$, which is admissible and has $a_0 := 1$.

Now for the inductive step: assume that suitable $\rho_{i}$ have been defined for all $i=0,\ldots,n-1$.
The idea at step $n$ is that we send the geometric sequence of annuli $A^-_i:=B(x_-,3^{-i}/2)\setminus B(x_-,3^{-(i+1)}/2)$, for $i=0,\ldots,n-1$, to an arithmetic sequence of annuli each of size $1/2n$, and likewise for the annuli $A^+_i$ centred at $x_+$.
This will define an admissible weight function; see Figure~\ref{fig:toymodeldeform} for an illustration.  
\begin{figure}
\def\svgwidth{.95\textwidth}
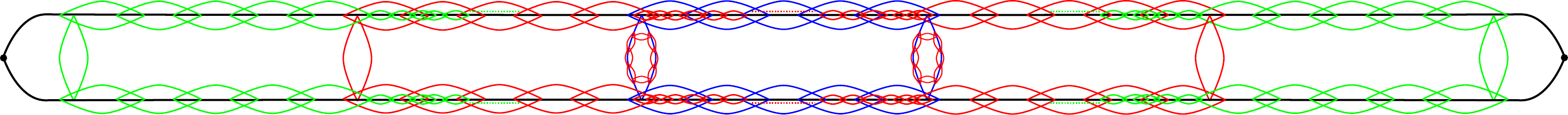
\caption{A cartoon of the weight function $\rho_3$}\label{fig:toymodeldeform} 
\end{figure}

Now, we describe $\rho_n$ in more detail (though not with an explicit formula), and we estimate its $p$-volume $a_n$.
For each $i=0,\ldots,n-1$, and each $j=i+1,\ldots,n$, the annuli $A^-_i, A^+_i$ contain in total $\leq C 3^{-i}/3^{-j} = C 3^{j-i}$ copies of $X_{n-j}$, which we endow with weights using $\rho_{n-j}$; here $C>1$ is a constant.
As we want these to have diameters totalling $\geq 1/n$ in the image, we apply a scaling factor of $1/( 3^{j-i} n )$ to these copies, which scales $a_{n-j}$ by $1/(3^{j-i} n)^p$.
Thus, summing these up and using geometric series bounds, we have
\begin{align*}
	a_n & \leq C \sum_{i=0}^{n-1} \sum_{j=i+1}^n 3^{j-i} \cdot \frac{a_{n-j}}{(3^{j-i} n)^p}
	\\ & = \frac{C}{n^p} \sum_{j=1}^n \sum_{i=0}^{j-1} 3^{-(j-i)(p-1)} a_{n-j}
	\\ & \leq \frac{C'}{n^p} \sum_{j=1}^n a_{n-j}
	\leq \frac{C'}{n^{p-1}} \max \{a_0,\ldots,a_{n-1} \},
\end{align*}
for some constant $C'>1$.
This inequality implies that for $n$ large, the sequence $(a_n)$ is nonincreasing, hence $(a_n)$ is bounded, hence the inequality again implies that $a_n \to 0$ as $n \to \infty$.
The proof is complete.

\medskip 

The general argument in the paper is more involved in several ways, but has the same key idea of deforming geometric sequences of annuli into arithmetic sequences at its foundation.
Many additional complications are laid on by incorporating deformations of $\bdry G_i$ which nearly achieve the conformal dimension of the boundaries of each space $\bdry G_i$, carefully checking admissibility (for all curves of given diameter, not just those joining two points), and setting up a suitable induction for the volume bounds.

\subsection{Outline of paper} 
In Section~\ref{sec:bdry-structure} we describe the metric properties of limit sets in hyperbolic groups with quasiconvex splittings.  In Sections~\ref{sec:defns-modulus}-\ref{sec:volume} we prove Theorem~\ref{thm:main-precise}: Section~\ref{sec:defns-modulus} reduces the theorem to a statement about combinatorial modulus, and in Section~\ref{sec:weight-def} a sequence of weight functions is defined. The weight functions are shown to have maximum values going to zero, to be admissible, and to have bounded volume in Sections~\ref{sec:maxbound}, \ref{sec:admissibility}, and \ref{sec:volume} respectively.  Finally, in Section~\ref{sec:attained} we consider attainment of conformal dimension and prove Theorem~\ref{thm:main-attained}. 

\subsection{Notation}
We write $A\preceq B$ if $A \leq C B$ for some constant $C>0$, and $A \asymp B$ if $A\preceq B$ and $B\preceq A$.  We may write $A \preceq_C B$ or $A \asymp_C B$ to indicate which $C$.  We also write $A \approx B$ if $|A-B|\leq C$ for some constant $C\geq 0$.
Throughout the paper, $C, C', C'', \ldots, C_1, C_2, \ldots $ refer to constants only depending on the relevant data; sometimes we make the dependence clear by writing $C=C(\alpha, \beta,...)$ and so on.
For $A,B \in \R$, we write $A\vee B := \max\{A,B\}$ and $A\wedge B:= \min\{A,B\}$.

\subsection{Acknowledgements}
The first author thanks Peter Ha{\"\i}ssinsky who first introduced him to this question many years ago. The second author thanks Bruce Kleiner and Daniel Meyer for helpful conversations about this question over the years.  Thanks go to Daniel Groves, Lars Louder and Henry Wilton for helpful comments.

\section{Graph of groups decompositions and boundaries}
\label{sec:bdry-structure}

In this section we present useful facts about boundaries and quasiconvex splittings of hyperbolic groups that will be used later. For references on graph of groups and Bass--Serre theory, see Serre~\cite{Ser80-trees}, Scott--Wall~\cite{ScottWall79-top} and Dru\c tu--Kapovich~\cite{DruKap18-ggt}.

\subsection{Boundaries of quasiconvex splittings}

An \emph{abstract (oriented) graph $\cG$} consists of two sets, the vertices $V\cG$ and the edges $E\cG$, with an initial vertex map $(\cdot)_-:E\cG\to V\cG, e\mapsto e_-$ and a terminal vertex map $(\cdot)_+:E\cG\to V\cG, e\mapsto e_+$.

Suppose $G$ acts on a tree $T$ without inversions on edges, minimally (i.e.\ there is no proper invariant sub-tree of $T$), and with the quotient graph $G\setminus T$ finite.
Any such action corresponds to a \emph{graph of groups decomposition $\cG$} for $G$ where the underlying graph is $G\setminus T$,
for each vertex $i \in V(G\setminus T)$ the vertex group is a copy of the stabilizer $G_v$ for some $v\in VT$ in the orbit corresponding to $i$,
for each edge $k \in E(G\setminus T)$ the edge group is a copy of the stabilizer $G_e$ for some $e\in ET$ in the orbit corresponding to $k$,
and the injective homomorphisms from edge groups into vertex groups are induced by the inclusions of stabilizers $G_e \to G_{e_-}, G_e \to G_{e_+}$.
We call $T$ the Bass--Serre tree for the graph of groups decomposition $\cG$.

As all the stabilizers in an orbit are conjugate, for $v \in VT$ we can define $i_v \in V\cG, g_v \in G$ so that $G_v = g_v G_{i_v} g_v^{-1}$,
and for $e \in ET$ we can define $k_e \in E\cG, g_e \in G$ so that $G_e = g_e G_{k_e} g_e^{-1}$.

We now build a model space $Z$ for $G$.
For each $i\in V\cG$ let $M_i$ be a presentation complex for $G_i$, so $M_i$ is a $2$-dimensional cell complex with $\pi_1(M_i)=G_{i}$.  Likewise for each $k\in E\cG$ let $M_k$ be a presentation complex for $G_k$.
The homomorphisms from edge groups to vertex groups are induced by continuous maps $f_{k_-}:M_k\to M_{k_-}, f_{k_+}:M_k\to M_{k_+}$ for $k\in E\cG$.
The graph of spaces $M$ is built from the collection $\{M_i\}_{i\in V\cG} \cup \{M_k\times[-1,1]\}_{k\in E\cG}$ where we glue each $M_k\times\{\pm 1\}$ to $M_{k_\pm}$ by the map $(z,\pm 1)\mapsto f_{k_\pm}(z)$.
By Bass--Serre theory, the fundamental group $\pi_1(M)$ equals $G$.

Define a length metric on $M$ which induces a geodesic metric on the universal cover $Z := \widetilde{M}$.
This space $Z$ is a tree of spaces with a copy $Z_v$ of $\widetilde{M_{i_v}}$ for each $v\in VT$ and a copy $Z_e \times [-1,1]$ of $\widetilde{M_{k_e}}\times [-1,1]$ for each $e \in ET$,
where the subset $\widetilde{M_{k_e}}\times\{\pm 1\}$ is glued into the corresponding vertex spaces.
The action $G \curvearrowright Z=\widetilde{M}$ preserves this tree-of-spaces structure, and so if we collapse each vertex space $Z_v$ to a point and each edge space $Z_e \times[-1,1]$ to an edge we recover our original tree $T$ and action $G \curvearrowright T$.

Fix a base vertex $v_0 \in T$ and a basepoint $o \in Z$ so that $Z\to T$ maps $o$ to $v_0$.
As $G$ acts geometrically on $Z$ the orbit map $G \to G\cdot o$ induces a quasi-isometry $G \to Z$.
This quasi-isometry coarsely maps each left coset $g_v G_{i_v}$, $v \in VT$ to $Z_v$, and likewise coarsely maps each $g_e G_{k_e}$, $e \in ET$, to $Z_e\times[-1,1]$.

In our case $G$ is a hyperbolic group, and so $Z$ is hyperbolic also.
We fix a visual metric $d$ on $X:=\bdry Z$ with visual parameter $\epsilon>0$, i.e.\ $d(\cdot,\cdot)\asymp e^{-\epsilon (\cdot|\cdot)_o}$, where $(\cdot|\cdot)_o$ denotes the Gromov product with basepoint $o$.
We may rescale to assume $\diam X = 1$.

For a subgroup $H$ of $G$, let $\Lambda(H) \subset X$ be the \emph{limit set of $H$}, i.e.\ the accumulation points of $H \cdot o$ in $X=\bdry Z$.
For $v \in VT$ we denote the limit set of the stabilizer $G_v$ by $\Lambda_v = \Lambda(G_v)$, and likewise for $e \in ET$ we let $\Lambda_e = \Lambda(G_e)$.

In each case considered here, the edge groups are uniformly quasiconvex as they are either finite or two-ended.
Therefore the vertex groups are uniformly quasiconvex also (see e.g.\ \cite[Proposition 1.2]{Bow98-jsj}), and so hyperbolic, and consequently for each $v \in VT$ the quasi-isometric embedding $g_v G_{i_v} \to Z$ found by restricting the orbit map induces a quasisymmetry $\bdry g_v G_{i_v} = g_v \bdry G_{i_v} \to \Lambda_v \subset \bdry Z$.  

\begin{lemma}[{cf.\ \cite[Proposition 1.3]{Bow98-jsj}, \cite[Lemma 10]{KapKle00-boundaries}}]
	\label{lem:bdry-limsets}
	If $G$ is a hyperbolic group with a graph of groups decomposition $\cG$ over quasiconvex edge groups with Bass--Serre tree $T$, with $G$ acting geometrically on the model space $Z$, and $X = \bdry Z$ with a visual metric,
	then every $x \in X$ corresponds to exactly one of the following:
	\begin{itemize} 
		\item a point of $\bdry T$, with a unique $x$ for each $t\in\bdry T$, or
		\item a point of $\Lambda_e$ for some $e \in ET$, or
		\item a point of $\Lambda_v$ for some unique $v \in VT$ (but not in any $\Lambda_e$).
	\end{itemize}
\end{lemma}
\begin{proof}
	Consider a geodesic ray $\gamma$ from $o$ in $Z$ representing $x \in X$.
	
	For an edge $e \in ET$ let $Z_{ e\to}$ be the component of $Z\setminus (Z_e\times\{0\})$ not containing $o$. 
	Let us say that $\gamma$ \emph{essentially crosses} the edge space corresponding to $e \in ET$, or just \emph{$\gamma$ essentially crosses $e$}, if for every $C>0$, $\gamma$ has unbounded intersection with $Z_{ e\to} \setminus N_C(Z_e\times\{0\})$.

	If $\gamma$ essentially crosses $e\in ET$, then it essentially crosses every edge between $v_0$ and $e$ in $T$.
	Moreover if a simple path from $v_0$ to some vertex $v$ in $T$ can be extended by either $e'$ or $e''$ in $ET$, then by quasiconvexity $\gamma$ cannot essentially cross both $e'$ and $e''$.
	Therefore the collection of edges in $T$ which $\gamma$ essentially crosses gives a simple path from $v_0$, either (i) infinite or (ii) finite.  Let us call this path $\hat\gamma$: by definition it depends only on the point $x \in X$ and not the choice of $\gamma$.

	In case (i), the path $\hat\gamma$ determines a unique point in $\bdry T$.
	We claim that there is a bijection between the set of $x\in X$ represented by $\gamma$ with $\hat\gamma$ unbounded, and points in $\bdry T$.
	First, given any point $t \in\bdry T$, by an Arzel\`a--Ascoli argument one can choose a geodesic ray $\gamma$ in $Z$ so that $\hat\gamma$ limits to $t$. 

	Second, if $\gamma, \alpha$ are geodesic rays and $\hat\gamma=\hat\alpha$ is unbounded, then $\gamma$ and $\alpha$ must represent the same point in $X$: suppose not, then $(\gamma|\alpha)_o < \infty$.  Choose a large constant $R$ and an edge $e\in ET$ which $\gamma$ and $\alpha$ essentially cross so that the edge space $Z_e\times\{0\}$ is outside $B(o,(\gamma|\alpha)_o+R)$.  Let $p, q \in Z_e\times\{0\}$ be points where $\gamma$ and $\alpha$ respectively meet the edge space.
	By hyperbolicity, the geodesic from $p$ to $q$ must go within distance $(\gamma|\alpha)_o+C$ of $o$, but by quasiconvexity it must remain within a distance $C$ of $Z_e\times\{0\}$, a contradiction for $R>2C$.  So case (i) is understood.

	Now suppose we are in case (ii), where $\hat\gamma$ is a finite path with final vertex $v$, and final edge $e$.
	If $\gamma$ leaves $Z_v$ through some $Z_{e'}\times\{0\}$ and does not return, as it does not essentially cross $e'$ it must limit to a point of $\Lambda_{e'}$.
	So if $\gamma$ does not limit to a point of any edge space, by quasiconvexity its tail must live in $N_C(Z_v)$ for some constant $C$, and so $x \in \Lambda_v$.
	If $x \in \Lambda_{v'}$ also for some $v' \neq v$, then the tail of $\gamma$ must live in $N_C(Z_{v'})$ also, hence in $N_C(Z_{e'})$ for any edge $e'$ between $v$ and $v'$; as this contradicts our assumption on $\gamma$ we have that $v$ is unique as required.
\end{proof}

In the rest of this section we will use the \emph{approximate self-similarity} of the boundary of a hyperbolic group:
	there exists $L_0 \geq 1$ so that for all $x \in X=\bdry Z$, and all $0<r\leq \diam X$, there exists $g\in G$ so that the action of $g$ on $X$ induces an $L_0$-bi-Lipschitz map from the rescaled ball $(B(x,r),\frac{1}{r}d)$ to an open set $U \subset X$ with $B(gx,\frac{1}{L_0})\subset U$.

\begin{lemma}[{cf. \cite[Proposition 6.2]{BuyLeb08-dimensions}, \cite[Proposition 3.3]{bou-kle13-CLP}, \cite[Corollary 4.9]{MacSis-qhyp-planes}}] 
	\label{lem:hypbdry-selfsim}
	Suppose $Z$ is a hyperbolic, proper, geodesic metric space with a geometric group action $G \curvearrowright Z$, base point $o$, and a visual metric $d$ on $X = \bdry Z$ with visual parameter $\epsilon$.
	Then there exists $L_0\geq 1$ so that $X$ is approximately self-similar.
\end{lemma}
\begin{proof}
	By the cocompactness of $G\curvearrowright Z$ there exists $D>0$ so that $G \cdot B_Z(o,D)=Z$.  Let $L_0$ be given by \cite[Corollary 4.9]{MacSis-qhyp-planes} applied to $D$, the hyperbolicity constant $\delta_Z$ for $Z$, and parameters $C_0, \epsilon$ for the visual metric $d$.

	Suppose we are given $x \in X$ and $r \in (0,\diam X]$. 
	If $-\epsilon^{-1} \log(2rC_0)-\delta_Z-1 \geq 0$ then \cite[Corollary 4.9(1)]{MacSis-qhyp-planes} with ``$y$''$=$``$x$'', ``$r'$''$=r$, and an appropriate $g\in G$ gives an $L_0'$-bi-Lipschitz map from $(B(x,r),\frac{1}{r})$ to an open set $U \subset X$ with $B(f(x),\frac{1}{L_0'}) \subset U$.
Otherwise, $-\epsilon^{-1} \log(2rC_0)-\delta_Z-1 < 0$ and \cite[Corollary 4.9(2)]{MacSis-qhyp-planes} shows that $1\in G$ gives approximate self-similarity.
\end{proof}

\subsection{Connected components in boundaries}

Maximal splittings over finite edge groups enable us to control the geometry of connected components in any space arising as the boundary of any space admitting a geometric action by a hyperbolic group.

  Recall that a metric space $X$ is \emph{$C$-linearly connected} for some $C\geq 1$ if for any two points $x,y\in X$ there is a compact connected set $I \subset X$ with $\diam I \leq C d(x,y)$.  
The following definition is used in the proof of Theorem~\ref{thm:kkc-confdim-hypgroup}.
  
\begin{definition}[{see \cite[Theorem 3.11]{carr13-conf-gauge}}]\label{def:unif-lin-conn-sep}
  The components of a metric space are \emph{uniformly linearly connected} if they are each $K_\ell$-linearly connected for some fixed $K_\ell\geq1$.

  The components are \emph{uniformly separated} if for some fixed $K_s \geq 1$, for
all $0 < r \leq \diam X$, there exists a covering $\cW_r$ of $X$, by open and closed sets, such
that for all $W \in \cW_r$, we have $d(W, X \setminus W ) \geq r/K_s$ and there exists a connected
component $Y$ of $X$ with $Y \subset W$ and $W$ is contained in the $r$-neighbourhood of $Y$.
\end{definition}
  Recall that by Stallings--Dunwoody \cite{Sta68-torsion,Dun85-accessibility}, there is a maximal graph of groups decomposition of $G$ where all edge groups are finite, and the vertex groups $\{G_i\}$ are all finite or one-ended.
\begin{lemma}
	\label{lem:bdry-split-over-finite}
	Suppose $G$ is a hyperbolic group acting geometrically on a geodesic space $Z'$, 
	and let $X'=\bdry Z'$ with a fixed visual metric $d'$.
	Let $T$ be a Bass-Serre tree corresponding to a Stallings--Dunwoody decomposition for $G$ with vertex stabilizers denoted $\{G_v\}$.
	Then
	\begin{enumerate}
		\item the connected components of $X'$ correspond to $\Lambda_v$ for $G_v$ one-ended, and to points in $\bdry T$;
		\item $X'$ has uniform linear connectivity of components;
		\item $X'$ has uniform separation of components.
	\end{enumerate}
\end{lemma}

  The uniform separation of components condition is tricky to work if we let $X'$ be arbitrarily quasisymmetrically equivalent to $X$, but we only need the case of visual metrics $d'$ as in the lemma.
\begin{proof}
Consider a Stallings--Dunwoody graph of groups decomposition of $G$ with corresponding tree $T$.
	Let $Z$ be the model space for this graph of groups decomposition for $G$ constructed as in the previous section with base point $o$, let $X=\bdry Z$ be its boundary with a visual metric $d$ as before.

	Since $G$ acts geometrically on $Z'$, there is a quasi-isometry $\phi:Z\to Z'$ which sends the orbit $G \cdot o$ to the orbit $G \cdot \phi(o)$ equivariantly; let $o':=\phi(o)$.  Let $\psi:Z'\to Z$ be a quasi-inverse of $\phi$, sending $G\cdot o'$ to $G\cdot o$ equivariantly.
	As before, write $\Lambda_v=\Lambda(G_v), \Lambda_e=\Lambda(G_e)$ for the given limit sets in $X$, and let $\Lambda'_v, \Lambda'_e$ be the corresponding limit sets in $X'$.
	By equivariance, $\bdry \phi (\Lambda_v) = \Lambda'_v$ and $\bdry \phi (\Lambda_e) = \Lambda'_e$.

	We now begin the proof of (1).
Since the edge groups are finite,
in $Z$ the edge spaces $Z_e \times \{0\}$, $e\in ET$, have uniformly bounded diameter.
	
	Consider a geodesic ray $\gamma'$ from $o'$ in $Z'$.
	Let $\gamma$ be a geodesic ray from $o$ in $Z$ at bounded Hausdorff distance from the quasi-geodesic $\psi(\gamma')$.
By Lemma~\ref{lem:bdry-limsets}, and since $\Lambda_e = \emptyset$ for any $e\in ET$, every $x\in X$ either corresponds to a point of $\bdry T$, or to a point in $\Lambda_v$ for a unique $v\in VT$.

	We define the simple path $\hat\gamma$ in $T$ as in the proof of Lemma~\ref{lem:bdry-limsets}; by construction it is independent of the choice of $\gamma$, so we write $\hat{\gamma}':=\hat{\gamma}$.
	Suppose we have two geodesic rays $\gamma',\alpha'$ in $Z'$ with corresponding geodesic rays $\gamma,\alpha$ in $Z$ as above.
	If $\gamma$ essentially crosses some edge $e \in ET$ and $\alpha$ does not, then the Gromov product $(\gamma|\alpha)_o$ is $\leq d_{Z}(o, Z_e\times\{0\})+C$, and so the condition ``essentially crossing $e$'' splits $X$ into two sets at positive distance, thus the limit points of $\gamma$ and $\alpha$ are in different connected components of $X$.
	As $\bdry \phi$ is a homeomorphism, the limit points of $\gamma'$ and $\alpha'$ are in different connected components of $X'$.
	Thus if the limit points of $\gamma'$ and $\alpha'$ are in the same connected component of $X'$, then $\hat{\gamma}'=\hat{\alpha}'$.
	
	If $\hat\gamma'$ is unbounded, and $\alpha'$ is a geodesic ray in $Z'$, either $\hat\alpha'\neq\hat\gamma'$ and so $\alpha'$ limits to a different connected component of $X'$, or $\hat\alpha'=\hat\gamma'$, so by Lemma~\ref{lem:bdry-limsets} $\alpha$ and $\gamma$ represent the same point in $X = \bdry Z$, and so $\alpha'$ and $\gamma'$ represent the same point in $X' = \bdry \phi(X)$.  This point corresponds to the point in $\bdry T$ represented by $\hat\gamma'$.

	On the other hand, if $\hat\gamma'$ is a finite path, let $v$ be the final vertex of the path.
	Therefore $\gamma$ must meet $Z_v$ in an unbounded set, and limits to a point of $\Lambda_v$, and so $\gamma'$ limits to a point of $\Lambda'_v$, which is the image of $\bdry g_v G_{i_v}$ under the boundary of the orbit map $G\to Z'$.
	As $G_{i_v}$ is infinite it is one-ended, so $\bdry G_{i_v}$ is connected and thus so is $\Lambda'_v$, and every geodesic ray $\alpha'$ with $\hat\alpha'=\hat\gamma'$ is in the same connected set $\Lambda'_v$.
	So (1) is proved.

	We now prove (2).
	By Bonk--Kleiner~\cite{bonk-kleiner2005-qh-planes} the boundary of a one-ended hyperbolic group is linearly connected.
	If $v \in VT$ corresponds to a one-ended vertex group, then
	as $g_v G_{i_v} \to Z_v \to \phi(Z_v)$ is a quasi-isometry embedding into $Z'$,
	the boundary map $g_v \bdry G_{i_v} = \bdry g_v G_{i_v} \to \Lambda_v \to \Lambda'_v$ is a quasisymmetric embedding,
	so $\Lambda'_v$ is also linearly connected, though not a priori with constants independent of $v$.

	However, we can use the approximate self-similarity of $X' = \bdry Z'$.
	Let $\epsilon'$ denote the visual parameter of the metric $d'$.
	For $v \in VT$ we have $\diam \Lambda'_v \preceq e^{-\epsilon' d_{Z'}(o', \phi Z_v)}$ because all geodesic rays from $o'$ to points in $\Lambda'_v$ must pass within bounded distance of the same edge space adjacent to $\phi Z_v$, and in particular their Gromov products with each other are all $\geq d_{Z'}(o',\phi Z_v) -C$.
	By Lemma~\ref{lem:hypbdry-selfsim}, for all $v$ we can find a $g\in G$ so that, up to scaling, $\Lambda'_v$ is uniformly bi-Lipschitz to $g \cdot \Lambda'_v = \Lambda'_{gv}$, where $\diam \Lambda'_{gv} \geq 1/C>0$. That is, $\Lambda'_{gv}$ is one of finitely many possible candidates.
	Thus the linear connectivity constant of $\Lambda'_v$ may be taken independent of $v$.
	We have proven (2).
	
	It remains to show (3).

	Given $R>0$, let $E_R$ be the set of edges of $T$ so that the corresponding edge spaces $\phi(Z_e)$ are within distance $R$ of the base point $o'\in Z'$.
	Partition the boundary $X'$ according to the last edge in $E_R$ which the corresponding geodesic rays $\gamma'$ essentially cross, i.e., a geodesic representative $\gamma$ of the quasi-geodesic ray $\psi(\gamma')$ essentially crosses the edge.  Denote the partition by $\cW_R$.

  Notice that there is a set in this partition that corresponds to the rays whose class does not essentially cross any edge in $E_R$.
  This set is a neighbourhood of $\Lambda'_{v_0}$. 
  The sets in this partition are closed since a limit of rays in the same set is also in the same set.  Since the partition is finite the sets are open as well.
 
  For $W \in \cW_R$ if $x \in X' \setminus W$ and $w \in W$ then by the definition of $\cW_R$ we must have $(x|w)_{o'} \leq R+C$ and so $d'(x,w) \geq e^{-\epsilon' R}/C'$ for some constant $C'$.

	Consider $W \in \cW_R$ corresponding to geodesic rays which essentially cross an edge $e \in E_R$ last, and let $v$ be the vertex of $e$ furthest from $v_0$. If $W$ corresponds to rays not essentially crossing any edge of $E_R$, let $v=v_0$.

	If $G_v$ is finite, then the (bounded) edge space for $e$ is at distance $\geq R-C$ from $o'$, else we would have to essentially cross another edge of $E_R$.
	So we can take as $Y \subset X'$ the connected component of some such geodesic ray $\gamma'$ in $W$.  
	Indeed, by the proof of (1) above, every geodesic ray corresponding to a point of $Y$ essentially crosses the same edge $e$ as $\gamma'$, so $Y \subset W$.
	Also, if $w \in W$ then for any $y \in Y$ we have $(y|w)_{o'}\geq R-C$ and so $d'(y,w) \leq e^{-\epsilon' R} C'$.
	
	On the other hand, if $G_v$ is infinite, let $Y = \Lambda'_v \subset W$.
	If $w \in W$ then the last point $z$ of a geodesic ray to $w$ in $B(o',R)$ is (within bounded distance of) some point in $\phi Z_v$, as the ray essentially crosses the same edges of $E_R$ as any geodesic from $o'$ to $\phi Z_v$.
	Since $G_v$ is infinite there is a geodesic line close to $\phi Z_v$ passing within bounded distance of $z$; let $y,y' \in Y = \Lambda'_v$ be the limit points of the line.
	By hyperbolicity, one of the geodesic rays from $o'$ to $y$ or to $y'$ passes within a uniformly bounded distance of $z$, so $\max \{ (w|y)_{o'}, (w|y')_{o'} \} \geq R-C$, thus $d'(Y,w)\leq e^{-\epsilon' R} C'$.
	
    In conclusion the uniform separation of components, statement (3) of the lemma, is satisfied for $K_s:= (C')^2$ and by taking $\cW_r := \cW_R$ for $R = \frac{-1}{\epsilon'} \log(r/C')$.
\end{proof}

\subsection{Two-ended edge groups}

We now consider in more detail the case when all edge groups are two-ended (and hence all vertex groups are infinite).
In general, stabilizers of different edge groups can have the same limit sets, so to get stronger results about the geometry of such groups we switch from the given graph of groups to a new, bipartite, graph of groups.
We follow Guirardel--Levitt \cite{GuiLev11-cylinders}.

\begin{proposition}\label{prop:tree-of-cylinders}
	Given a hyperbolic group $G$ with a graph of groups decomposition over $2$-ended edge groups,
	we can find a graph of groups corresponding to an action $G \curvearrowright T$ where the tree is bipartite with $VT = V_0T\sqcup V_1T$, and
	\begin{enumerate}
		\item[(1)] all $V_0T$ vertex groups are non-elementary and are conjugate to some original vertex group;
		\item[(2)] all $V_1T$ groups and all edge groups are $2$-ended;
		\item[(3)] different $V_1T$ vertex groups are not commensurable, and hence have disjoint limit sets in $G$;
		\item[(4)] every original vertex group that was non-elementary is also a new $V_0T$ vertex group.
	\end{enumerate}
\end{proposition}
\begin{proof}
Let $G\curvearrowright S$ be the original tree action. The new tree $T$ is what Guirardel and Levitt call the tree of cylinders of $S$. Their construction is as follows (for details see \cite[Section 4]{GuiLev11-cylinders}, for an example see Figure \ref{fig:cylinders}).

\begin{figure}
\centering
\begin{tikzpicture}[text opacity=1]

\draw[ultra thick] (1,5) -- node[below] {\footnotesize $C$} (3,5);
\fill[red!60] (1,5) circle (0.1) node[below] {\footnotesize $A$};
\fill[blue!60] (3,5) circle (0.1) node[below] {\footnotesize $B$};
\node[below] at (2,4.5) {\footnotesize $a \mapsfrom c\mapsto b^2$};

\draw[ultra thick] (1+6,5) -- node[below] {\footnotesize $C_1$} (3+6,5);
\draw[ultra thick] (3+6,5) -- node[below] {\footnotesize $C_2$} (5+6,5);
\fill[red!60] (1+6,5) circle (0.1) node[below] {\footnotesize $A$};
\fill[blue!60] (3+8,5) circle (0.1) node[below] {\footnotesize $B$};
\draw[fill=white] (3+6,5) circle (0.1) node[below] {\footnotesize $\langle b \rangle$};
\node[below] at (2+6,4.5) {\footnotesize $a \mapsfrom c_1\mapsto b^2$};
\node[below] at (4+6,4.5) {\footnotesize $b \mapsfrom c_2\mapsto b$};

\draw[ultra thick,opacity=0.3] (0,0) -- node {\footnotesize $\langle b^2\rangle$} (2,0);
\draw[ultra thick,opacity=0.3,rotate around={30:(2,0)}] (0,0) -- node {\footnotesize $\langle b^2 \rangle$} (2,0);

\draw[ultra thick,opacity=0.3,rotate around={60:(2,0)}] (0,0) --  (2,0);
\draw[ultra thick,opacity=0.3,rotate around={90:(2,0)}] (0,0) --  (2,0);
\draw[ultra thick,opacity=0.3,rotate around={120:(2,0)}] (0,0) --  (2,0);
\draw[ultra thick,opacity=0.3,rotate around={135:(2,0)}] (0,0) --  (2,0);
\draw[ultra thick,opacity=0.3,rotate around={150:(2,0)}] (0,0) --  (2,0);
\draw[ultra thick,opacity=0.3,rotate around={157.5:(2,0)}] (0,0) --  (2,0);
\draw[ultra thick,opacity=0.3,rotate around={-30:(2,0)}] (0,0) --  (2,0);
\draw[ultra thick,opacity=0.3,rotate around={-60:(2,0)}] (0,0) --  (2,0);
\draw[ultra thick,opacity=0.3,rotate around={-90:(2,0)}] (0,0) --  (2,0);
\draw[ultra thick,opacity=0.3,rotate around={-105:(2,0)}] (0,0) --  (2,0);
\draw[ultra thick,opacity=0.3,rotate around={-120:(2,0)}] (0,0) --  (2,0);
\draw[ultra thick,opacity=0.3,rotate around={-127.5:(2,0)}] (0,0) --  (2,0);

\foreach \x in {120,90,60,30,0,-30,-60,-90}{
\draw[ultra thick,opacity=0.2,rotate around={\x:(2,0)}] (0,0) --  ({cos(150)},{sin(150)});
\draw[ultra thick,opacity=0.2,rotate around={\x:(2,0)}] (0,0) --  ({cos(160)},{sin(160)});
\draw[ultra thick,opacity=0.2,rotate around={\x:(2,0)}] (0,0) --  ({cos(170)},{sin(170)});
\draw[ultra thick,opacity=0.2,rotate around={\x:(2,0)}] (0,0) --  ({cos(-150)},{sin(-150)});
\draw[ultra thick,opacity=0.2,rotate around={\x:(2,0)}] (0,0) --  ({cos(-160)},{sin(-160)});
\draw[ultra thick,opacity=0.2,rotate around={\x:(2,0)}] (0,0) --  ({cos(-170)},{sin(-170)});
\fill[blue!60,rotate around={\x:(2,0)}] ({cos(150)},{sin(150)}) circle (0.05);
\fill[blue!60,rotate around={\x:(2,0)}] ({cos(160)},{sin(160)}) circle (0.05);
\fill[blue!60,rotate around={\x:(2,0)}] ({cos(170)},{sin(170)}) circle (0.05);
\fill[blue!60,rotate around={\x:(2,0)}] ({cos(-150)},{sin(-150)}) circle (0.05);
\fill[blue!60,rotate around={\x:(2,0)}] ({cos(-160)},{sin(-160)}) circle (0.05);
\fill[blue!60,rotate around={\x:(2,0)}] ({cos(-170)},{sin(-170)}) circle (0.05);
}

\fill[opacity=0.7,rotate around={170:(2,0)}] (1,0) circle (0.03);
\fill[opacity=0.7,rotate around={180:(2,0)}] (1,0) circle (0.03);
\fill[opacity=0.7,rotate around={190:(2,0)}] (1,0) circle (0.03);
\fill[opacity=0.7,rotate around={200:(2,0)}] (1,0) circle (0.03);
\fill[opacity=0.7,rotate around={210:(2,0)}] (1,0) circle (0.03);
\fill[opacity=0.7,rotate around={220:(2,0)}] (1,0) circle (0.03);

\fill[red!60] (0,0) circle (0.1) node[below] {\footnotesize $A$};
\fill[blue!60] (2,0) circle (0.1) node[below] {\footnotesize $B$};
\fill[red!60,rotate around={30:(2,0)}] (0,0) circle (0.1) node[below] {\footnotesize $bAb^{-1}$};
\fill[red!60,rotate around={60:(2,0)}] (0,0) circle (0.1);
\fill[red!60,rotate around={90:(2,0)}] (0,0) circle (0.1);
\fill[red!60,rotate around={120:(2,0)}] (0,0) circle (0.1);
\fill[red!60,rotate around={135:(2,0)}] (0,0) circle (0.1);
\fill[red!60,rotate around={150:(2,0)}] (0,0) circle (0.1);
\fill[red!60,rotate around={157.5:(2,0)}] (0,0) circle (0.1);
\fill[red!60,rotate around={-30:(2,0)}] (0,0) circle (0.1);
\fill[red!60,rotate around={-60:(2,0)}] (0,0) circle (0.1);
\fill[red!60,rotate around={-90:(2,0)}] (0,0) circle (0.1);
\fill[red!60,rotate around={-105:(2,0)}] (0,0) circle (0.1);
\fill[red!60,rotate around={-120:(2,0)}] (0,0) circle (0.1);
\fill[red!60,rotate around={-127.5:(2,0)}] (0,0) circle (0.1);

\foreach \x in {120,90,60,30,0,-30,-60,-90}{
\draw[ultra thick,opacity=0.2,rotate around={\x:(9,0)}] (0+7,0) --  ({cos(150)+7},{sin(150)});
\draw[ultra thick,opacity=0.2,rotate around={\x:(9,0)}] (0+7,0) --  ({cos(160)+7},{sin(160)});
\draw[ultra thick,opacity=0.2,rotate around={\x:(9,0)}] (0+7,0) --  ({cos(170)+7},{sin(170)});
\draw[ultra thick,opacity=0.2,rotate around={\x:(9,0)}] (0+7,0) --  ({cos(-150)+7},{sin(-150)});
\draw[ultra thick,opacity=0.2,rotate around={\x:(9,0)}] (0+7,0) --  ({cos(-160)+7},{sin(-160)});
\draw[ultra thick,opacity=0.2,rotate around={\x:(9,0)}] (0+7,0) --  ({cos(-170)+7},{sin(-170)});
\fill[blue!60,rotate around={\x:(9,0)}] ({cos(150)+7},{sin(150)}) circle (0.05);
\fill[blue!60,rotate around={\x:(9,0)}] ({cos(160)+7},{sin(160)}) circle (0.05);
\fill[blue!60,rotate around={\x:(9,0)}] ({cos(170)+7},{sin(170)}) circle (0.05);
\fill[blue!60,rotate around={\x:(9,0)}] ({cos(-150)+7},{sin(-150)}) circle (0.05);
\fill[blue!60,rotate around={\x:(9,0)}] ({cos(-160)+7},{sin(-160)}) circle (0.05);
\fill[blue!60,rotate around={\x:(9,0)}] ({cos(-170)+7},{sin(-170)}) circle (0.05);

\draw[fill=white,rotate around={\x:(9,0)}] ({cos(150)/2+7},{sin(150)/2}) circle (0.05);
\draw[fill=white,rotate around={\x:(9,0)}] ({cos(160)/2+7},{sin(160)/2}) circle (0.05);
\draw[fill=white,rotate around={\x:(9,0)}] ({cos(170)/2+7},{sin(170)/2}) circle (0.05);
\draw[fill=white,rotate around={\x:(9,0)}] ({cos(-150)/2+7},{sin(-150)/2}) circle (0.05);
\draw[fill=white,rotate around={\x:(9,0)}] ({cos(-160)/2+7},{sin(-160)/2}) circle (0.05);
\draw[fill=white,rotate around={\x:(9,0)}] ({cos(-170)/2+7},{sin(-170)/2}) circle (0.05);
}

\draw[ultra thick,opacity=0.3] (7,0) -- node {\tiny $\langle b^2\rangle$} ({(2*cos(210)+9)/4+7/4+9/2},{(2*sin(210))/4+0/4+0/2}) -- node {\tiny $\langle b\rangle$} (9,0);
\draw[ultra thick,opacity=0.3] ({2*cos(210)+9},{2*sin(210)}) -- node {\tiny $\langle b^2 \rangle$} ({(2*cos(210)+9)/4+7/4+9/2},{(2*sin(210))/4+0/4+0/2});
\fill[red!60] (7,0) circle (0.1) node[below] {\footnotesize $A$};
\fill[red!60] ({2*cos(210)+9},{2*sin(210)}) circle (0.1) node[right] {\footnotesize $bAb^{-1}$};
\draw[fill=white] ({(2*cos(210)+9)/4+7/4+9/2},{(2*sin(210))/4+0/4+0/2}) circle (0.2) node {\tiny $\langle b\rangle$};
\fill[blue!60] (9,0) circle (0.1) node[below] {\footnotesize $B$};

\foreach \x in {120,60,-60,-120}{
\draw[ultra thick,opacity=0.3,rotate around={\x:(9,0)}] (7,0) -- ({(2*cos(210)+9)/4+7/4+9/2},{(2*sin(210))/4+0/4+0/2}) -- (9,0);
\draw[ultra thick,opacity=0.3,rotate around={\x:(9,0)}] ({2*cos(210)+9},{2*sin(210)}) -- ({(2*cos(210)+9)/4+7/4+9/2},{(2*sin(210))/4+0/4+0/2});
\fill[red!60,rotate around={\x:(9,0)}] (7,0) circle (0.1);
\fill[red!60,rotate around={\x:(9,0)}] ({2*cos(210)+9},{2*sin(210)}) circle (0.1);
\draw[fill=white,rotate around={\x:(9,0)}] ({(2*cos(210)+9)/4+7/4+9/2},{(2*sin(210))/4+0/4+0/2}) circle (0.2);
\fill[blue!60,rotate around={\x:(9,0)}] (9,0) circle (0.1);

}

\fill[opacity=0.7,rotate around={170:(9,0)}] (8,0) circle (0.03);
\fill[opacity=0.7,rotate around={180:(9,0)}] (8,0) circle (0.03);
\fill[opacity=0.7,rotate around={190:(9,0)}] (8,0) circle (0.03);
\fill[opacity=0.7,rotate around={200:(9,0)}] (8,0) circle (0.03);
\fill[opacity=0.7,rotate around={210:(9,0)}] (8,0) circle (0.03);
\fill[opacity=0.7,rotate around={220:(9,0)}] (8,0) circle (0.03);

\end{tikzpicture}
\caption{Let $A=\langle a_1,a_2\rangle$, $B=\langle b_1,b_2\rangle$ be two copies of the free group and $C=\langle c\rangle \simeq \Z$. Let $a=[a_1,a_2]$ and $b=[b_1,b_2]$, and consider the amalgamated product $G=A\ast_C B$ where the injection maps are given by $c\mapsto a$ and $c\mapsto b^2$. The group $G$ is isomorphic to the fundamental group of the complex obtained by gluing two punctured tori to a Mobius band, one of them glued along its boundary to the boundary of the band, and the other one glued along its boundary to the mid-circle of the band. The splitting $A\ast_C B$ corresponds to the graph of groups decomposition shown on the left.
Notice that the edges of $S$ issuing from a $B$-vertex are naturally paired: in this case the cylinders are the pairs of edges having the same stabilizer. The tree of cylinders $T$ corresponds to replacing these pairs by a tripod. The associated graph of groups is shown on the right.
}
\label{fig:cylinders}
\end{figure}
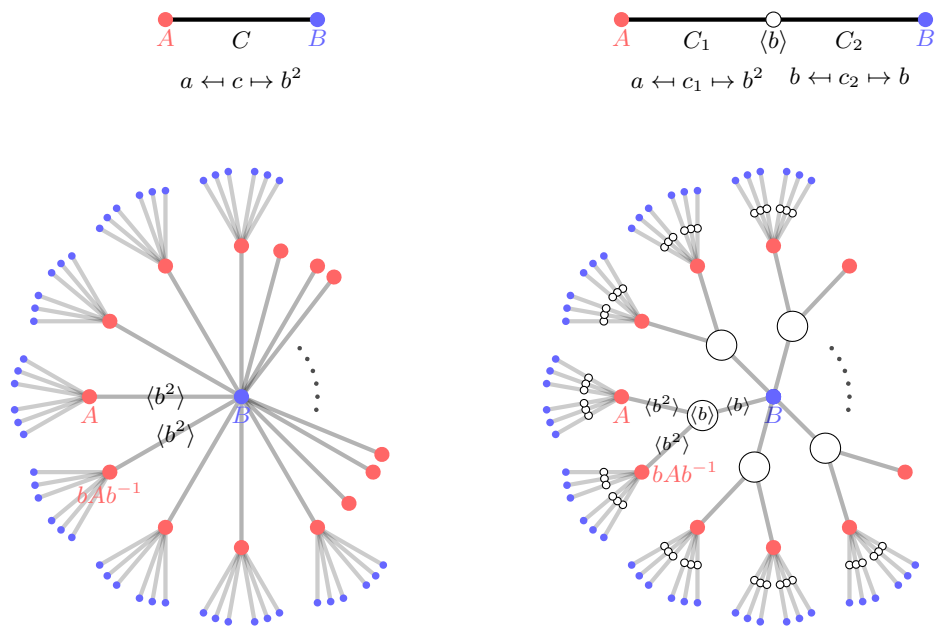

Define an equivalence relation $\sim_S$ on the set of non-oriented edges of $S$ by $e\sim_S e'$ if $G_{e}$ and $G_{e'}$ are commensurable (i.e.\ if $G_{e}\cap G_{e'}$ has finite index in both $G_e$ and $G_{e'}$). A \emph{cylinder} of $S$ is an equivalence class $[e]_S$.

Notice that since $G$ is hyperbolic, given two edges $e$ and $e'$ of $S$ either $\Lambda_e\cap\Lambda_{e'}=\emptyset$ or $\Lambda_e=\Lambda_{e'}$. Moreover, $e\sim_S e'$ if and only if $\Lambda_e=\Lambda_{e'}$. 

By \cite[Lemma 4.2]{GuiLev11-cylinders} every cylinder of $S$ is connected, and hence a subtree. Since there are only finitely many conjugacy classes of edge groups, and there are only finitely many conjugate edge groups that can contain a given loxodromic, every cylinder is finite.

The tree of cylinders $T$ is the bipartite tree with vertex set $VT=V_0T\sqcup V_1T$ defined as follows:
\begin{enumerate}
\item $V_0T$ is the set of vertices $v$ of $S$ belonging to at least two distinct cylinders;
\item $V_1T$ is the set of cylinders $[e]_S$ of $S$;
\item and there is an edge between $v$ and $[e]_S$ if $v$ as a vertex of $S$ belongs to the union of edges of $[e]_S$. 
\end{enumerate}
That is, the tree $T$ is obtained from $S$ by replacing each cylinder by the cone on its boundary. See \cite[Definition 4.8]{Gui04-limit} for the proof that $T$ is indeed a tree. Moreover, the group $G$ acts on $T$ and the action $G\curvearrowright T$ is also minimal,  \cite[Lemma 4.9]{Gui04-limit}.

Notice that a non-elementary stabilizer $G_v$ of a vertex $v$ of $S$ has infinite degree in $S$. Therefore, if a vertex of $S$ belongs to only one cylinder, it has finite degree and its stabilizer must be two-ended. That is, vertices of $S$ with non-elementary stabilizers are also vertices in $V_0T$. This shows (4). Moreover, the stabilizer of a vertex in $V_0T$ is the same as the stabilizer of the corresponding vertex in $S$, so no new non-elementary vertices are created by this construction, and this gives (1).

The stabilizer of a vertex in $V_1T$ is the global stabilizer of a cylinder $[e]_S$ in $S$, which coincides with the maximal two-ended subgroup containing $G_{e'}$ for any edge $e'\in [e]_S$. This proves the first claim of (2), and shows that an edge stabilizer of $T$ is elementary. But if $(v,[e]_S)$ is an edge of $T$, then its stabilizer contains the stabilizer of the edge of $[e]_S$ incident to $v$, which is two-ended. Therefore edge stabilizers of $S$ are two-ended (see also \cite[Proposition 6.1]{GuiLev11-cylinders}). This completes the proof of (2). Property (3) follows directly by the definition of cylinders.
\end{proof}

\subsection{Metric estimates for the limit sets of the bipartite tree action}

To compute conformal dimension we need metric estimates on boundaries.  In this section we estimate the distances and diameters of the limit sets of the vertex groups appearing in the bipartite tree action of Proposition~\ref{prop:tree-of-cylinders}.
We use $K_1, K_2, \ldots$ for the constants found in these estimates so that their use is clear later in the paper.

We think of the $V_1T$ vertex spaces/groups as generalised edge spaces/groups, and indeed we do not need to consider edges any more, since every edge space is at finite Hausdorff distance from the adjacent $V_1T$ vertex space. Nevertheless we keep the notation so that $v$ stands for a vertex in $V_0T$ and $e$ stands for a vertex in $V_1T$. So Lemma~\ref{lem:bdry-limsets} becomes:
\begin{lemma}
	\label{lem:bdry-limsets-2}
	If $G$ is a hyperbolic group with $G \curvearrowright T$ as in Proposition~\ref{prop:tree-of-cylinders}, with $G$ acting geometrically on the model space $Z$, and $X = \bdry Z$ with a visual metric,
	then every $x \in X$ corresponds to exactly one of the following:
	\begin{itemize} 
		\item a point of $\bdry T$, with a unique $x$ for each $t\in\bdry T$, or
		\item a point of $\Lambda_e$ for some unique $e \in V_1T$, or
		\item a point of $\Lambda_v$ for some unique $v \in V_0T$ (but not in any $\Lambda_e$).
	\end{itemize}
\end{lemma}

As before, by quasiconvexity $\Lambda_v$ is a quasisymmetric image of $g_v \bdry G_{i_v}$ for each $v\in V_0T$.  Likewise, for each $e\in V_1T$, $\Lambda_e$ is a quasisymmetric image of $g_e \bdry G_{k_e}$, that is, it is a pair of points in $X$.

Fix corresponding basepoints $v_0 \in V_0T$, $o\in Z$.
For each $v \in V_0T\setminus \{v_0\}$, let $e_v \in V_1T$ be the last $V_1T$ vertex on the geodesic from $v_0$ to $v$.
We have that $\Lambda_{e_v}$ cuts $X$ into at least two components \cite[Sec 1]{Bow98-jsj}, while the interior of the open edge $(e_v,v)$ cuts $T$ into exactly two components, one containing $v_0$ and the other not.  
Let $Z_{\gets e_v}$ be the component of $Z\setminus Z_{(e_v,v)}\times\{0\}$ containing $o$,
	and let $Z_{ e_v\to}$ be the other component.
	We define $U_{\gets v} := \bdry Z_{\gets e_v}$ and $U_{v\to} := \bdry Z_{ e_v\to}$.
	Since $Z_{\gets e_v}$ and $Z_{ e_v\to}$ are quasiconvex, these correspond to the closure of the limit sets of the corresponding components of $T\setminus (e_v,v)$.
Note that $U_{\gets v} \cap U_{v\to} = \Lambda_{e_v}$.

We let $U_{v_0\to} := X$ and leave $U_{\gets v_0}$ and $\Lambda_{e_{v_0}}$ undefined.  

We say that $w\in V_0T$ is a descendant of $v\in V_0T\setminus\{v_0\}$ if $v$ separates $w$ from $v_0$ in $T$. We also say that all vertices of $T$ are descendants of $v_0$. For $v \in T$, we denote by $T_0(v)$ the collection of $v$ and all its descendants in $V_0T$.

In all the following lemmas we assume as above that $Z$ is a tree of spaces for a graph of groups decomposition of the group $G$ like in Proposition~\ref{prop:tree-of-cylinders}.

The following lemma implies that for any $e \neq e' \in V_1T$, we have $\Delta(\Lambda_e, \Lambda_{e'}) \geq 1/K_1$, where 
\[
	\Delta(U,V) := \frac{d(U,V)}{ \diam U \wedge \diam V  }
\]
is the \emph{relative distance} of $U, V \subset X$.
\begin{lemma} \label{lem:reldist1}
There exists a constant $K_1$ so that for $e \neq e' \in V_1T$ we have
			\[
			d(\Lambda_e, \Lambda_{e'}) \succeq_{K_1} \diam \Lambda_e \wedge  \diam \Lambda_{e'} \ .
			\]
\end{lemma}
\begin{proof}
Pick loxodromic elements $g, g'$ so that $g^{\pm \infty} = \Lambda_e$ and $(g')^{\pm \infty} = \Lambda_{e'}$, and let $\ell, \ell'$ be their translation lengths; as there are finitely many conjugation classes of edge stabilizers, we may assume that $\ell,\ell'$ are uniformly bounded away from $0$ and $\infty$, and that there are uniform bounds on the quasi-geodesic constants for $n\mapsto g^n$ and $n\mapsto (g')^n$.
	\begin{figure}
	   \def\svgwidth{.8\textwidth}
	   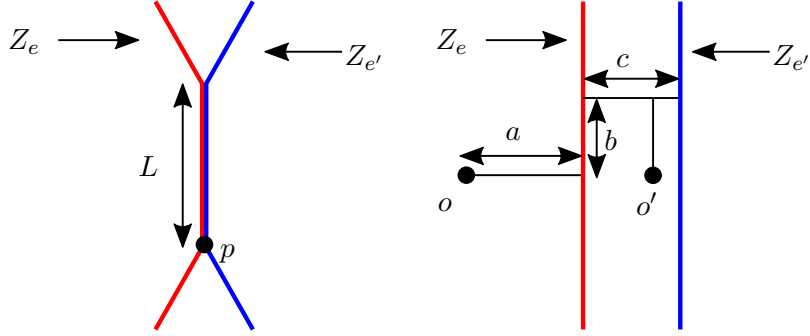
	 	\caption{Tree approximation for Lemma~\ref{lem:reldist1}}
		\label{fig:reldist1}
	\end{figure}

	Consider the tree approximation to geodesic axes for $\Lambda_e$ and $\Lambda_{e'}$ as in Figure~\ref{fig:reldist1}.  
	Suppose, as in the left of the figure, the axes remain $2\delta_Z$-close for a large distance $L$, where $\delta_Z$ is the hyperbolicity constant for $Z$.  Up to swapping $g, g^{-1}$ this means that there is a point $p$ so that for any $i \leq L/\ell'$, the point $g^{-\lfloor i\ell'/\ell\rfloor} (g')^i p$ is uniformly close to $p$.  Thus, by the uniform properness of $G \curvearrowright Z$, there exists $L'$ independent of $e,e'$ so that if $L>L'$ then there exist $i_1 \neq i_2$ so that $g^{-\lfloor i_1 \ell'/\ell\rfloor} (g')^{i_1} = g^{-\lfloor i_2 \ell'/\ell\rfloor} (g')^{i_2}$, hence $\langle g \rangle$ and $\langle g' \rangle$ are commensurable, a contradiction to the disjointness of $\Lambda_e, \Lambda_{e'}$.

	Thus, up to a uniformly bounded error, the tree approximation of $\Lambda_e, \Lambda_{e'}$ must look like the right of Figure~\ref{fig:reldist1}, for some $c\geq 0$.
	Up to swapping $e,e'$, the position of $o$ in the tree approximation must look like that of $o$ or $o'$ in the figure; suppose the former (the latter case is similar and easier), and label the other relevant distances $a,b$, up to bounded error.
	One can compute that 
	$d(\Lambda_e, \Lambda_{e'}) \asymp e^{-\epsilon (a+b)}$,
	$\diam \Lambda_e \asymp e^{-\epsilon a}$, and
	$\diam \Lambda_{e'} \asymp e^{-\epsilon (a+b+c)}$,
	so as $a+b+c\geq a+b$ we are done.
\end{proof}

The tree-of-spaces structure of $Z$ implies the following bounds when we consider how edge limit sets cut $X$.
\begin{lemma} \label{lem:reldist2}
	There exists a constant $K_2$ so that for $v,w \in V_0T\setminus\{v_0\}$ with $w \in T_0(v)$ we have
			\[
			d(U_{\gets v}, U_{w\to}) \asymp_{K_2} d(\Lambda_{e_v} , \Lambda_{e_w}).
			\]
\end{lemma}
\begin{proof}
	Since $\Lambda_{e_v} \subset U_{\gets v}$ and $\Lambda_{e_w} \subset U_{w\to}$, we have $d(U_{\gets v}, U_{w\to}) \leq d(\Lambda_{e_v},\Lambda_{e_w})$. In particular, if $\Lambda_{e_v}=\Lambda_{e_w}$ then $d(U_{\gets v}, U_{w\to}) = 0 = d(\Lambda_{e_v},\Lambda_{e_w})$, so we can assume that $\Lambda_{e_v}$ and $\Lambda_{e_w}$ are disjoint.

	Suppose $x \in U_{\gets v}$ and $y \in U_{w\to}$.
	By the quasiconvexity of $Z_{e_v}$, a geodesic from $o$ to $x$ must lie in the $C$-neighbourhood of $Z_{\gets e_v}$.
	By the quasiconvexity of $Z_{e_w}$, a geodesic from $o$ to $y$ must consist of an initial segment of length $d_Z(o, Z_{e_w})$ from $o$ to a point within distance $C$ of $Z_{e_w}$, then a tail which remains within distance $C$ of $Z_{ e_w\to}$.
	If we let $e=e_v$ and $e'=e_w$ as in the proof of Lemma \ref{lem:reldist1} and consider Figure~\ref{fig:reldist1}, this means that $(x|y) \leq a+b+C$, and so $d(x,y) \succeq e^{-\epsilon (a+b)} \asymp d(\Lambda_{e_v},\Lambda_{e_w})$ by the argument of Lemma \ref{lem:reldist1}.  Thus $d(U_{\gets v}, U_{w\to})\succeq d(\Lambda_{e_v},\Lambda_{e_w})$ and we are done. 
\end{proof}
\begin{lemma}
	\label{lem:reldist3}
	There exists a constant $K_3$ so that for any $v \in V_0T\setminus\{v_0\}$ and $p \in U_{v\to}$, we have
	$d(p, U_{\gets v}) \asymp_{K_3} d(p, \Lambda_{e_v})$.
\end{lemma}
\begin{proof}
	Take $u \in V_0T$ so that $v\in T_0(u), d_T(u,v)=2$.
	Then since $\Lambda_{e_v} \subset \Lambda_u \subset U_{\gets v}$ we have $d(p,U_{\gets v}) \leq d(p, \Lambda_{e_v})$.

	Now suppose $p \notin \Lambda_{e_v}$.
	By the quasiconvexity of $Z_{e_v}$ a geodesic $\gamma$ from $o$ to $p$ travels from $o$ to within $C$ of a nearest point in $Z_{e_v}$ to $o$,
	then travels within $N_C Z_{e_v}$ to a point $q$, then stays in $Z_{ e_v\to} \setminus N_C Z_{e_v}$.
	Moreover, $d(p,\Lambda_{e_v}) \asymp e^{-\epsilon d_Z(o,q)}$.

	Suppose we have $y \in U_{\gets v} = \bdry Z_{\gets e_v}$.
	By the quasiconvexity of $Z_{e_v}$, a geodesic from $o$ to $y$ cannot stay close to $\gamma$ past $q$, thus $(y|p) \leq d_Z(o,q)+C'$.
	So $d(p,\Lambda_{e_v}) \asymp e^{-\epsilon d_Z(o,q)} \preceq d(p,y)$.
	Taking the infimum over all $y \in U_{\gets v}$, we conclude that $d(p,\Lambda_{e_v}) \preceq d(p,U_{\gets v})$.
\end{proof}

A vertex limit set $\Lambda_v$, its parent edge limit set $\Lambda_{e_v}$ and $U_{v\to}$, the part of $X$ containing $\Lambda_v$ which $\Lambda_{e_v}$ cuts out, all have comparable diameters.
\begin{lemma} \label{lem:reldist4}
	There exists a constant $K_4$ so that for $v \in V_0T\setminus\{v_0\}$, we have 
			\[
			\diam \Lambda_{e_v} \leq \diam \Lambda_v \leq \diam U_{v\to} \leq K_4 \diam \Lambda_{e_v}. \qedhere
			\]
\end{lemma}
\begin{proof}
As $\Lambda_{e_v} \subset \Lambda_v \subset U_{v\to}$, the first two inequalities are trivial.
	Now as $\Lambda_{e_v}$ is two-ended, $\diam \Lambda_{e_v} \asymp e^{-\epsilon d_Z(o, Z_{e_v})}$.
	By the quasiconvexity of $Z_{e_v}$ any geodesic $\gamma$ from $o$ with a tail in $Z_{ e_v\to}$ must satisfy $d_Z(\gamma(t),Z_{ e_v\to}) \leq C$ for all $t\geq d_Z(o, Z_{e_v})$.
	So by the definition of $U_{v\to}$, if $x \in U_{v\to}$ then the geodesic ray from $o$ to $x$ must have a tail in the $C$-neighbourhood of $Z_{ e_v\to}$ also.
	Thus for two points $x,y \in U_{v\to}$, the quasiconvexity of $Z_{e_v}$ implies that the geodesic line from $x$ to $y$ will live in a bounded neighbourhood of $Z_{ e_v\to}$, and hence $(x|y) \geq d_Z(o,Z_{e_v})-C$ thus 
	\[
		\diam U_{v\to}
		=\sup_{x,y \in U_{v\to}} d(x,y)
		\preceq e^{-\epsilon d_Z(o,Z_e)} \asymp \diam \Lambda_{e_v}.\qedhere
	\]
\end{proof}

The (relative) diameter of limit sets reflect the configuration of the corresponding vertex spaces.
\begin{lemma} \label{lem:reldist5}
There exists a constant $K_5$ so that
for any $v,w \in V_0T$ with $w \in T_0(v)$, we have
\begin{align}
\diam \Lambda_v & \asymp_{K_5} e^{-\epsilon d_Z(o, Z_v)}
\leq e^{-\epsilon d_T(v_0,v)}\label{eq:reldist51}, \text{ and}\\
\frac{\diam \Lambda_w}{\diam \Lambda_v}& \asymp_{K_5} e^{-\epsilon d_Z(p_v, Z_w)} \leq e^{-\epsilon d_T(v,w)},\label{eq:reldist52}
\end{align}			
where $p_v \in Z_v$ is a closest point in $Z_v$ to $o \in Z$.
\end{lemma}
\begin{proof}	
The projection $Z \to T$ that collapses each vertex space to a point, and each edge space to an edge is $1$-Lipschitz, so the second inequalities are trivial.

For the first inequality in (\ref{eq:reldist51}), as $Z_v$ is quasi-isometric to the coset $g_v G_{i_v}$ and $G_{i_v}$ is an infinite group, for any point $p \in Z_v$ there is a geodesic line $\gamma$ so that $d(\gamma,p)\leq C$ and $\gamma$ is in the $C$-neighbourhood $N_C(Z_v)$ of $Z_v$.
	Suppose $p_v\in Z_v$ is a closest point to $o$.
	As there is a geodesic line almost through $p_v$ which limits to points in $\Lambda_v$, we have $\diam \Lambda_v \succeq e^{-\epsilon d_Z(o,p_v)} \asymp e^{-\epsilon d_Z(o,Z_v)}$.
	On the other hand, for any distinct $x,y \in \Lambda_v$ if $\alpha$ is a geodesic line from $x$ to $y$, by quasiconvexity $\alpha \subset N_C(Z_v)$, and so $(x|y) \gtrsim d_Z(o,Z_v)-C$ and thus, taking the supremum over all $x,y \in \Lambda_v$, $\diam \Lambda_v \preceq e^{-\epsilon d_Z(o,Z_v)}$.
	
Let $p_v \in Z_v$ and $p_w \in Z_w$ be closest points to $o$ in $Z_v$ and $Z_w$, respectively.
	By (\ref{eq:reldist51}) we have $\diam \Lambda_w / \diam \Lambda_v \asymp e^{-\epsilon (d_Z(o,p_w)-d_Z(o,p_v))}$.
	By quasiconvexity and hyperbolicity, the geodesic from $o$ to $p_w$ passes within distance $C$ of $p_v$, and $p_w$ is within $C$ of a closest point to $p_v$ in $Z_w$, thus 
	\[ 
		\left| d_Z(p_v, Z_w) - (d_Z(o,p_w)-d_Z(o,p_v)) \right| \leq C,
	\]
	and the conclusion follows.
\end{proof}

Points in two different limit sets cannot be much closer to each other than they are to their first common ancestor.
\begin{lemma}\label{lem:reldist6}
There exists a constant $K_6$ so that if $v \in V_0T$, $w,w' \in T_0(v)$ with $d_T(v,w)=d_T(v,w')=2$ and $w\neq w'$, then for $x \in U_{w\to}$, $d(x,\Lambda_v) \leq K_6 d(x,U_{w'\to})$.
\end{lemma}
\begin{proof}
	Suppose $x \in U_{w\to}$ and $y \in U_{w'\to}$.
	By quasiconvexity of edge and vertex spaces, there is a $C$ so that
	a bi-infinite geodesic $\gamma$ from $x$ to $y$ has an initial tail in $N_C(Z_{ e_w\to})$, then a segment in $N_C(Z_v)$, then a terminal tail in $N_C(Z_{ e_{w'}\to })$ (these may overlap).
	Let $p$ be a closest point in $\gamma$ to $o$; necessarily $p \in N_C(Z_v)$.  Note that $e^{-\epsilon d_Z(o,p)}\asymp d(x,y)$.

	Take a bi-infinite geodesic $\beta$ in $N_C(Z_v)$ passing within distance $C$ of $p$.  Since $p$ is within $C$ of a geodesic ray from $o$ to $x$, either $(x|\beta(-\infty))\geq d_Z(o,p)-C'$ or $(x|\beta(+\infty))\geq d_Z(o,p)-C'$.  Without loss of generality, suppose the latter holds.  Then 
	\[
		d(x,\Lambda_v) \leq d(x, \beta(+\infty))
		\preceq e^{-\epsilon d_Z(o,p)} \asymp d(x,y).
	\]
	Taking the infimum of the right-hand side over all $y \in U_{w'\to}$ we get $d(x,\Lambda_v) \preceq d(x, U_{w'\to})$.
\end{proof}

Limit sets in the same orbit are, up to rescaling, uniformly bi-Lipschitz (as we do not use the explicit constant later, we just call it $C$).
\begin{lemma}\label{fact1}
	There exists $C$ so that for any $v \in T$, the metric spaces
	\[
	\left\{\frac{1}{\diam \Lambda_{gv}} \Lambda_{gv}\right\}_{g \in G}
	\]
	are all pairwise $C$-bi-Lipschitz.
\end{lemma}
\begin{proof}
	As there are finitely many vertex orbits it suffices to show the theorem for a fixed $v\in T$.
	We have that $g \Lambda_v = \Lambda_{gv}$ for any $g\in G$ by the equivariance of the map $Z\to T$.
	
	By approximate self-similarity (Lemma~\ref{lem:hypbdry-selfsim}) applied to a ball of radius $\diam \Lambda_{gv}$ around a point of $\Lambda_{gv}$, there exists $h\in G$ so that the map
	\[\left( \Lambda_{gv}, \frac{1}{\diam \Lambda_{gv}} d\right) \to (h \Lambda_{gv}, d)\]
	is bi-Lipschitz with uniform constant.
	Since $\diam h\Lambda_{gv} = \diam \Lambda_{hgv}$ is then $\succeq 1$, we have by Lemma~\ref{lem:reldist5} that $d_Z(o, Z_{hgv}) \leq C_1$ for some constant $C_1$.

	Recall that $g_v o \in Z_v$, so $g g_v o \in gZ_v = Z_{gv}$.
	Let $g_1, \ldots, g_k \in G$ be chosen so that any $Z_{g'v}$, $g'\in G$, with $d_Z(o,Z_{g'v})\leq C_1$ has $g'v = g_i v$ for some $i \in \{1,\ldots,k\}$.
	Moreover, we can choose $g_i$ so that $g_i g_v o \in Z_{g_i v}$ is a closest point to $o$ in the orbit $Go \cap Z_{g_i v}$, and so $d_Z(o, g_ig_v o)\leq C_2$.
	Thus for any $i,j \in \{1,\ldots,k\}$, 
	\[
	d_Z(o, g_j g_i^{-1} o)
	\leq d_Z(o, g_j g_v o) + d_Z(g_j g_v o, g_j g_i^{-1} o)
	\leq C_2 + d_Z(g_i g_v o, o)
	\leq 2C_2.
	\]

	Suppose for any two $i, j \in \{1,\ldots,k\}$ we map $\Lambda_{g_iv}$ to $\Lambda_{g_jv}$ by $h':=g_jg_i^{-1}$.
	Then for any two points $x,y \in \Lambda_{g_i v}$, 
	we have
	$d(x,y)\asymp e^{-\epsilon (x|y)_o} = e^{-\epsilon (h'x | h'y)_{h'o}}$,
	and $d(h'x,h'y) \asymp e^{-\epsilon (h'x|h'y)_o}$.
	As $|(h'x|h'y)_o-(h'x|h'y)_{h'o}|\leq d_Z(o,h'o) \leq 2C_2$,
	we then have that the map $h'$ acts to send $\Lambda_{g_i v}$ to $\Lambda_{g_j v}$ in a uniformly bi-Lipschitz way.
	
	So in conclusion, by a uniformly bi-Lipschitz map one can send any of the spaces $\left(\Lambda_{gv}, \frac{1}{\diam \Lambda_{gv}} d\right)$ to one of a finite set of spaces $\Lambda_{g_1 v},\ldots,\Lambda_{g_k v}$ where each $\diam \Lambda_{g_i v} \asymp 1$, and these spaces are each pairwise bi-Lipschitz with uniform constants.  
\end{proof}

A metric space $X$ is \emph{$C$-uniformly perfect} if for any $x \in X, r \in (0,\diam X)$, we have $B(x,r)\setminus B(x,r/C)\neq\emptyset$.
This property is preserved by quasisymmetric maps, up to changing the constant $C$ (see \cite[Exercise 11.2]{Hein-01-lect-an-mtc-spc}).
For completeness, we recall that a homeomorphism $f:X\to X'$ is \emph{quasisymmetric} if there exists a homeomorphism $\eta:[0,\infty)\to[0,\infty)$ so that for all $x,y,z \in X$, $d(x,y)\leq t d(x,z)$ implies that $d(f(x),f(y))\leq \eta(t)d(f(x),f(z))$~\cite{Tuk-Vai-80-QS}.
\begin{lemma}\label{fact2}
	There exists $C$ so that for any $v \in V_0T$, the metric space $\Lambda_{v}$ is $C$-uniformly perfect.
\end{lemma}
\begin{proof}
	Suppose $H$ is an infinite hyperbolic group, and $\bdry H$ is endowed with a visual metric with visual parameter $\epsilon$.

	If $\bdry H$ has at least $3$ points there is an ideal hyperbolic triangle limiting to distinct points $y_1, y_2, y_3 \in \bdry H$.
	For any $x \in \bdry H$ and $r \in (0,\diam \bdry H]$ there exists $h \in H$ so that the action of $h$ moves the quasi-centre of the ideal triangle to a point $p$ on the geodesic from the basepoint $o$ to $x$ at distance $\approx \frac{-1}{\epsilon} \log r$ from $o$.  Inspecting the tree approximation to $o, x, hy_1,hy_2,hy_3$, we see that for at least one $i \in \{1,2,3\}$, $(x|hy_i)_o$ is approximately $d_H(o,p)$, and so $d_{\bdry H}(x, hy_i) \asymp r$.  This suffices to show $\bdry H$ is uniformly perfect.
	
	For any $v \in V_0T$, since $\Lambda_v$ has more than two points then as the map $\bdry g_v G_{i_v} \to \Lambda_v$ is a quasisymmetry, $\bdry G_{i_v}$ has more than two points and so is uniformly perfect.
	The composition $\bdry G_{i_v} \to g_v \bdry G_{i_v} \to \Lambda_v$ is a quasisymmetry, and so $\Lambda_v$ is uniformly perfect too.
	By Lemma~\ref{fact1}, up to rescaling the spaces $\{\Lambda_{gv}\}$ are uniformly bi-Lipschitz, so $\{\Lambda_{gv}\}$ are uniformly uniformly perfect.
	As there are only finitely many vertex orbits in $T$ we are done.
\end{proof}

\section{Conformal dimension and Combinatorial modulus}
  \label{sec:defns-modulus}

  In this section we describe how conformal dimension can be calculated using combinatorial modulus by work of \cite{bou-kle13-CLP,carr13-conf-gauge}.
  Using this we reduce Theorem~\ref{thm:main-precise} to a statement about such modulus, Theorem~\ref{thm:weight-statement} below.

First, a complete metric space $X$ is \emph{Ahlfors ($Q$-)regular} if for some $Q \geq 0$ there is a Borel measure $\mu$ on $X$ so that for all $x \in X, r \in (0, \diam X]$ we have $\mu(B(x,r)) \asymp r^Q$.  In such a situation $Q$ must equal the Hausdorff dimension of $X$, and moreover $\mu$ must be comparable to the Hausdorff $Q$-measure on $X$.

If $Z$ is a Gromov hyperbolic space admitting a geometric action (that is, a proper and cocompact action by isometries) by a finitely generated group, then the boundary $\bdry Z$ endowed with a visual metric is Ahlfors regular by work of Coornaert~\cite{coornaert1993}.
We work with the following variation on Pansu's conformal dimension.
\begin{definition}
  \label{def:confdim}
  Let $X$ be a metric space.
  Then the \emph{(Ahlfors regular) conformal dimension} of $X$ is
  the infimum of all $Q$ such that $X$ is quasisymmetric to an Ahlfors $Q$-regular space.
\end{definition}
If $G$ is a Gromov hyperbolic group then $\Confdim \bdry G$ is a well-defined invariant of $G$, and if a group $H$ is quasi-isometric to $G$ then $\Confdim \bdry H = \Confdim \bdry G$.

The (Ahlfors regular) conformal dimension of a space which is approximately self-similar can be calculated using estimates on `combinatorial modulus' \cite{bou-kle13-CLP,carr13-conf-gauge}, which we now go on to describe.

We fix a large constant $a>1$ from now on ($a\geq 2$ suffices). 
For each $i \in \N$, let $X_i$ be a maximal $a^{-i}$-separated set in $X$, and let $\cS_i = \{B(x,a^{-i})\}_{x \in X_i}$ be the corresponding cover of $X$.

 For $\delta>0$ let $\Gamma_\delta$ be the collection of all paths in $X$ of diameter $\geq \delta$.

Let $\rho_n:\cS_n \ra [0,\infty)$ be a function (a ``weight function'').  We say that $\rho_n$ is \emph{$\Gamma_\delta$-admissible} if for any $\gamma \in \Gamma_\delta$, we have
\[
\ell_{\rho_n}(\gamma) := \sum_{A\in \cS_{n}, A \cap \gamma \neq \emptyset} \rho_n(A) \geq 1.
\]
The \emph{$\cS_n$-combinatorial $p$-modulus of $\Gamma_\delta$} is defined by
\[
\Mod_p(\Gamma_\delta,\cS_{n}) := \inf_{\rho_n} \Vol_p(\rho_n), \text{ where } \Vol_p(\rho_n) := \sum_{A \in \cS_{n}} \rho_n(A)^p
\]
and where we infimise over all $\Gamma_\delta$-admissible $\rho_n:\cS_n \to [0,\infty)$.
The \emph{critical exponent} for the $p$-modulus is defined by
\[
p_c(\delta) := \inf \left\{ p>0 : \liminf_{n\to\infty} \Mod_p(\Gamma_\delta,\cS_n) = 0 \right\}.
\]
\begin{theorem}[{Keith--Kleiner, Carrasco~\cite[Corollary 3.13]{carr13-conf-gauge}}]
	\label{thm:kkc-confdim-hypgroup}
	If $G$ is a hyperbolic group acting geometrically on an unbounded geodesic (hyperbolic) space $Z$, with boundary at infinity $X= \bdry Z$ endowed with a visual metric $d$ and $p_c(\delta)$ defined as above, then there exists $\delta_0>0$ so that for all $0 < \delta \leq \delta_0$,
	\[
		\Confdim \bdry G = \Confdim X = p_c(\delta).
	\]
\end{theorem}
\begin{proof}
	Such an $X$ equipped with a visual metric satisfies the hypotheses of \cite[Corollary 3.13]{carr13-conf-gauge} by Lemmas~\ref{lem:hypbdry-selfsim} and \ref{lem:bdry-split-over-finite}.
\end{proof}

In order to estimate $p_c(\delta)$, it actually suffices to show that $\Mod_p(\Gamma_\delta,\cS_n)$ is bounded independently of $n$, provided the maximum value of $\rho_n$ goes to zero:
\begin{lemma}[{Bourdon--Kleiner \cite[Corollary 3.7(3)]{bou-kle13-CLP}}]
	\label{lem:vol-max-p-bounds-combine}
	For any $p\geq 1$ and $\delta$, 
	for some $\cS_n$, $\Gamma_\delta$ as above,
	if there exists $\rho_n: \cS_n \to [0,\infty)$ weights that are $\Gamma_\delta$-admissible, and
	$\| \rho_n \|_\infty \to 0$ as $n\to\infty$, and $\sup_n \Vol_p(\rho_n) < \infty$, then
	$p_c(\delta) \leq p$.
\end{lemma}
\begin{proof}
	For any $\epsilon>0$, 
	\[
		\Vol_{p+\epsilon}(\rho_n) = 
		\sum_{A \in \cS_n} \rho_n(A)^{p+\epsilon}
		\leq \|\rho_n\|_\infty^\epsilon \Vol_p(\rho_n) \to 0 \text{ as } n \to \infty,
	\]
	therefore $p_c(\delta) \leq p+\epsilon$; as $\epsilon$ was arbitrary we are done.
\end{proof}

So for each $p$ bigger than our intended upper bound, it will suffice to find $\delta \in (0,\delta_0)$ and such a sequence of weight functions.

\begin{theorem}\label{thm:weight-statement}
	Suppose $G$ is as in the statement of Theorem~\ref{thm:main-precise} and $X$ a visual metric on the boundary of the model space arising from the tree of cylinders construction of Proposition~\ref{prop:tree-of-cylinders}, and $\delta >0$ is fixed.
	
	Then for any $p > 1\vee \max\{\Confdim \bdry G_i\}$,
	there exists weight functions $\rho_n$ on $\cS_n$ so that 
	each $\rho_n$ is $\Gamma_{\delta}$-admissible, 
	$\lim_{n\to\infty} \|\rho_n\|_\infty =0$, 	
	and the sequence $\Vol_p(\rho_n)$ is bounded.
\end{theorem}
This theorem will be proved in subsequent sections, as we now summarise.
\begin{proof}
	The weights are defined in Section~\ref{sec:weight-def}, up to a choice of parameters $\delta'$, $E_1$, $E_2$ and $E_3$.
	Theorem~\ref{thm:maxbound} shows that $\lim_{n\to\infty} \|\rho_n\|_\infty = 0$ and fixes the value of $E_2$.
	Admissibility is shown, for suitable (now fixed) parameters $\delta'$, $E_1$ and $E_3$, by Theorem~\ref{thm:admissible}.
	The uniform bounds on $\Vol_p(\rho_n)$ are then shown by Theorem~\ref{thm:volbound}.
\end{proof}

\begin{proof}[Proof of Theorem~\ref{thm:main-precise}]
  The lower bound 
  \[ \Confdim \bdry G \geq 1\vee \max\{\Confdim \bdry G_i\} \]
  follows from the fact that $G$ is not virtually free and that each vertex group $G_i$ is quasiconvex in $G$.

  For the upper bound, let $\delta_0$ be given by Theorem~\ref{thm:kkc-confdim-hypgroup} for $X=\bdry G$, and fix $\delta \in (0,\delta_0]$.
  By Theorem~\ref{thm:kkc-confdim-hypgroup}, Lemma~\ref{lem:vol-max-p-bounds-combine} and Theorem~\ref{thm:weight-statement} we then have 
  \[
  	\Confdim \bdry G  = p_c(\delta) \leq 1\vee \max\{\Confdim \bdry G_i\}.
	\qedhere
	\]
\end{proof}

\section{Candidate weight function}\label{sec:weight-def}

Our goal in this section is, given a choice of $p > \max \{\Confdim \bdry G_i\}$, to define suitable weight functions as in Theorem~\ref{thm:weight-statement}.
The idea is similar to that of the example in Subsection~\ref{ssec:toyexample}: to iteratively define weights that turn geometric sequences of scales into arithmetic.  There are additional complications which we describe as they arise.

We continue with the notation of Section~\ref{sec:bdry-structure}, and $T$ is the tree of cylinders of Proposition~\ref{prop:tree-of-cylinders} with $VT = V_0T \sqcup V_1T$.
Let $v_0\in V_0T$ be the fixed basepoint in $T$.  

\subsubsection*{Projections to $T$}
We project $\cS_n$ onto $T$ as follows: for $A \in \cS_n$, define the \emph{tree projection $\pi(A) \in V_0T$} to be the closest vertex to $v_0$ in the convex hull
\begin{equation*}
\Conv\left(v \in T_0V : \Lambda_v \cap A \neq \emptyset \right)
\end{equation*}
of all vertices whose limit set intersect $A$.
The relationship between $A$ and $\Lambda_{\pi(A)}$ is indicated by the following:
\begin{lemma}\label{lem:piA-properties}
	There exists $K_7 \geq 1$ so that
	for $A \in \cS_n$, $\diam \Lambda_{\pi(A)} \geq \frac{1}{K_7} a^{-n}$, and the distance from the centre of $A$ to $\Lambda_{\pi(a)}$ is at most $K_7a^{-n}$. 
\end{lemma}
\begin{proof}
	If $\pi(A)=v_0$ the bounds are trivial, so assume otherwise. Since $A$ is centred on a point $p\in U_{\pi(A)\to}$ and does not meet $\Lambda_{e_{\pi(A)}}$, by Lemma~\ref{lem:reldist3} $d(p,U_{\gets \pi(A)}) \asymp d(p, \Lambda_{e_{\pi(A)}}) \succeq a^{-n}$, thus $U_{\pi(A)\to} = \overline{X\setminus U_{\gets \pi(A)}}$ contains a ball centred on $p$ of radius $\asymp a^{-n}$. So by Lemma~\ref{lem:reldist4} and the uniform perfectness of $X$,
	$\diam \Lambda_{\pi(A)} \succeq \diam U_{\pi(A)\to} \geq a^{-n}$.

	If $A \cap \Lambda_{\pi(A)} \neq \emptyset$, we are done for any $K_7\geq 1$.
	Otherwise $A \cap \Lambda_{\pi(A)} = \emptyset$, but by the definition of $\pi(A)$, $A$ must meet $U_{w\to}$ and $U_{w'\to}$ for two distinct $w,w' \in T_0(v)$ with $d_T(v,w)=d_T(v,w')=2$.
	Therefore by Lemma~\ref{lem:reldist6}, $d(A, \Lambda_{\pi(A)}) \preceq a^{-n}$.
\end{proof}

Given $v,w \in V_0T$, let $[v,w] \subset V_0T$ be the unique simple path from $v$ to $w$.
Suppose $A \in \cS_n$ and $[v_0, \pi(A)]$ consists of $v_0, v_1, \ldots, v_m=\pi(A)$.
If $v = v_i$ for some $i \in \{0,1,\ldots,m-1\}$ then let $v_{\to A} = v_{i+1}$; if $v=\pi(A)$ then let $v_{\to A} = \pi(A)$; and if $v \notin [v_0,\pi(A)]$ let $v_{\to A}$ be undefined.

Let us also define for any $\delta>0$ 
\begin{equation}\label{eq:defTdelta}
	T_\delta := \Conv \left( \{v_0\} \cup \{v \in V_0T : \diam \Lambda_v > \delta \} \right),
\end{equation}
which is the convex hull of the finite set of vertices in $T$ whose limit sets are large (see the first equality in \eqref{eq:reldist51}); such sets will be used in the definition below.

\subsubsection*{Model spaces}
We are given a choice of $p > \max \{\Confdim \bdry G_i\}$, and want to define suitable weight functions as in Theorem~\ref{thm:weight-statement}.
For $v \in V_0T$, fix $Q_v \in [\Confdim \bdry G_v, p)$, with the choice uniform on each $G$-orbit.

For $v \in V_0T$, let $D_v = \diam \Lambda_v$.
For each $G$-orbit  $G v \subset V_0T$, the collection of rescaled spaces $\{\frac{1}{D_{gv}} \Lambda_{gv}\}$ are all uniformly bi-Lipschitz to each other (Lemma~\ref{fact1}).
For each $v \in T$, we fix a $Q_v$-regular space $X_v=(X_v,d_v)$ of diameter $1$ in the conformal gauge of $\bdry G_v$, and an $\eta$-quasisymmetry map $h_v:\Lambda_v \to X_v$.
Again by Lemma~\ref{fact1}, $X_v$ and $h_v$ may be chosen so that the maps $h_{gv} : \frac{1}{D_{gv}}\Lambda_{gv} \to X_{gv}$ have $X_{gv}$ independent of $g$ and the different maps $h_{gv}$ differing from each other only by a uniform bi-Lipschitz homeomorphism.
(This last condition means that there exists $C$ so that for any $g, g' \in G$, there exists a $C$-bi-Lipschitz homeomorphism $f: \frac{1}{D_{gv}}\Lambda_{gv}\to \frac{1}{D_{g'v}}\Lambda_{g'v}$ so that $h_{gv} = h_{g'v} \circ f$.)
Finally, the distortion function $\eta$ may be chosen uniformly for all $v$, as dilations do not affect distortion.

As $\eta$ is fixed and the spaces $\Lambda_v$ are uniformly perfect with constant independent of $v$ (Lemma~\ref{fact2}), we can find $\tau \in (0,1]$ and $\lambda \geq 1$ so that the maps $h_{v} : \frac{1}{D_{v}}\Lambda_{v} \to X_{v}$ are uniform $(\tau,\lambda)$-bi-H\"older maps by \cite[Theorem 3.14]{Tuk-Vai-80-QS}, i.e.\ for all $v \in V_0T$ and all $x,y \in \Lambda_v$,
\begin{equation}\label{eq:taulambdaHolder}
\frac{1}{\lambda} \left(\frac{d(x,y)}{D_v}\right)^{1/\tau} 
\leq d_v(h_v(x),h_v(y)) \leq
\lambda \left(\frac{d(x,y)}{D_v}\right)^{\tau}.
\end{equation}

When we push the cover $\cS_n$ forward by $h_v$ to $X_v$, it is useful to know that the images are contained in balls of radius smaller than $a^{-m_v}/2$ for a suitable $m_v$; by \eqref{eq:taulambdaHolder} we can take
\begin{equation}\label{eq:defmv}
	m_v := \lfloor \tau(n+\log_a D_v) -\log_a (2\lambda) \rfloor \vee 0.
\end{equation}

\subsubsection*{Definition of weight function}
For each $n \in \N$, and a constant $E_1$ found later, we define the weight function $\rho_n : \cS_n \ra \R_+$ by 
\begin{equation}\label{eq:defrhon}
	\rho_n(A) := E_1 a^{-n} \prod_{v \in V_0T} \rho^n_v(A).
\end{equation}

For $v \in V_0T$ and $A \in \cS_n$, 
$\Lambda_v$ and $A$ can interact in three ways according to whether $v \notin [v_0,\pi(A)], v \in [v_0,\pi(A))$ or $v=\pi(A)$.
In the first case, we don't want $\rho^n_v$ to influence $\rho_n(A)$ at all; in the latter two we need to define a subset of $\Lambda_v$ corresponding to the location of $A$ in or near $\Lambda_v$:
\begin{equation}\label{eq:defWvA}
	W_{v,A} := 
	\begin{cases}
		\emptyset & \text{if } v \notin [v_0,\pi(A)],\\
		\Lambda_{e_{v_{\to A}}} & \text{if } v \in [v_0,\pi(A)),\\ 
		B_A & \text{otherwise},
	\end{cases}
\end{equation}
where $B_A$ is a ball in $\Lambda_{v_{\to A}}=\Lambda_v$ of radius $a^{-n}$ centred on a point at most $K_7a^{-n}$ from the centre of $A$; such a ball exists by Lemma~\ref{lem:piA-properties}.

How shall we define $\rho_v^n(A)$?
The first ingredient is to distort according to $h_v$: we want the relative size of $W_{v,A}$ in $\Lambda_v$ to match the relative size of $h_v(W_{v,A})$ in $X_v$, so we have a factor of 
$\frac{\diam h_v(W_{v,A}) / 1}{\diam W_{v,A} / D_v}$.

The second ingredient is to transform geometric to arithmetic scales:
for $i=1,\ldots,m_v$ the annulus of points at distance $[a^{-(i+1)},a^{-i}]$ from $h_v\Lambda_{e_v}$ should be sent to an annulus of width $\frac{1}{m_v}$, so sets at distance $\sim a^{-i}$ from $h_v \Lambda_{e_v}$ should be stretched by $\sim \frac{1}{m_v a^{-i}}$.
But we don't want to do this to large vertex limit sets, or if a set is too close to $\Lambda_{e_v}$, as either could interfere with showing admissibility.
So, for $W \subset \Lambda_v$ we let
\begin{equation}\label{eq:deffv}
	f_v(W) := 
	\begin{cases}
		1 & \text{if } d_v(h_v W,h_v\Lambda_{e_v}) \leq a^{-m_v} \text{ or } v\in T_{\delta'}\text{, and} \\
		m_v d_v(h_v W, h_v \Lambda_{e_v}) & \text{otherwise},
\end{cases}
\end{equation}
where $T_{\delta'}$ is the finite subtree defined by \eqref{eq:defTdelta} for a suitable parameter $\delta'\in (0,\diam \Lambda_{v_0})$ determined later. 
Note that we only use $e_v$ in \eqref{eq:deffv} when it is defined since $v\notin T_{\delta'}$ implies $v \neq v_0$.

Combining the two deformations leads us to define, for each $v \in V_0T$, 
\begin{equation}\label{eq:defrhonv}
	\rho^n_v(A) :=
	\begin{cases}
		1 & \begin{aligned} &\text{if $d(A, U_{\gets v}) \leq E_2 a^{-n}$ or}\\ & \text{$m_v \leq 1$ or $W_{v,A}=\emptyset$, and} \end{aligned} \\
		{\displaystyle \frac{\diam h_v(W_{v,A})}{\diam W_{v,A}} \frac{E_3 D_v}{f_v(W_{v,A})}} & \text{otherwise.}
	\end{cases}
\end{equation}
Here $E_2,E_3$, along with $E_1$ from \eqref{eq:defrhon}, are constants we choose later.

By Lemma~\ref{lem:reldist5} and \eqref{eq:defmv}, for a given $n$ there are finitely many $v$ with $m_v > 1$, so $\rho_n$ is well-defined,
given choices of the constants $\delta'$, $E_1$, $E_2$ and $E_3$.

\section{Bounding the maximum value of $\rho_n$}
	\label{sec:maxbound}
	Recall that the idea of $\rho_n$ is to send a geometric sequence of annuli of points in $\Lambda_v$ at distance $[a^{-(i+1)},a^{-i}]$ for $i=1,\ldots,k$ (and suitable $k$) to an arithmetic sequence of annuli of points at distance $[\frac{i+1}{2k+1},\frac{i}{2k+1}]$.
	In particular, for some $v\notin T_{\delta'}$ but with $\Lambda_v\asymp 1$, the smallest annulus in $\Lambda_v$ has size $\asymp a^{-n}$, so is covered by boundedly many balls in $\cS_n$.  The $\rho_n$ value of these balls will be $\asymp 1/n$, giving a heuristic estimate $\|\rho_n\|_\infty \succeq 1/n$. This is essentially the worst case, as we now show.

\begin{theorem}
	\label{thm:maxbound}
	For $\rho_n$ as in Section~\ref{sec:weight-def}, for large enough $E_2$ and for any $E_1, E_3, \delta'$, $\lim_{n\to\infty} \|\rho_n\|_\infty = 0$.
\end{theorem}
\begin{proof}
	Given $A \in \cS_n$, consider the path $[v_0, \pi(A)] = \{v_0, v_1, \ldots, v_k=\pi(A)\}$ in $V_0T$.
	By \eqref{eq:defrhon}, \eqref{eq:defWvA}, \eqref{eq:defrhonv},
	\begin{equation}\label{eq:max0}
		\rho_n(A) = E_1 a^{-n} \prod_{i=0}^k \rho_{v_i}^n(A).
	\end{equation}
	In the proof we track the dependence of constants on $E_1, E_2, E_3, \delta'$.

	\emph{Step 1:} Let $t \geq 0$ be maximal with $v_t \in T_{\delta'}$.
	As $T_{\delta'}$ is finite, $t \leq C_1$ for some constant $C_1=C_1(\delta')$.
	For $i \leq t$, we have $f_{v_i}(W_{v_i,A})=1$ and $D_{v_i} \asymp 1$.
	For $i < t$, $\diam W_{v_i, A} \asymp 1$ and $\diam h_{v_i}(W_{v_i,A}) \asymp 1$.
	So
	\begin{equation}
	  \label{eq:max1}
		\prod_{i=0}^{t-1} \rho_{v_i}^n(A) \asymp_{C(\delta',E_3)} 1
	  \quad \text{ and } \quad
		\rho_{v_t}^n(A) \asymp_{C(\delta',E_3)} \frac{\diam h_{v_t}(W_{v_t,A})}{\diam W_{v_t,A}}.
	\end{equation}
	
	\emph{Step 2:} 	A useful fact is the following: 	
	by Lemma~\ref{lem:reldist4}, for $0 \leq i < k$, as $W_{v_i,A}=\Lambda_{e_{(v_i)_{\to A}}}=\Lambda_{e_{v_{i+1}}}$ we have
	\begin{equation}\label{eq:max2}
		1 \leq \frac{ D_{v_{i+1}} }{ \diam W_{v_i, A} } \leq K_4.
	\end{equation}

	\emph{Step 3:}
	Consider the definition of $\rho_v^n$ in \eqref{eq:defrhonv}.
	Suppose for some $i \in \{0,\ldots,k\}$ we have $d(A,U_{\gets v_i}) \leq E_2 a^{-n}$ or $m_{v_i} \leq 1$, then let $s$ be the minimal such $i$.  If no such $i$ exists, set $s=k+1$.  If $s \leq k$ then either

	(1) $d(A, U_{\gets v_s}) \leq E_2 a^{-n}$, and so $\rho_{v_i}^n(A) =1$ for all $i \geq s$, 
	
	(2) $m_{v_s} \leq 1$ so $D_{v_s} \asymp a^{-n}$, and so $s \leq k \leq s+C$ for some $C$ by Lemmas~\ref{lem:reldist5} and \ref{lem:piA-properties}.
	For each $i \geq s$ we have $m_{v_i} \leq C$, $\diam h_{v_i}(W_{v_i,A}) \asymp 1$, and $\frac{ D_{v_i} }{ \diam W_{v_i, A} } \asymp \frac{a^{-n}}{a^{-n}} =1$.
	If $m_{v_i} \leq 1$ then $\rho_{v_i}^n(A)=1$, else $m_{v_i}\in(1,C]$ thus $f_{v_i}(W_{v_i,A}) \asymp 1$ and so $\rho_{v_i}^n(A) \asymp_{C(E_3)} 1$ also.
	
	Therefore in either case (1) or (2) we have
	\begin{equation}\label{eq:max4}
		\prod_{i=s}^k \rho_{v_i}^n(A) \leq C_2=C_2(E_3).
	\end{equation}

	\emph{Step 4:} Now for every $t <  i < s$ we claim that
	\begin{equation}
	  \label{eq:max3}
	  \frac{\diam h_{v_i}(W_{v_i,A})}{f_{v_i}(W_{v_i,A})} 
		\leq \frac{C}{m_{v_i}}.
	\end{equation}
	First, if $i<k$ then by Lemma~\ref{lem:reldist1} the relative distance of $W_{v_i,A}=\Lambda_{e_{v_{i+1}}}$ and $\Lambda_{e_{v_i}}$ is bounded below.  If $i=k<s$ then as $d(A,U_{\gets v_k}) > E_2 a^{-n}$ we have that the relative distance of $W_{v_k,A}=B_A$ and $\Lambda_{e_{v_k}}$ is bounded below by the definition of $B_A$, provided we fix $E_2:=K_7+2$ say, by Lemma~\ref{lem:piA-properties}.  Since uniformly quasisymmetric maps uniformly distort relative distances (e.g.\ \cite[Lemma 3.2]{BK-02-S2-unif}),
	\[
		\frac{\diam h_{v_i}(W_{v_i,A})}{m_{v_i} d_{v_i}(h_{v_i}W_{v_i,A},h_{v_i}\Lambda_{e_{v_i}})}
		\leq \frac{C_3(E_2)}{m_{v_i}}.
	\]
 	Second, if $d_{v_i}(h_{v_i} W_{v_i,A}, h_{v_i}\Lambda_{e_{v_i}}) \leq a^{-m_{v_i}}$  then for $i<k$
	since the relative distance of $W_{v_i,A}$ and $\Lambda_{e_{v_i}}$ is $\geq 1/C$, so the relative distance of $h_{v_i} W_{v_i,A}$ and $h_{v_i} \Lambda_{e_{v_i}}$ is $\geq 1/C$, but this last relative distance is also $\leq a^{-m_{v_i}} / \diam h_{v_i} W_{v_i,A}$, we thus have $\diam h_{v_i}W_{v_i,A} \preceq a^{-m_{v_i}}$.
	If $i=k$, as $W_{v_i,A}$ is a ball of radius $a^{-n}$, $\diam h_{v_i}(W_{v_i,A}) \preceq (a^{-n}/D_{v_i})^{\tau} \asymp a^{-m_{v_i}}$.
	So for $i<k$ or $i=k$ in this second case we have
	\[
	\frac{\diam h_{v_i}(W_{v_i,A})}{f_{v_i}(W_{v_i,A})} = \diam h_{v_i}(W_{v_i,A}) 
		\leq C a^{-m_{v_i}} \leq \frac{C}{m_{v_i}}.
	\]

	\emph{Step 5:}
	By \eqref{eq:max0}, \eqref{eq:max1}, \eqref{eq:max4}, \eqref{eq:max3}
	we have
	\begin{align*}
		\rho_n(A) & \preceq_{C(E_1,\delta',E_3,E_2)}
	  a^{-n} \cdot 
	  \frac{\diam h_{v_t}(W_{v_t,A})}{\diam W_{v_t,A}}
	  \cdot \prod_{i=t+1}^{s-1} \frac{C D_{v_i} }{m_{v_i} \diam W_{v_i, A}}
	\end{align*}

	If $s-1 < t+1$ this last product is vacuous.
	In this case by \eqref{eq:taulambdaHolder}
	\[
		\rho_n(A) \preceq_{C(\delta')} \frac{a^{-n}}{(\diam W_{v_t,A})^{1-\tau}}
		\preceq \frac{a^{-n}}{a^{-n(1-\tau)}} = a^{-\tau n} \preceq \frac{1}{n}.
	\]
	So we may assume $t+1 \leq s-1$.

	By \eqref{eq:max2} for $t+1 \leq i \leq s-1$ we have that $C D_{v_i} / \diam W_{v_{i-1},A}$ is bounded.  As the sequence $m_{v_i}$ is roughly decreasing at least linearly in $i$ (by~\eqref{eq:reldist52}), for all but boundedly many terms at the tail of the sequence $i=t+1,\ldots,s-1$ we have that $C D_{v_i} / (m_{v_i} \diam W_{v_{i-1,A}}) \leq C^2 / m_{v_i} \leq 1$.
	Once $m_{v_i}$ is small (but still $\geq 1$), $D_{v_i} \asymp a^{-n}$ and $\diam W_{v_{i-1},A} \asymp a^{-n}$ also, so $C D_{v_i}/(m_{v_i} \diam W_{v_{i-1},A}) \preceq 1$.  Taken together, applying these bounds for $i=t+2,\ldots,s-1$, we have
	\begin{align*}
	  \rho_n(A) & \preceq
	  a^{-n} \cdot 
	  \frac{\diam h_{v_t}(W_{v_t,A})}{\diam W_{v_t,A}} 
		\cdot \frac{D_{v_{t+1}}}{m_{v_{t+1}}} \cdot \frac{1}{\diam W_{v_{s-1}, A}}
	  \\ & \preceq \frac{\diam h_{v_t}W_{v_t,A}}{m_{v_{t+1}}}
	\end{align*}
	by \eqref{eq:max2} and $\diam W_{v_{s-1},A} \succeq a^{-n}$.

	If $m_{v_{t+1}} \geq \tau n/2$, then $\rho_n(A) \preceq \frac{1}{n}$.
	Otherwise $m_{v_{t+1}} < \tau n/2$ so 
	as $D_{v_t} \asymp 1$ we then have by \eqref{eq:taulambdaHolder} that
	$\diam h_{v_t}W_{v_t,A} \preceq D_{v_{t+1}}^\tau \asymp a^{-\tau n+m_{v_{t+1}}} \leq a^{-\tau n /2}$
	and $\rho_n(A) \preceq \frac{a^{-\tau n/2}}{1} \preceq \frac{1}{n}$.

	As in either case $\rho_n(A) \preceq \frac{1}{n}$, we are done.
\end{proof}

\section{Admissibility}\label{sec:admissibility}

Our goal in this section is to show that for $\delta<\delta_0$ there are suitable choices of parameters $\delta', E_1, E_3$ making the weight  $\rho_n : \cS_n \ra \R$ as defined as in \eqref{eq:defrhon} admissible for $\Gamma_\delta$.
We now treat the parameter $E_2$ as a fixed constant given by Section~\ref{sec:maxbound}.
\begin{theorem}
	\label{thm:admissible}
	For $\delta<\delta_0$ fixed, we can find $\delta'\in (0,\delta]$ and $E_1, E_3$ large enough independent of $n$ so that $\rho_n$ defined as in Section~\ref{sec:weight-def} is $\Gamma_\delta$-admissible for all $n$.
\end{theorem}

Recall from \eqref{eq:defTdelta} that 
\begin{equation*}
	T_\delta = \Conv \left( \{v_0\} \cup \{v \in T : \diam \Lambda_v > \delta \} \right)
\end{equation*}
is the convex hull of the finite set of vertices in $T$ whose limit sets are large.

Curves in $\Gamma_{\delta}$ need not be embedded and can start and end at arbitrary points in $X$; the following proposition finds a nice subcurve for any $\gamma\in\Gamma_\delta$.
\begin{proposition}\label{prop:tame}
	There exist $\delta'\in (0,\delta]$ so that:

	Given $\gamma \in \Gamma_\delta$, we can find an arc $\hat{\gamma} \in \Gamma_{\delta'}$ so that
	\begin{enumerate}
	  \item $\hat{\gamma}$ is contained in the 
		image of $\gamma$.  
		\item $\hat{\gamma}$ is contained in $U_{v\to}$ and has endpoints at least $\delta'$ apart in $\Lambda_v$, for some $v \in T_{\delta'}$.
	\end{enumerate}
\end{proposition}

Before proving this, in the following lemma we relate points in $X$ with points in $\bar T$, the compactification of $T$.
For $x \in X$, let $\Pi(x) \subset \bar T$ be the corresponding point(s) in $\bar T$ determined by Lemma~\ref{lem:bdry-limsets-2}: 
$\Pi(x)$ is either a unique point in $\bdry T$,
a closed ball of radius $1$ around a unique $e \in V_1 T$ (with $x\in \Lambda_e$),
or a unique $v\in V_0 T$ (with $x \in \Lambda_v$).
\begin{lemma}\label{lem:Tbar}
	For $G \curvearrowright T$ as in Proposition~\ref{prop:tree-of-cylinders}, and $\Pi$ as above,
			if $C \subset X$ is connected, then $\Pi(C) := \bigcup_{x\in C} \Pi(x)$ is connected.
\end{lemma}
\begin{proof}
	Suppose $\Pi(C)$ is disconnected.
	Then as $\Pi(C) \subset \bar T$ is a union of a subset of $V_0T$, radius-$1$ balls around vertices in $V_1T$, and points of $\bdry T$,
	then there is a vertex $e \in V_1T \setminus \Pi(C)$ so that $\Pi(C)$ meets more than one component of $\bar{T} \setminus \{e\}$.
	Since $C \cap \Lambda_e = \emptyset$, this means that $C$ meets at least two components of $X \setminus \Lambda_e$, and so $C$ is not connected.
\end{proof}

\begin{proof}[Proof of Proposition~\ref{prop:tame}]
	First, find an arc, that is, an embedded path $\gamma_1:[0,1]\to X$ in the image of $\gamma$ with endpoints $\diam(\gamma)$ apart.
	
	Let $w_0' \in V_0T$ be the closest point to $v_0$ in $\Pi(\gamma_1) \subset \bar T$, following the notation of Lemma~\ref{lem:Tbar}. 
	We call $w\in V_0T$ a \emph{child} of $w_0'$ if $d_T(v_0,w)=d_T(v_0,w_0')+2$.

	If $\gamma_1$ meets $\Lambda_{w_0'}$ in exactly one or two points, those point(s) lie in some $\Lambda_{e} \subset \Lambda_{w_0'}$ for some $e\in V_1T$ adjacent to $w_0'$ with $d_T(v_0,e)=d_T(v_0,w_0')+1$.
	The points of $\Lambda_e$ split $\gamma_1$ into two or three subarcs each living in some $U_{w\to}$ for some child $w$ of $w_0'$.  Necessarily, at least one of these subarcs has endpoints $\delta/3$ apart.
	Let $\gamma_2$ be such a subarc of $\gamma_1$, and let $w_0\in V_0T$ be the child of $w_0'$ with $\gamma_2\subset U_{w_0\to}$.

	If $\gamma_1$ meets $\Lambda_{w_0'}$ in more than two points, let $\gamma_2=\gamma_1$ and let $w_0=w_0'$.
	In either case, $\gamma_2$ meets $\Lambda_{w_0}$ in more than two points, has endpoints at least $\delta/3$ apart, and $\gamma_2$ lies in $U_{w_0\to}$.

	If there is a path in $\bar T \setminus \{w_0\}$ from $\Pi(\gamma_2(0))$ to $\Pi(\gamma_2(1))$, then as $\gamma_2$ is an arc, there is $e\in V_1T$ and two (possibly equal) children $w',w''$ of $w_0$ so that $\gamma_2$ consists of an initial subarc in $U_{w'\to}$ that joins $\gamma_2(0)$ to a point of $\Lambda_e$, a subarc $\hat{\gamma}$ in $U_{w_0\to}$ joining the endpoints of $\Lambda_e$, and a final subarc from the other point of $\Lambda_e$ to $\gamma_2(1)$ in $U_{w''\to}$.  
	Therefore by Lemma~\ref{lem:reldist4} 
\[ 
	\frac{\delta}{3} \leq d(\gamma_2(0),\gamma_2(1)) \leq \diam U_{w'\to} + \diam \Lambda_e + \diam U_{w''\to} \preceq_{2K_4+1} \diam \Lambda_e, 
\]
so $\hat{\gamma}$ satisfies our desired property.
	
	So we now assume that $w_0$ disconnects $\Pi(\gamma_2(0))$ from $\Pi(\gamma_2(1))$.  This includes the case that $w_0$ is in one of these sets; if it is in both, $\gamma_2$ already has our property.

	Let $t_{0}$ (resp.\ $t_1$) be the first (resp.\ last) time $\gamma_2$ meets $\Lambda_{w_0}$.  If $t_{0}>0$, the subarc $\gamma_2|_{[0,t_{0}]}$ lives in $U_{w_{-1}\to}$ for some child $w_{-1}$ of $w_0$.  Let $t_{-1}$ be the first time $\gamma_2$ meets $\Lambda_{w_{-1}}$, and if $t_{-1}>0$ let $w_{-2}$ be the child of $w_{-1}$ with $\gamma_2|_{[0,t_{-1}]}\subset U_{w_{-2}\to}$.
	Similarly, if $t_1<1$, let $w_1$ be the child of $w_0$ with $\gamma_2|_{[t_1,1]} \subset U_{w_1\to}$, let $t_2$ be the last time $\gamma_2$ meets $\Lambda_{w_1}$, and if $t_2<1$ let $w_2$ be the child of $w_1$ with $\gamma_2|_{[t_2,1]}\subset U_{w_2\to}$.

	We claim that we can take $\delta'=\delta/100K_1K_4$ and our desired arc $\hat{\gamma}$ to be $\gamma_2|_{[t_i,t_{i+1}]}$ for $i=-1, 0$ or $1$.
	Note these subarcs if defined have endpoints in $\Lambda_{w_i}$ and live in $U_{w_i\to}$ for $i=-1,0,1$ respectively.  
	For $i=-1,0,1$ let $\epsilon_i = d(\gamma_2(t_i),\gamma_2(t_{i+1}))$, when defined.  Certainly if $\epsilon_0>\delta'$ we can take $\hat{\gamma}=\gamma_2|_{[t_0,t_1]}$, so assume $\epsilon_0 \leq \delta'$.

	If $w_1$ is defined, either $\epsilon_1>\delta'$ and we are done, or $\epsilon_1\leq \delta'$.
	If $w_2$ is defined then the tail $\gamma_2|_{[t_2,1]}$ has diameter
	\[ 
		\leq \diam U_{w_2\to} 
		\leq K_4 \diam \Lambda_{e_{w_2}}
		\leq K_4 K_1 d(\Lambda_{e_{w_1}},\Lambda_{e_{w_2}})
		\leq K_4 K_1 \epsilon_1 
		\leq \delta/100
	\]
	by Lemmas~\ref{lem:reldist1}~and~\ref{lem:reldist4}.
	Likewise, if $w_{-1}$ is defined, either $\epsilon_{-1}>\delta'$ and we are done, or $\epsilon_{-1}\leq \delta'$, and then if $w_{-2}$ is defined the subarc $\gamma_2|_{[0,t_{-1}]}$ has diameter $\leq \delta/100$.  
	If we reach this point, then whether or not $w_i$ exist for $i=-2,-1,1,2$, we deduce by the triangle inequality that the endpoints of $\gamma_2$ have distance
 	$ \leq \delta/100 + \epsilon_{-1} +\epsilon_0+ \epsilon_1 + \delta/100  < \delta/10$,
 	a contradiction.

 	So for $\delta'=\delta/100K_1K_4$, we have found a subarc $\hat{\gamma}$ which has endpoints in some $\Lambda_v, v \in T_{\delta'}$, which are $\delta'$-separated, and $\hat{\gamma} \subset U_{v\to}$.  
\end{proof}

We will use the following observation about the relative positions of cut pairs.
\begin{lemma}\label{lem:relative-child-size}
	There exists $C$ so that if $v,w \in V_0T$ with $v\in [v_0,w]$, $v\neq v_0$, $d_T(v,w)=2$ and $\Lambda_{e_w}=\{p_+,p_-\} \subset \Lambda_v$, then 
	\[
		\frac{1}{C}
		\leq \frac{d(h_v p_+, h_v \Lambda_{e_v})}{d(h_v p_-, h_v \Lambda_{e_v})}
		\leq C.
	\] 
\end{lemma}
\begin{proof}
	By symmetry it suffices to prove that $d(h_v p_+, h_v \Lambda_{e_v}) \preceq d(h_v p_-, h_v \Lambda_{e_v})$.
	
	Choose (not necessarily distinct) $q_-, q_+ \in \Lambda_{e_v}$
	so that we have
	$d(h_v p_+, h_v \Lambda_{e_v}) = d(h_v p_+, h_v q_+)$
	and
	$d(h_v p_-, h_v \Lambda_{e_v}) = d(h_v p_-, h_v q_-)$.

	By Lemmas~\ref{lem:reldist5}~and~\ref{lem:reldist4} 
	$\diam \Lambda_{e_w} \preceq \diam \Lambda_{e_v}$
	so by Lemma~\ref{lem:reldist1}
	\begin{equation}\label{eq:relchild1}
		d(p_-,p_+) = \diam \Lambda_{e_w} \preceq d(\Lambda_{e_w},\Lambda_{e_v}) .
	\end{equation}
	In particular, $d(p_-, p_+) \preceq d(p_-, q_-)$, and so by quasisymmetry $d(h_vp_-, h_vp_+)\preceq d(h_vp_-,h_vq_-)$.

	Thus
	\begin{align}
		d(h_vp_+,h_v\Lambda_{e_v})
		& = d(h_vp_+, h_vq_+)
		\leq d(h_vp_+, h_vp_-) + d(h_vp_-,h_vq_+) \notag
		\\ & \preceq d(h_vp_-, h_vq_-) + d(h_vp_-, h_vq_+).
		\label{eq:relchild2}
	\end{align}
	If $q_-=q_+$ we are done as $d(h_vp_-,h_vq_-)=d(h_vp_-,h_v\Lambda_{e_v})$.

	Suppose $q_-\neq q_+$.
	Since $d(h_vp_+,h_vq_+) \leq d(h_vp_+,h_vq_-)$, by the quasisymmetry of $h_v^{-1}$, $d(p_+,q_+)\preceq d(p_+,q_-)$.
	Combining this with  \eqref{eq:relchild1},
	\begin{align*}
		\diam \Lambda_{e_v} = d(q_+,q_-)
		& \leq d(q_+,p_+)+d(p_+,q_-)
		\preceq d(p_+,q_-)
		\\ & \leq d(p_+,p_-)+d(p_-,q_-)
		\preceq d(\Lambda_{e_w}, \Lambda_{e_v})+d(p_-,q_-)
		\\ & \leq 2d(p_-,q_-),
	\end{align*}
	therefore $d(h_vq_+,h_vq_-)\preceq d(h_vp_-,h_vq_-)$.
	By Lemma~\ref{lem:reldist4}, $d(p_-,q_+)\leq \diam \Lambda_v \preceq d(q_+,q_-)$, so $d(h_vp_-,h_vq_+) \preceq d(h_vq_+,h_vq_-)$.
	Therefore
	\[
		d(h_vp_-,h_vq_+)\preceq d(h_vq_-,h_vq_+) \preceq d(h_vp_-,h_vq_-),
	\]
	and applying this to \eqref{eq:relchild2} we are done.
\end{proof}

\begin{proposition}\label{prop:admiss}
  There are choices of parameters $E_1,E_3$ so that there exists $J>0$ so that for all $v \in T$, and any arc $\beta$ joining $\Lambda_{e_v}$ in $U_{v\to}$, we have
	\begin{equation}\label{eq:admiss1}
		\sum_{A\in \cS_n: A\cap\beta\neq\emptyset} \prod_{w \in T_0(v)} \rho_w^n(A) \geq J a^n \diam \Lambda_{e_v},
	\end{equation}
	where we take $\diam\Lambda_{e_{v_0}}:=1$.
	Moreover, if $\beta$ is an  arc in some $U_{v\to}$, $v \in T_{\delta'}$ with endpoints in $\Lambda_v$ that are $\delta'$-separated, then
	\begin{equation}\label{eq:admiss2}
		\ell_{\rho_n} (\beta) \geq 1.
	\end{equation} 
\end{proposition}
\begin{proof}
	We prove that \eqref{eq:admiss1} holds in stages.
	Before we begin, we summarise the dependence of constants chosen in the proof.  
	All constants, in particular $C_1,\ldots,C_4 \geq 1$, depend on the data of our space and the constants $K_1, \ldots, K_7 \geq 1$.
	We choose $k_0, k_1 \in \N$ with $k_0 := \lceil \log_a (2a\lambda (K_4\vee \frac{1}{\delta'})^{1/\tau})\rceil$ and $k_1 := \lceil \log_a(6K_2 K_3) \rceil$.
	We choose $j_0 \in \N$ based on Lemma~\ref{lem:relative-child-size}.
	We introduce a parameter $E_4$ which is chosen large enough depending on $j_0, k_0, k_1$, and set $J:=1/E_4$.
	We find a constant $C_1^*=C_1^*(J)$.
	The parameter $E_3$ is chosen large enough depending on $j_0, k_0, C_1^*$ (and $C_1,C_4$).
	Finally we find $C_2^*=C_2^*(\delta')$ and set $E_1 := 1/(C_2^* J\delta') = E_4/C_2^*\delta'$.

	\smallskip
	{\noindent \emph{Step 1}:}
	Suppose $v$ is a vertex with $\diam \Lambda_v \leq E_4a^{-n}$ for a choice of $E_4 \geq E_2$ below.  (The important case is when $v$ is the child of some $\hat{v}$ with $\diam \Lambda_{\hat{v}} \geq E_4a^{-n}$.)
	Thus $\diam \Lambda_{e_v} \leq E_4a^{-n}$, so $a^n \diam \Lambda_{e_v} \leq E_4$.

In the left hand side of \eqref{eq:admiss1}, for any $A$ meeting $\Lambda_{e_v}$ (as $\beta$ does), we have that $d(A,U_{\gets v})=0$ so $\rho_v^n(A)=1$.
Also, for all $w \in T_0(v) \setminus \{v\}$, $\pi(A)$ belongs to $[v_0,v]$ so $w \notin [v_0,\pi(A)]$, thus $W_{w,A}=\emptyset$ and $\rho_w^n(A)=1$ also.  Therefore, the left-hand side is $\geq 1$, and so \eqref{eq:admiss1} holds for $J=1/E_4$.

	\smallskip
	{\noindent \emph{Step 2:}} Suppose $v$ has $\diam \Lambda_v \in [E_4a^{-n},\delta')$ and all children of $v$ satisfy \eqref{eq:admiss1} with $J$. Note that $v\neq v_0$. 

	The idea is that by requiring $E_3$ large enough, $J$ doesn't get worse in our estimate for \eqref{eq:admiss1}.

	If $A \in \cS_n$ meets $\beta \subset U_{v\to}$ and has $W_{v,A}=\emptyset$, then as $v \notin [v_0,\pi(A)]$, $A$ must also meet $U_{\gets v}$.
	If $d(A, \Lambda_{e_v}) \geq 4 a^{-n} \geq 2 \diam A$ then for any $p \in A$, and using $K_3 \geq 1$ from Lemma~\ref{lem:reldist3},
	\[
	 d(A, \Lambda_{e_v}) \asymp_2 d(p, \Lambda_{e_v}) \asymp_{K_3} d(p, U_{\gets v}) \leq \diam A \leq 2a^{-n}.
	\]
	Thus we conclude that 
	\begin{equation}\label{eq:admiss3}
		d(A, \Lambda_{e_v}) \leq 4K_3 a^{-n}.
	\end{equation}

	The path $\beta$ joins the endpoints of $\Lambda_{e_v}$, travelling through $\Lambda_v$ with subarcs passing through $U_{w\to}$ for various children $w$ of $v$.

	Let $k_0, k_1$ be constants chosen as above, then add the condition $m_v>k_0+k_1$ to $E_4$, so that $0 \leq k_0 \leq m_v-k_1 \leq m_v$.

	For $k \in \{k_0, \ldots, m_v-k_1 \}$, consider the points of $X_v=h_v \Lambda_v$ at distance $(a^{-k}, a^{-k+1}]$ from $h_v \Lambda_{e_v}$, and call this set $Y_k$.
	By Lemma~\ref{lem:relative-child-size}, there exists $j_0$ so that if a child $w$ of $v$ has $\Lambda_{e_w} = \{p_+,p_-\}$, and $p_+ \in Y_k, p_- \in Y_l$, then $|k-l|\leq j_0$.

	The pair $\Lambda_{e_v}$ has $\diam \Lambda_{e_v} \geq \frac{1}{K_4}\diam \Lambda_v$ by Lemma~\ref{lem:reldist4}, so $\diam h_v\Lambda_{e_v} \geq \lambda^{-1} K_4^{-1/\tau}$ by \eqref{eq:taulambdaHolder}.
	Note that $a^{-k_0+1} \leq \frac{1}{2} \lambda^{-1}K_4^{-1/\tau}$ by the choice of $k_0$, so each $Y_k, k\geq k_0,$ consists of two disjoint balls centred on $h_v\Lambda_{e_v}$.

	Let $M := \lfloor (m_v-k_0-k_1-2j_0-2)/(2j_0+1) \rfloor$; add the condition that $E_4$ is large enough so that $M \geq 1$.

	For each $i \in 0, \ldots, M$, consider \[ k=k(i):=(2j_0+1)i+j_0+1+k_0 \in\{k_0+(j_0+1),\ldots,m_v-k_1-(j_0+1)\}, \] 
	and the collection $\cB_i$ of subarcs of $\beta$ that are either (i) in $\Lambda_v$ with $h_v$-image in $Y_k$, or (ii) join $\Lambda_{e_w}$ in some $U_{w\to}$ with $w$ a child of $v$ and $h_v\Lambda_{e_w}$ meets or jumps over $Y_k$.
	Here ``jumps over'' means that one end point lies in $\bigcup_{s > k} Y_s$ and the other in $\bigcup_{s<k} Y_s$.
	Note that by the choices of $M$ and $k=k(i)$, the collections $\cB_i$ are disjoint, and if $h_v\Lambda_{e_w}$ jumps over $Y_k$, both points of $h_v\Lambda_{e_w}$ are in $B(z,a^{-k_0+1})$ for the same $z \in h_v\Lambda_{e_v}$.

	\begin{sublemma}
		For $k_1 \geq \log_a(6K_2K_3)$, for any $i =0,\ldots, M$ and any $A \in \cS_n$ which intersects some arc in $\cB_i$, we have that $\pi(A) \in T_0(v)$.
	\end{sublemma}
	\begin{proof}
		\emph{Case 1, $A\cap\Lambda_v\neq\emptyset$:}
		By the definitions of $\cB_i$, \eqref{eq:defmv} and $k$,
	\begin{align*}
		d(h_v(A \cap \Lambda_v), h_v \Lambda_{e_v}) & \geq a^{-k}-\diam h_v(A\cap\Lambda_v)
		\\ &	\geq a^{-k} - \tfrac{1}{2}a^{-m_v} \geq a^{-m_v}\left(a^{k_1+j_0} -\tfrac{1}{2}\right).
	\end{align*}
	Thus by \eqref{eq:taulambdaHolder} and \eqref{eq:defmv},
	\begin{align*}
		d(A, \Lambda_{e_v}) & \geq d(A\cap \Lambda_v, \Lambda_{e_v}) - \diam A
		\\ & \geq D_v \left( \frac{a^{-m_v}(a^{k_1+j_0}-\tfrac{1}{2})}{\lambda} \right)^{1/\tau} - 2a^{-n}
		\\ & \geq a^{-n} \left(2a^{k_1+j_0}-1\right)^{1/\tau} -2a^{-n}
		> 4K_3 a^{-n}
	\end{align*}
		because $k_1+j_0\geq k_1 \geq \log_a(6K_2K_3) \geq \log_a(6K_3)$.
	Thus by \eqref{eq:admiss3} we have $W_{v,A} \neq \emptyset$ and as $A\cap\Lambda_v\neq\emptyset$ we have $v=\pi(a)$.
		(Note that therefore $W_{v,A} = B_A$ for a ball $B_A$ of radius $a^{-n}$ satisfying $A \subset (K_7+1)B_A$, so if a collection of $A$'s covers $\beta'$, then the corresponding collection of $(K_7+1)B_A$'s also covers $\beta'$.)
	
		\emph{Case 2: For some child $w$ of $v$, $A$ intersects a subarc $\beta'\subset U_{w\to}$ that meets or jumps over $Y_k$.}
	The endpoints of $h_v\Lambda_{e_w}$ lie in $\bigcup_{s=k-j_0}^{k+j_0}Y_s$, and both endpoints have the same closest point in $h_v\Lambda_{e_v}$.
	It suffices to show that $W_{v,A}$ is either $B_A$ (as $\pi(A)=v$) or $\Lambda_{e_w}$ (as $\pi(A) \in T_0(w)$ but $\pi(A)\neq v$).
	This follows if we rule out $\pi(A) \in [v_0,v)$.
	
	Similarly to Case 1,
	\[ d(h_v \Lambda_{e_w}, h_v\Lambda_{e_v}) \geq a^{-k-j_0} \geq a^{-m_v}a^{k_1} \]
	so
	$d(\Lambda_{e_w}, \Lambda_{e_v}) \geq a^{-n} \left(2a^{k_1}\right)^{1/\tau}$.
 Since $A \cap \beta' \subset U_{w\to}$, Lemma~\ref{lem:reldist2} gives
	\begin{align*}
		d(A,\Lambda_{e_v}) 
		& \geq d(U_{w\to}, U_{\gets v}) - \diam A
		\geq \frac{1}{K_2} d(\Lambda_{e_w},\Lambda_{e_v}) - 2a^{-n}
		\\& \geq a^{-n} K_2^{-1} (2a^{k_1})^{1/\tau} - 2 a^{-n}
		> 4K_3 a^{-n},
	\end{align*}
		where the last inequality uses $k_1\geq \log_a(6K_2 K_3)$,
	and so by \eqref{eq:admiss3} $W_{v,A} \neq \emptyset$.
	\end{proof}

	A jump $\beta' \in \cB_i$ going through some $U_{w\to}$ is \emph{large} if there is some $A \in \cS_n$ with $A \cap \beta' \neq \emptyset$ and $\pi(A) \in T_0(w)$.
	In this case, by Lemma~\ref{lem:reldist4} and Lemma~\ref{lem:piA-properties},
	\begin{equation}\label{eq:large-jump-gap}
		\diam \Lambda_{e_w} \geq \frac{1}{K_4} \diam U_{w\to} \geq \frac{1}{K_4} \diam \Lambda_{\pi(A)} \geq \frac{1}{K_4 K_7} a^{-n}.
	\end{equation}

	Suppose $\{\beta_j\} \subset \cB_i$ are the large jumps in $\cB_i$, going through $U_{w_j\to}$ for $w_j$ children of $v$.
	Consider the sets $\cC_{i,j} := \{ A \in \cS_n : A \cap \beta_j \neq \emptyset\}$ for each $j$.
	If $A \in \cC_{i,j}\cap \cC_{i,j'}$ for $j\neq j'$ then $A$ intersects both $U_{w_j\to}$ and $U_{w_{j'}\to}\subset U_{\gets w_j}$ so by Lemma~\ref{lem:reldist3} $d(A, \Lambda_{e_{w_j}})\preceq a^{-n}$.
	Therefore by \eqref{eq:large-jump-gap}, Lemma~\ref{lem:reldist1} and the doubling of $X_v$ we have that there exists $C_1$ independent of $v, i, j$ so that any $A$ appears in at most $C_1$ of these sets.
	
	For each choice of $i$ (which fixes $k$), we either have `many' or `few' large jumps.

	\smallskip
	{\noindent \emph{Case of many large jumps:}}
	Suppose $\sum_{j} \diam h_v(\Lambda_{e_{w_j}}) \geq \frac{1}{2}a^{-k}$.
	Consider a given (large) jump $h_v \Lambda_{e_{w_j}}$.
	There is a constant $C_2$ so that there are at most $C_2$ many $A \in \cS_n$ with $A\cap \beta_j\neq \emptyset$ and $\pi(A)=v$;
		for such $A$, $\prod_{u\in T_0(w_j)} \rho_u^n(A)=1$.
	By the Step 2 hypothesis
	$ \sum_{A \in \cS_n: A\cap\beta\cap U_{w_j\to} \neq\emptyset} \prod_{u\in T_0(w_j)} \rho_u^n(A) \geq Ja^n \diam \Lambda_{e_{w_j}}$.
	So if $\diam \Lambda_{e_{w_j}} \geq (2C_2/J) a^{-n}$ then we have
	\[
		\sum_ {\substack{A \in \cS_n : \pi(A)\neq v,\\ A \cap \beta \cap U_{w_j\to} \neq \emptyset} }
		\prod_{u\in T_0(w_j)} \rho_u^n(A)
		\geq \frac{1}{2} Ja^n \diam \Lambda_{e_{w_j}}.
	\] 
	Moreover, when $\pi(A) \neq v$ in this sum $W_{v,A} = \Lambda_{e_{w_j}}$ so
	\begin{align*}
	\sum_{A \in \cS_n : A \cap \beta \cap U_{w_j\to} \neq \emptyset}
	\prod_{u \in T_0(v)} \rho^n_u(A)
	& =
	\sum_{A \in \cS_n : A \cap \beta \cap U_{w_j\to} \neq \emptyset}
	\rho^n_v(A)
	\prod_{u \in T_0({w_j})} \rho^n_u(A)
	\\ & \geq 
		\sum_{\substack{A \in \cS_n : \pi(A)\neq v,\\ A \cap \beta \cap U_{w_j\to} \neq \emptyset} }
	\rho^n_v(A)
	\prod_{u \in T_0({w_j})} \rho^n_u(A)
	\\ & \geq
	\frac{E_3 D_v \diam h_v(\Lambda_{e_{w_j}})}{m_v a^{-k+j_0} \diam \Lambda_{e_{w_j}}}
		\sum_{\substack{A \in \cS_n : \pi(A)\neq v,\\ A \cap \beta \cap U_{w_j\to} \neq \emptyset}} 
	\prod_{u \in T_0({w_j})} \rho^n_u(A)
	\\ & \geq 
	\frac{E_3 D_v}{m_v a^{-k+j_0}} \frac12 J a^n \diam h_v(\Lambda_{e_{w_j}}).
	\end{align*}
	If $\diam \Lambda_{e_{w_j}} < (2C_2/J) a^{-n}$ holds, then $\diam \Lambda_{e_{w_j}} \asymp_{C(J)} a^{-n}$.
	So for any $A \in \cS_n$ with $A \cap \beta \cap U_{w_j\to} \neq \emptyset$ and $\pi(A) = v$ (and so $W_{v,A}=B_A$ with $\diam B_A \asymp a^{-n}$), by the uniform quasisymmetry of $h_v$ we have
	\[
	\frac{\diam h_v\Lambda_{e_{w_j}}}{\diam h_v B_A} \asymp 1 \asymp \frac{\diam \Lambda_{e_{w_j}}}{\diam B_A},
	\]
	thus, at the cost of a constant $C_1^*=C_1^*(J)$, we can replace $\diam h_v B_A \, / \diam B_A$ in the relevant $\rho_v^n(A)$ by $\diam h_v \Lambda_{e_{w_j}} / \diam \Lambda_{e_{w_j}}$, and so induction again gives
	\begin{align*}
		\sum_{\substack{A \in \cS_n :\\ A \cap \beta \cap U_{w_j\to} \neq \emptyset}}
	\prod_{u \in T_0(v)} \rho^n_u(A)
	& =
	\sum_{\substack{A \in \cS_n :\\ A \cap \beta \cap U_{w_j\to} \neq \emptyset}}
	\rho^n_v(A)
	\prod_{u \in T_0({w_j})} \rho^n_u(A)
	\\ & \geq
	\frac{E_3 D_v \diam h_v(\Lambda_{e_{w_j}})}{m_v a^{-k+j_0} C_1^* \diam \Lambda_{e_{w_j}}}
	\sum_{\substack{A \in \cS_n :\\ A \cap \beta \cap U_{w_j\to} \neq \emptyset}}
	\prod_{u \in T_0({w_j})} \rho^n_u(A)
	\\ & \geq 
	\frac{E_3 D_v}{C_1^*m_v a^{-k+j_0}} J a^n \diam h_v(\Lambda_{e_{w_j}}).
	\end{align*}
	Summing over all large jumps $\{\beta_j\} \subset \cB_i$ we have
	\begin{align*}
		\sum_{A\in \bigcup_j \cC_{i,j}}
		\prod_{u \in T_0(v)} \rho^n_u(A) 
		& \geq \frac{1}{C_1} \sum_{j} 
		\sum_{A\in \cC_{i,j}}
		\prod_{u \in T_0(v)} \rho^n_u(A)
		\\ & \geq \frac{1}{C_1} \sum_j
		\frac{E_3 D_v}{(C_1^* \vee 2)m_v a^{-k+j_0}} J a^n \diam h_v(\Lambda_{e_{w_j}})
		\\ & \geq 
		\frac{E_3 D_v}{C_1(C_1^* \vee 2)m_v a^{-k+j_0}} J a^n \frac{a^{-k}}{2}
		 = \frac{E_3 D_v}{2C_1(C_1^* \vee 2)m_v a^{j_0}} J a^n .
	\end{align*}

	\smallskip
	{\noindent \emph{Case of few large jumps:}}
	We have that $\sum_j \diam h_v(\Lambda_{e_{w_j}}) < \frac{1}{2}a^{-k}$,
	but the arcs and jumps in $\cB_i$ must total at least $a^{-k+1}-a^{-k}\geq a^{-k}$ in diameter.
	Thus the arcs and small jumps (that is jumps $\beta'$ where any $A \in \cS_n$ with $A \cap \beta' \neq \emptyset$ has $\pi(A)=v$) must have total diameter in $h_v \Lambda_v$ at least $\frac{1}{2} a^{-k}$.

	Suppose we have a small jump $\beta'$ through some $U_{w'\to}$, $w'$ a child of $v$.
	The $A$ which cover $\beta'$ each have a ball $B_A$ centred at a point of $\Lambda_v$ at most $K_7a^{-n}$ from the centre of $A$, 
	so $\beta'$ is in the $(K_7+1)a^{-n}$ neighbourhood of $\Lambda_v$. Therefore by Lemma~\ref{lem:reldist3} we have $\beta'$ is in the $K_3 (K_7+1) a^{-n}$ neighbourhood of $\Lambda_{e_{w'}}$, so $\diam \Lambda_{e_{w'}} \leq 2 K_3 (K_7+1) a^{-n}$, and thus by Lemma~\ref{lem:reldist4} $\diam U_{w'\to} \leq 2K_4K_3 (K_7+1)a^{-n}$.
	So there exists $C_3 \geq (K_7+1)$ so that if $A$ meets the small jump $\beta'$, then $C_3B_A$ covers the entire small jump including its endpoints.  Here for a ball $B=B(x,r)$ and $C>0$, we set $CB:=B(x,Cr)$.

	In the image then, if there is a sequence of arcs and small jumps connecting points at least some distance $L$ apart, then the sum of $\diam h_v(C_3B_A)$ for those $A$ covering the corresponding arcs in $\cB_i$ must total at least $L$.
	So as we do not have many large jumps, 
	we must have that $\sum_{A \in \cS_n(\cB_i, v) } \diam h_v(C_3B_A) \geq \frac{1}{2}a^{-k}$, where $\cS_n(\cB_i,v)$ is defined to be the set of all $A \in \cS_n$ so that $\pi(A)=v$ and $A \cap \beta' \neq \emptyset$ for some $\beta' \in \cB_i$.

	By uniform quasisymmetry $\diam h_v(C_3B_A) \asymp_{C_4} \diam h_v(B_A)$.
	Putting it together,
	\begin{align*}
		\sum_{A \in \cS_n(\cB_i,v)}
		\prod_{u \in T_0(v)} \rho_u^n(A)
		& =
		\sum_{A \in \cS_n(\cB_i,v)}
		\rho_v^n(A)
		 \geq  
		\sum_{A \in \cS_n(\cB_i,v)} 
		\frac{E_3 D_v \diam h_v(B_A)}{m_v a^{-k+j_0} \diam B_A} 
		\\ & \geq 
		\frac{1}{C_4} 
		\sum_{A \in \cS_n(\cB_i,v)} 
		\frac{E_3 D_v \diam h_v(C_3B_A)}{m_v a^{-k+j_0} 2a^{-n}} 
		\\ & \geq
		\frac{E_3 D_v a^{-k}}{4C_4 m_v a^{-k+j_0} a^{-n}} 
		= \frac{E_3 D_v }{4C_4 m_v a^{j_0}} a^n.
	\end{align*}

	As for each $i$ there are either many large jumps or not, we have:
	\begin{align*}
		\sum_{A\in \cS_n: A\cap\beta\neq\emptyset} \prod_{w \in T_0(v)} \rho_w^n(A) 
		& \geq \sum_{i=0}^{M} 
		\frac{E_3 D_v}{  m_v a^{j_0} } \left( \frac{J}{2C_1(C_1^*\vee 2)} \wedge \frac{1}{4C_4} \right) a^n
		\\ & 
		= \frac{E_3 D_v (M+1)}{  m_v a^{j_0} } \left( \frac{J}{2C_1(C_1^*\vee 2)} \wedge \frac{1}{4C_4} \right) a^n
		\\ & \geq J \left(\frac{(M+1)}{4a^{j_0} (C_1(C_1^*\vee 2) \vee C_4) m_v} \right)E_3  a^n D_v.
	\end{align*}
	By our earlier conditions on $E_4$, we have that $M = \lfloor (m_v -k_0-k_1-2j_0-2)/(2j_0+1)\rfloor$ satisfies $M \geq 1$.  Thus $(M+1)/m_v$ is bounded away from zero, so we can and do require that $E_3$ is large enough depending on $j_0,C_1,C_1^*,C_4$ so that the term in parentheses is at least $1/E_3$ 
	thus \eqref{eq:admiss1} holds for $v$ with the same $J$.
	Note that $D_v = \diam \Lambda_v \geq \diam \Lambda_{e_v}$.

	\smallskip
	{\noindent \emph{Step 3:}} Suppose $v$ has $\diam \Lambda_v \geq \delta'$ and all children satisfy \eqref{eq:admiss1} for some $J$.

	The argument is identical to that of step 2 until we apply the definition of $f_v$ from \eqref{eq:deffv}, and as we don't resize the annuli it suffices to consider the largest annulus $Y_k$ with $i=0, k=j_0+1+k_0$.
	We now indicate the slight differences.
	
	{\noindent\emph{Case of many large jumps $\{\beta_j\} \subset \cB_0$}:}
	If a given large jump $\beta_j$ has $\diam \Lambda_{e_{w_j}} \geq (2C_2/J) a^{-n}$ then
	\begin{align*}
		\sum_{\substack{A \in \cS_n :\\ A \cap \beta \cap U_{w_j\to} \neq \emptyset} }
	\prod_{u \in T_0(v)} \rho^n_u(A)
	& \geq 
		\sum_{\substack{ A \in \cS_n : \pi(A)\neq v \\ A \cap \beta \cap U_{w_j\to} \neq \emptyset}}
	\rho^n_v(A)
	\prod_{u \in T_0({w_j})} \rho^n_u(A)
	\\ & \geq
	\frac{E_3 D_v \diam h_v(\Lambda_{e_{w_j}})}{1 \cdot \diam \Lambda_{e_{w_j}}}
	\sum_{\substack{ A \in \cS_n : \pi(A)\neq v \\ A \cap \beta \cap U_{w_j\to} \neq \emptyset}}
	\prod_{u \in T_0({w_j})} \rho^n_u(A)
	\\ & \geq 
	\frac{E_3 D_v}{1} \frac12 J a^n \diam h_v(\Lambda_{e_{w_j}}).
	\end{align*}
	While if $\diam \Lambda_{e_{w_j}} < (2C_2/J) a^{-n}$ then, for the same $C_1^*=C_1^*(J)$ as before
	\begin{align*}
		\sum_{\substack{A \in \cS_n :\\ A \cap \beta \cap U_{w_j\to} \neq \emptyset}}
	\prod_{u \in T_0(v)} \rho^n_u(A)
	 & \geq
	\frac{E_3 D_v \diam h_v(\Lambda_{e_{w_j}})}{1 \cdot C_1^* \diam \Lambda_{e_{w_j}}}
		\sum_{\substack{A \in \cS_n :\\ A \cap \beta \cap U_{w_j\to} \neq \emptyset}}
	\prod_{u \in T_0({w_j})} \rho^n_u(A)
	\\ & \geq 
	\frac{E_3 D_v}{C_1^*} J a^n \diam h_v(\Lambda_{e_{w_j}}).
	\end{align*}
	Summing over all large jumps $\{\beta_j\} \subset \cB_0$, as $k=j_0+1+k_0$ we have
	\begin{align*}
		\sum_{A\in \bigcup_j \cC_{i,j}}
		\prod_{u \in T_0(v)} \rho^n_u(A) 
		& \geq \frac{1}{C_1} \sum_j
		\frac{E_3 D_v}{(C_1^* \vee 2)\cdot 1} J a^n \diam h_v(\Lambda_{e_{w_j}})
		\\ & \geq 
		\frac{E_3 D_v}{C_1(C_1^* \vee 2) \cdot 1} J a^n \frac{a^{-(j_0+1+k_0)}}{2} .
	\end{align*}

	{\noindent\emph{Case of few large jumps $\{\beta_j\} \subset \cB_0$}:}
	The argument is the same, giving the bound
	\begin{align*}
		\sum_{A \in \cS_n(\cB_0,v)}
		\prod_{u \in T_0(v)} \rho_u^n(A)
		 & \geq  
		\sum_{A \in \cS_n(\cB_0,v)} 
		\frac{E_3 D_v \diam h_v(B_A)}{1 \cdot \diam B_A} 
		\\ & \geq 
		\frac{1}{C_4} 
		\sum_{A \in \cS_n(\cB_0,v)} 
		\frac{E_3 D_v \diam h_v(C_3B_A)}{1 \cdot \diam B_A} 
		\\ & \geq
		\frac{E_3 D_v a^{-(j_0+1+k_0)}}{2C_4 \diam B_A} 
		\geq \frac{E_3 D_v a^{-(j_0+1+k_0)}}{4C_4} a^n.
	\end{align*}
	
	As $\cB_0$ either has many large jumps or not, we have:
	\begin{align*}
		\sum_{A\in \cS_n: A\cap\beta\neq\emptyset} \prod_{w \in T_0(v)} \rho_w^n(A) 
		& \geq 
		E_3 D_v a^{-(j_0+1+k_0)} \left( \frac{J}{2C_1(C_1^*\vee 2)} \wedge \frac{1}{4C_4}\right) a^n
		\\ & \geq J \left(\frac{a^{-(j_0+1+k_0)}}{4 (C_1(C_1^*\vee 2) \vee C_4) } \right)E_3  a^n D_v.
	\end{align*}
	Since $D_v=\diam \Lambda_v \geq \diam \Lambda_{e_v}$, provided $E_3$ is required to be large enough depending only on $j_0,k_0,C_1,C_1^*,C_4$ we get \eqref{eq:admiss1} for $v$ with the same value of $J$.  As this point we fix the value of $E_3$.

	\smallskip
	{\noindent \emph{Conclusion:}}
	So we have shown that \eqref{eq:admiss1} for all $v \in T$.

	Suppose now that the curve $\beta$ has endpoints in $\Lambda_v$, $v \in T_{\delta'}$, that are $\delta'$ separated, but not necessarily agreeing with $\Lambda_{e_v}$.
	By the bi-H\"older estimates on $h_v$ for such $v$, this implies that the distance between the $h_v$-images of the endpoints of $\beta$ is $\geq \lambda^{-1} (\delta')^{1/\tau} >0$.
	Notice that in step 3 the location (if defined) of $\Lambda_{e_v}$ was not relevant, only that $\beta$ crossed an annulus of width proportional to $D_v$.  So as we already required $k_0$ to satisfy $a^{-k_0+1} \leq \frac12 \lambda^{-1} (\delta')^{1/\tau}$,   we can set $i=0$, $k=j_0+1+k_0$ and let $Y_k$ be all points in $h_v\Lambda_v$ at distance in $(a^{-k},a^{-k+1}]$ from $\{h_v\beta(0),h_v\beta(1)\}$; note that the balls of radius $a^{-k+1}$ centred on $h_v\Lambda_v$ are disjoint. Step 3 then gives, with the same choice of $E_3$, 
	\[
		\sum_{A\in\cS_n:A\cap\beta\neq\emptyset}\prod_{w\in T_0(v)}\rho_w^n(A) \geq Ja^nD_v.
	\]

	It remains to bound $\ell_{\rho_n}(\beta)$.
	Since $\beta$ is in $U_{v\to}$, we have for any $A\in\cS_n$ with $A\cap\beta\neq\emptyset$ that $\pi(A)\in [v_0,v]$.
	This means that if $w\in V_0T$ has $\rho_w^n(A)\neq 1$ then $w \in T_0(v)$ or $w \in [v_0,v]$.  Write $\{v_0,v_1,\ldots,v_m=v\} \subset V_0T$ for the vertices of $[v_0,v]$; these are in $T_{\delta'}$.
	For $i=0,\ldots,m-1$, we have $\diam W_{v_i,A} \asymp_{\delta'} 1$, so $\diam h_{v_i}W_{v_i,A} \asymp 1$, and $D_{v_i}\asymp $ so there is some constant $C'$ depending on $\delta'$ so that $\rho_{v_i}^n(A) \asymp_{C'} E_3$.
	Thus, using the trivial bound $m\leq|T_{\delta'}|$ and setting $C_2^*:= (C'E_3)^{|T_{\delta'}|}$,
	\begin{align*}
		\ell_{\rho_n} (\beta)
		& = E_1 a^{-n} \sum_{A\in S_n: A\cap\beta\neq\emptyset} \prod_{w \in V_0T} \rho_w^n(A)
		\\ & = E_1 a^{-n} \sum_{A\in S_n: A\cap\beta\neq\emptyset} \Bigg( \prod_{i=0}^{m-1} \rho_{v_i}^n(A) \Bigg) \Bigg( \prod_{w \in T_0(v)} \rho_w^n(A) \Bigg)
		\\ & \geq E_1 a^{-n} \sum_{A\in S_n: A\cap\beta\neq\emptyset} \left( C'E_3 \right)^{|T_{\delta'}|} \Bigg( \prod_{w \in T_0(v)} \rho_w^n(A) \Bigg)
		\\ & \geq E_1 a^{-n} C_2^* Ja^n D_v
		\geq E_1 C_2^* J \delta' \geq 1,
	\end{align*}
	thanks to our choice of $E_1 := 1/(C_2^* J \delta')$.
\end{proof}

Admissibility now follows.
\begin{proof}[Proof of Theorem~\ref{thm:admissible}]
	For $\delta < \delta_0$, let $\delta'>0$ be chosen by Proposition~\ref{prop:tame}, and then let $E_1, E_3$ be given by Proposition~\ref{prop:admiss}.
	By Proposition~\ref{prop:tame}(2), any curve $\gamma \in \Gamma_\delta$ has a subarc $\hat\gamma \in \Gamma_{\delta'}$, which Proposition~\ref{prop:admiss} shows has $\ell_{\rho_n}(\hat{\gamma}) \geq 1$, so $\ell_{\rho_n}( \gamma )\geq \ell_{\rho_n} (\hat\gamma) \geq 1$.
	Therefore $\rho_n$ is $\Gamma_\delta$-admissible.
\end{proof}

\section{Volume}\label{sec:volume}

It remains to bound the volume of the weights $\rho_n$.
For $v \in V_0T$, recall that $T_0(v)$ is the set of $v$ and its descendants in $V_0T$.
Let $\cS_n(v) := \pi^{-1}(T_0(v))$, and $\cS_n^0(v) := \pi^{-1}(v)$.
Note that for any $v \in T$ we have the partition
\begin{equation}\label{eq:def:Sn0}
	\cS_n(v)=\cS_n^0(v) \sqcup \bigsqcup_{w \text{ child of } v} \cS_n(w).
\end{equation}
We define for $v \in T$,
\begin{equation}\label{eq:def:Vnv}
	V_n(v) := E_1^p a^{-np} \sum_{A \in \cS_n(v)} \prod_{w \in T_0(v)} \rho_w^n(A)^p.
\end{equation}
Observe that the $p$-volume of $\rho_n$ is, by definition,
\begin{align}\label{eq:vol1}
  \Vol_p(\rho_n) =  \sum_{A \in \cS_n} \rho_n(A)^p 
	= E_1^p a^{-np} \sum_{A \in \cS_n} \prod_{v \in V_0T} \rho_v^n(A)^p
	= V_n(v_0) .
\end{align}

The goal of the section is the following, given fixed 
parameters $\delta', E_1, E_2, E_3$.
\begin{theorem}\label{thm:volbound}
	We have $V_n(v_0) =\Vol_p(\rho_n)$ bounded independently of $n$.
\end{theorem}
We are going to set up an induction to bound the quantities $V_n(v)$.
What is important here is the relative size of $D_v$ and $a^{-n}$: 
for a given $n$ and $v \in V_0T$, our cover of $U_{v\to}$ by balls of radius $a^{-n}$ 
is, if we rescale $U_{v\to}$ by $1/D_v$, like a cover of $\frac{1}{D_v} U_{v\to}$ by balls of radius $a^{-n}/D_v$,
and the corresponding $p$-volume is scaled by $1/D_v^p$.
Let $t_{v,n} = \lfloor n+\log_a D_v \rfloor$, then the balls in the rescaled cover are approximately of size $a^{-t_{v,n}}$. 

When $t \geq t_0$ for some fixed constant $t_0 \in \Z$ set below, and $t_{v,n}=t$ for some $n, v$, we are covering $U_{v\to}$ by sets significantly smaller than $U_{v\to}$, and estimating $V_n(v)$ is amenable to induction; the relevant quantity we try to bound is the following:
\begin{definition}\label{def:vhat}
For $t \in \Z$, if $t < t_0$ set $\hat{V}_{t} = 1$, otherwise let $\hat{V}_{t}$ be the supremum of $V_n(v)/D_v^p$ over all $n$, over all $v \in V_0T \setminus T_{\delta'}$ with $t_{v,n} = t$.  
\end{definition}

As an initial observation, for a given $t$ there is an easy uniform bound on $V_n(v)/D_v^p$ with $t_{v,n}=t$.
\begin{lemma}\label{lem:vhat-finite}
	For each $t \in \Z$ fixed, $\hat{V}_t < \infty$.
\end{lemma}
\begin{proof}
	The claim is trivial for $t < t_0$.
	We fix $t \geq t_0$.
	
	Suppose for a given $n$ that
	$v \in T\setminus T_{\delta'}$ has $t_{v,n}=t$,
	i.e. $\lfloor n + \log_a D_v \rfloor = t$,
	so $a^{-np}/D_v^p \asymp a^{-tp}$, and
	\begin{equation}
		\label{eq:vhat-fin1}
		\frac{V_n(v)}{D_v^p} \asymp a^{-tp} \sum_{A\in \cS_n}\prod_{w\in T_0(v)} \rho_w^n(A)^p .
	\end{equation}

	By definition, $\cS_n(v)$ consists of a collection of balls of radius $a^{-n}$ with centres separated by $a^{-n}$ in $U_{v\to}$ (not quite a cover since balls close to $\Lambda_{e_v}$ will not be included).
	As $X$ is Ahlfors regular, and by Lemma~\ref{lem:reldist4} $\diam U_{v\to} \preceq \diam \Lambda_v = D_v$, the number of such $A$ included is bounded above by a constant depending on $D_v/a^{-n} \asymp a^{t}$, i.e., depending only on $t$.

	For a given $w \in T_0(v)$, the only way that $\rho_w^n(A) \neq 1$ is if $d(A,U_{\gets w}) > E_2 a^{-n}$ and $w \in [v,\pi(A)]$.
	The first condition and Lemma~\ref{lem:reldist4} implies that $\diam \Lambda_w \succeq a^{-n}$.
	The second condition and Lemma~\ref{lem:reldist5} implies that
	\[
		a^{-t} \asymp \frac{a^{-n}}{D_v} \preceq \frac{\diam \Lambda_w}{\diam \Lambda_v} \preceq e^{-\epsilon d_T(v,w)},
	\]
	and so $d_T(v,w)$ is bounded by a constant depending on $t$.
	We have shown that the number of terms in the sum and in the product on the right-hand side of \eqref{eq:vhat-fin1} are both bounded by a constant depending on $t$, and not on $n$.

	It remains to show that $\rho_w^n(A)$ is bounded.
	This follows the argument in the proof of Theorem~\ref{thm:maxbound}.   For those $w$ that are in $T_0(w')$ for some $w'$ with $d(A,U_{\gets w'})\leq E_2a^{-n}$ or $m_{w'}\leq 1$, then by cases (1) and (2) of Step 3 of the proof of Theorem~\ref{thm:maxbound}, $\rho_{w}^n(A) \leq C(E_3)$.
	For all other $w$, we have $d(A,U_{\gets w}) > E_2a^{-n}$ and $m_w > 1$, and by \eqref{eq:max3} in Step 4 of the proof of Theorem~\ref{thm:maxbound}, we have $\diam(h_w W_{w,A})/f_w(W_{w,A}) \preceq 1/m_w$.  Moreover, by Lemma~\ref{lem:reldist5} $D_w \preceq D_v$ and $\diam W_{w,A} \succeq a^{-n}$, $D_w/\diam W_{w,A} \preceq D_v/a^{-n} \asymp a^{t}$.
	So 
	\[ 
		\rho_w^n(A) 
		= \frac{E_3 D_w}{\diam W_{w,A}} \cdot \frac{\diam h_w(W_{w,A})}{f_w(W_{w,A})} 
		\preceq a^t \cdot \frac{1}{m_w} \preceq a^t ,
	\]  
	which is a constant bounded in terms of $t$.

	So we conclude that $V_n(v)/D_v^p$ is bounded by a constant depending on $t$, independent of $n$.
\end{proof}

We are going to bound $V_n(v)/D_v^p$ in two stages: a general inductive step using $\hat{V}_t$ when $v \notin T_{\delta'}$, then the finitely-many vertices of $T_{\delta'}$ will be dealt with using a weaker bound.

\subsection{Volume bounds for $v \notin T_{\delta'}$}\label{ssec:vol-small-vertices}

The goal of this subsection is to bound $\hat{V}_t$ by induction on $t$, and thus to bound $V_n(v)$ for $v \notin T_{\delta'}$.  In fact, we show more.
\begin{proposition}\label{prop:volume-v-hat-goes-to-0}
	$\lim_{t\to\infty}\hat{V}_t = 0$.
\end{proposition}

There are two kinds of contribution to $V_n(v)$, those coming from $\Lambda_v$ and those coming from $U_{w\to}$ for some child $w$ of $v$.  Now for fixed $n$, a child $w$ of $v$ usually has $\Lambda_w$ smaller than $\Lambda_v$, and so $t_{w,n}<t_{v,n}$, and we can use an inductive bound to estimate the contribution of each child to $V_n(v)$; we have to treat carefully the finitely many children that are exceptions.

Note that if $\cS_n(w) \neq \emptyset$ for $w$ a child of $v$, then there exists $A \in \cS_n(w)$ so $\pi(A)$ is a descendant of $w$, therefore by \eqref{eq:reldist52} and Lemma~\ref{lem:piA-properties}
$\diam \Lambda_w \geq \frac{1}{K_5}\diam \Lambda_{\pi(A)} \geq \frac{1}{K_5K_7}a^{-n}$.
Thus in the induction we are only interested in finitely many $w$: let
\[
\cC(v) = \Big\{ w \text{ child of } v : \diam \Lambda_w \geq \frac{1}{K_5K_7} a^{-n} \Big\} \subset V_0T.
\]
For such a $w$, $t_{w,n} = \lfloor n + \log_a \diam \Lambda_w\rfloor \geq \lfloor -\log K_5K_7 \rfloor$; let \[ t_0:= \lfloor -\log K_5K_7 \rfloor. \]

In our induction, the problem is that a child $w$ of a given $v \in T$ may have $t_{w,n}\geq t_{v,n}$.
However, by Lemma~\ref{lem:reldist5}\eqref{eq:reldist52} for 
\[ 
	t_{\Delta} := \log K_5 +1 \geq 1
\]
if $w \in T_0(v)$ then $t_{w,n} \leq t_{v,n}+t_{\Delta}$.
To set up the induction, we group together the finitely-many large descendants of $v$,
where ``large'' depends on a parameter $q \leq t_{\Delta}$: let
\begin{gather*}
	\begin{split}
		\cL^*(v,q) & := \{w \in T_0(v)\setminus\{v\} : \exists u \in T_0(w), t_{u,n} \geq t_{v,n}-q\} 
		\\ & \quad\quad \cup (T_{\delta'} \cap (T_0(v)\setminus \{v\}))\ ,
\quad  \text{ and} 
	 \\ 
\cL(v,q)& :=\cL^*(v,q) \cap \cC(v).
\end{split}
\end{gather*}

Note that as $v \notin T_{\delta'}$, $T_{\delta'}\cap (T_0(v)\setminus\{v\}) = \emptyset$, but we will also use these definitions of $\cL^*(v,q), \cL(v,q)$ later for more general $v$.
There are uniform bounds $|\cL(v,q)| \leq |\cL^*(v,q)| \leq C(t_{\Delta},\delta')$ for any $v\in T$: by \eqref{eq:reldist52} $|\cL^*(v,q)|$ is bounded by the sum of $|T_{\delta'}|$ and the number of spaces $Z_w \subset Z$ which can meet a ball of a radius depending on $t_{\Delta}$ and $K_5$.
In the remainder of this subsection~\ref{ssec:vol-small-vertices}, $v \notin T_{\delta'}$.

Write $\cL(v):=\cL(v,0), \cL^*(v):=\cL^*(v,0)$.
For each $t\in\Z$ define
\[
\hat{V}_{<t} := \max\{\hat{V}_s:s<t\},
\]
and note that we immediately have:
\begin{lemma}\label{lem:volume-vhat-def}
For $t =t_{v,n} = \lfloor n+\log_a D_v \rfloor$, if $w \in \cC(v)\setminus\cL(v)$ then $V_n(w)/D_w^p \leq \hat{V}_{<t}$.
More generally, if $w \in \cC(v)\setminus \cL(v,q)$, then $V_n(w)/D_w^p \leq \hat{V}_{<(t-q)}$.
\end{lemma}

If $D_v$ is too close to $a^{-n}$ there isn't space to do induction, so we fix 
\begin{equation}\label{eq:deft0prime}
	t_0' := 
		\lceil \tau^{-1} (2+\log_a(2\lambda)) \rceil 
		\in \N
\end{equation} 
For any $v \in T$ with $t_{v,n}\geq t_0'$, $\tau t_{v,n} \geq m_v > 1$ by \eqref{eq:defmv}, and therefore we have that not only is $|m_v -\tau t_{v,n}|$ bounded by a universal constant, but also $1/C \leq m_v / \tau t_{v,n} \leq C$ for some universal $C>1$, i.e.\ 
\begin{equation}
	\label{eq:mv-nice}
	m_v \approx \tau t_{v,n} \quad \text{and} \quad m_v \asymp \tau t_{v,n}.
\end{equation}

Our main technical bound in this subsection is the following.

\begin{proposition}\label{prop:volume-small-v-case}
	There exists $C$ depending on $t_{\Delta}$ and the data of our construction so that for any $v \notin T_{\delta'}$ with $t:=t_{v,n} \geq t_0'$ and any $q \leq t_{\Delta}$, we have
\begin{align*}
	\frac{V_n(v)}{D_v^p}
	 \leq \frac{C}{t^{p-1}} \hat{V}_{<(t-q)} + 
	\frac{C}{t^p} 
	 \sum_{w\in\cL(v,q)} \frac{V_n(w)}{D_w^p} .
 \end{align*}
\end{proposition}
\begin{proof}
We decompose $V_n(v)$ as follows:
\begin{equation}\label{eq:vol2}
	\begin{aligned}	V_n(v) & = E_1^p a^{-np}\sum_{A \in \cS_n(v)} \prod_{w \in T_0(v)} \rho_w^n(A)^p
	\\ & = \underbrace{E_1^p a^{-np} \sum_{A \in \cS_n^0(v)} \rho_{v}^n(A)^p}_{\text{(I)}} 
	 +\underbrace{E_1^p a^{-np} \sum_{w \in \cC(v)} \sum_{A \in \cS_n(w)} \rho_{v}^n(A)^p \prod_{u \in T_0(w)} \rho_u^n(A)^p}_{\text{(II)}}.
\end{aligned}
\end{equation}
	and we proceed to bound $(I)$ and $(II)$ in the following lemmas, whose proofs we defer.  (Recall that our parameters $E_1,E_2,E_3,\delta'$ 
	are now fixed constants.)

\begin{lemma}
	\label{lem:vol-small-bound-I}
	There exists $C$ so that for any $v \notin T_{\delta'}$ with $t=t_{v,n} \geq t_0'$ we have
	\begin{equation*}
		(I) \preceq_{C} a^{-np} + \frac{D_v^p}{m_v^p} \ .
	\end{equation*}
\end{lemma}
\begin{lemma}
	\label{lem:vol-small-bound-II}
	There exists $C$ depending on $t_\Delta$ so that for any $v\notin T_{\delta'}$ with $t=t_{v,n}\geq t_0'$ and $q \leq t_\Delta$ we have
	\begin{equation*}
	(II)  
	\preceq_{C} a^{-np}
		+ \frac{D_v^p\; \hat{V}_{<(t-q)}}{m_v^{p-1}} 
	 + D_v^p \sum_{w\in\cL(v,q)} \frac{1}{m_v^p} \cdot \frac{V_n(w)}{D_w^p} \ .
\end{equation*}
\end{lemma}

We now combine these bounds. 
By \eqref{eq:vol2} and Lemmas~\ref{lem:vol-small-bound-I} and \ref{lem:vol-small-bound-II}:
\begin{align*}
	\frac{V_n(v)}{D_v^p}
	& \preceq \frac{a^{-np}}{D_v^p}
	+  \frac{1}{m_v^p} 
	+ \frac{a^{-np}}{D_v^p}+ 
	\hat{V}_{<(t-q)} \cdot 
	\frac{1}{m_v^{p-1}} 
  	 + \sum_{w\in\cL(v,q)} \frac{1}{m_v^p} \cdot \frac{V_n(w)}{D_w^p} \ .
\end{align*}
As $t =t_{v,n} \geq t_0'$ we have \eqref{eq:mv-nice}, so we can write this as
\begin{align*}
	\frac{V_n(v)}{D_v^p}
	& \preceq a^{-tp} + \frac{1}{t^p} 
	+ \frac{1}{t^{p-1}} \hat{V}_{<(t-q)}
	+ \frac{1}{t^p} \sum_{w\in\cL(v,q)} \frac{V_n(w)}{D_w^p}
	\\ & \preceq \frac{1}{t^{p-1}} \hat{V}_{<(t-q)} + 
	\frac{1}{t^p} 
	 \sum_{w\in\cL(v,q)} \frac{V_n(w)}{D_w^p} \ ,
\end{align*}
where we use that $\hat{V}_{<(t-q)} \geq 1$, so
\[
a^{-tp} \leq a^{-\tau t p} \preceq \frac{1}{t^p} \preceq \frac{1}{t^{p-1}} \hat{V}_{<(t-q)}.
\]
We have completed the proof of Proposition~\ref{prop:volume-small-v-case}.
\end{proof}

Proposition~\ref{prop:volume-small-v-case} applies as follows to prove that $\lim_{t\to\infty}\hat{V}_t =0$.
\begin{proof}[{Proof of Proposition~\ref{prop:volume-v-hat-goes-to-0}}]
	Note that by Lemma~\ref{lem:vhat-finite}, for $t \in \Z$ with $t_0 \leq t < t_0'$ we have $\hat{V}_t<\infty$, so we can restrict to values of $t \geq t_0'$ where Proposition~\ref{prop:volume-small-v-case} applies.

	Our goal is to bound $V_n(v)/D_v^p$ in terms of $\hat{V}_{<t}$, by applying Proposition~\ref{prop:volume-small-v-case} boundedly many times.
	As the first step, Proposition~\ref{prop:volume-small-v-case} for $v$ and $q=0$ gives:
\begin{align*}
	\frac{V_n(v)}{D_v^p}
	 \leq \frac{C}{t^{p-1}} \hat{V}_{<t} + 
	\frac{C}{t^p} \sum_{w\in\cL(v,0)} \frac{V_n(w)}{D_w^p} . 
 \end{align*}

	Now consider $w \in \cL(v,0)$; some descendant $u$ of $w$ has $t_u \geq t_{v,n}$, so $t_{w,n} \geq t_{v,n}-t_{\Delta}$, and also $t_{w,n} \leq t_{v,n}+t_{\Delta}$. 
	If $t_{w,n} < t_{v,n}$ then as $w\in\cC(v), t_{w,n} \geq t_0$ so $V_n(w)/D_w^p \leq \hat{V}_{t_{w,n}} \leq \hat{V}_{<t}$.
	If $t_{w,n} \geq t_{v,n}$, apply Proposition~\ref{prop:volume-small-v-case} to $w$ with $q=t_{w,n}-t_{v,n} \leq t_\Delta$ to get
\[
\frac{V_n(w)}{D_w^p} \leq \frac{C}{t_{w,n}^{p-1}} \hat{V}_{<(t_{w,n}-q)} + \frac{C}{t_{w,n}^p} 
\sum_{w'\in\cL(w,q)} \frac{V_n(w')}{D_{w'}^p}
\]
where $C$ is the constant from Proposition~\ref{prop:volume-small-v-case}.
Note that $\cL^*(w,q) \subset \cL^*(v,0)$ so these sums are getting smaller.
	If any $w'$ has $t_{w',n}<t_{w,n}-q=t_{v,n}$ we bound $V_n(w')/D_{w'}^p \leq \hat{V}_{<t}$,
	otherwise we apply Proposition~\ref{prop:volume-small-v-case} to $w'$ with ``$q$''$=t_{w',n}-(t_{w,n}-q)=t_{w',n}-t_{v,n}$ to bound the term, and continue this process.  Note that any summand $V_n(w'')/D_{w''}^p$ that appears has $w'' \in \cL^*(v,0)$ and appears once.
	Since $\cL^*(v,0)$ has a uniform finite bound on its size, after boundedly many steps this process terminates when there are no more vertices $w''$ with $t_{w'',n} \geq t_{v,n}$.  At the end we have
\begin{align*}
	\frac{V_n(v)}{D_v^p}
	& \leq \frac{C'}{t^{p-1}} \hat{V}_{<t}  
\end{align*}
for a constant $C'$.
Taking the supremum over all $n$ and all $v$ with $t_{v,n}=t$, we get
\begin{equation*}
  \hat{V}_{t} \leq \frac{C'}{t^{p-1}} \hat{V}_{<t},
\end{equation*}
so for $t$ large the sequence $\hat{V}_{t}$ is a decreasing function of $t$, hence the sequence is uniformly bounded, and by the same inequality $\lim_{t \to \infty} \hat{V}_{t} = 0$.
\end{proof}

We still have to prove the technical bounds of Lemmas~\ref{lem:vol-small-bound-I} and \ref{lem:vol-small-bound-II}.

\subsubsection{Bound (I)}
\begin{proof}[{Proof of Lemma~\ref{lem:vol-small-bound-I}}]
We further split $\cS_n^0(v) = \cR(v)\sqcup \cR(v)^c$ where $\cR(v)$ consists of those $A \in \cS_n^0(v)$ with $d(A, U_{\gets v}) > E_2 a^{-n}$, and $\cR(v)^c = \cS_n^0(v)\setminus \cR(v)$.
\begin{align}
	(I) & = E_1^p a^{-np} \sum_{A \in \cR(v)^c} \rho_{v}^n(A)^p + E_1^p a^{-np} \sum_{A \in \cR(v)} \rho_{v}^n(A)^p \notag
	\\ & \asymp  a^{-np} \left| \cR(v)^c \right| 
	+  a^{-np} \sum_{A \in \cR(v)} \left(\frac{\diam h_v(B_A) \cdot D_v}{\diam B_A \cdot f_v(B_A)}\right)^p \label{eq:volI1}
\end{align}
as for $A \in S_n^0(v)$, $W_{v,A}=B_A$.
	Lemma~\ref{lem:reldist3} gives that any $A\in\cR(v)^c$ is a distance $\leq K_3E_2a^{-n}$  to $\Lambda_{e_v}$, so the doubling property of $X$ gives
\begin{equation}
	\label{eq:volI2} |\cR(v)^c| \leq C = C(E_2, K_3). 
\end{equation}

The value of $f_v(B_A)$ depends on $d_v(h_v B_A, h_v \Lambda_{e_v})$, and so we write $\cR(v) = \cR(v,1) \sqcup \cdots \sqcup \cR(v,m_v) \sqcup \cR(v,m_v+1)$ where 
\[
\cR(v,j) := \{A \in \cR(v) : d_v(h_vB_A,h_v\Lambda_{e_v}) \in (a^{-j},a^{-j+1}]\},
\] 
for $1 \leq j\leq m_v$ and $\cR(v,m_v+1) = \{A \in \cR(v) : d_v(h_vB_A,h_v\Lambda_{e_v}) \leq a^{-m_v}\}$.
	So, pulling out for now the common factor $a^{-np} D_v^p$, the second term of \eqref{eq:volI1} is
\begin{align}
& \sum_{A \in \cR(v)} \left(\frac{\diam h_v(B_A)}{\diam B_A \cdot f_v(B_A)}\right)^p \notag
\\ & = \sum_{j=1}^{m_v+1} \sum_{A \in \cR(v,j)} \left(\frac{\diam h_v(B_A)}{\diam B_A \cdot f_v(B_A)}\right)^p \notag
\\ & \preceq \sum_{A \in \cR(v,m_v+1)} \frac{(\diam h_v(B_A))^{Q_v+p-Q_v}}{a^{-np} \cdot 1} 
+ \sum_{j=1}^{m_v} \sum_{A \in \cR(v,j)} \frac{(\diam h_v(B_A))^{Q_v+p-Q_v}}{a^{-np} m_v^p a^{-jp}} \ . \label{eq:volI3}
\end{align}
	Recall that by \eqref{eq:taulambdaHolder}, $\diam h_vB_A \leq a^{-m_v}/2$.  Since $\cS_n$ is a bounded multiplicity cover of $X$ and each $B_A \subset (K_7+1)A$ the collection $\{B_A : A \in \cS_n^0(v)\}$ has bounded multiplicity, with constants depending only on $X$.  Thus $\{h_v B_A : A \in \cS_n^0(v)\}$ is a bounded multiplicity collection of quasi-balls, and so the Ahlfors $Q_v$-regularity of $X_v$ gives, for $1 \leq j \leq m_v+1$,
\begin{equation}\label{eq:volI4}
	\begin{split} \sum_{A \in \cR(v,j)} (\diam h_v(B_A))^{Q_v} \diam h_v(B_A)^{p-Q_v}
& \leq  \sum_{A \in \cR(v,j)} (\diam h_v(B_A))^{Q_v} a^{-m_v (p-Q_v)}
\\ & \preceq a^{-(j-1)Q_v} \cdot a^{-m_v (p-Q_v)} .
\end{split}\end{equation}
So when $j=m_v+1$ the right-hand side is $a^{-pm_v}$.
	By~\eqref{eq:volI4} the second term in \eqref{eq:volI3} sums to
\begin{align}
& \preceq \sum_{j=1}^{m_v} \frac{a^{-(j-1)Q_v} \cdot a^{-m_v (p-Q_v)}}{a^{-np} m_v^p a^{-jp}}
\asymp \frac{a^{-m_v(p-Q_v)+np}}{m_v^p} \sum_{j=1}^{m_v} a^{j(p-Q_v)} \notag
\\ & \preceq \frac{a^{-m_v(p-Q_v)+np}}{m_v^p} \cdot a^{m_v(p-Q_v)} = \frac{a^{np}}{m_v^p}. \label{eq:volI5}
\end{align}
	Combining \eqref{eq:volI1},\eqref{eq:volI2},\eqref{eq:volI3},\eqref{eq:volI4},\eqref{eq:volI5} and $a^{-pm_v}\preceq 1/m_v^p$ we have
\begin{equation}
	(I) 
	\preceq a^{-np} 
	+ a^{-np} D_v^p \left( \frac{a^{-pm_v}}{a^{-np}} + \frac{a^{np}}{m_v^p} \right) 
	\preceq a^{-np} + D_v^p \cdot \frac{1}{m_v^p} \ . 
	\label{eq:volI6} \qedhere
\end{equation}
\end{proof}

\subsubsection{Bound (II)}
\begin{proof}[{Proof of Lemma~\ref{lem:vol-small-bound-II}}]
	Note that we are considering $A \in \cS_n(w)$ for some $w \in \cC(v)$, i.e.\ $\pi(A)$ equals $w$ or a descendant of $w$, therefore $W_{v,A} = \Lambda_{e_{v_{\to A}}} = \Lambda_{e_w}$ in this proof.

	Let us denote by $(IIa)$ the contribution to $(II)$ by $w \in \cC(v)$ and $A \in \cS_n(w)$ with $d(A, U_{\gets v}) \leq E_2 a^{-n}$.
	For such $w,A$ by Lemmas~\ref{lem:reldist1}--\ref{lem:reldist4} we have
	\begin{align*}
		a^{-n} \preceq \diam \Lambda_w &
		\overset{\text{\ref{lem:reldist4}}}{\asymp} \diam \Lambda_{e_w}
		\overset{\text{\ref{lem:reldist4}}}{\preceq} \min\{\diam \Lambda_{e_w}, \diam \Lambda_{e_v}\} 
		\\ & \overset{\text{\ref{lem:reldist1}}}{\preceq} d(\Lambda_{e_w},\Lambda_{e_v})
		\overset{\text{\ref{lem:reldist2}}}{\asymp} d(U_{\gets v},U_{w\to}) 
		\leq d(U_{\gets v},A) \preceq a^{-n},
	\end{align*}
	so all such $w$ have $\diam\Lambda_w \asymp a^{-n}$ and also $d(A, \Lambda_{e_v}) \preceq a^{-n}$ by Lemma~\ref{lem:reldist3}.
	By doubling and the separation property of Lemma~\ref{lem:reldist1} there are $\preceq 1$ such $w$.
	Moreover, $\rho_v^n(A)=1$ and $\rho_u^n(A)=1$ for all $u \in T_w$, since $U_{\gets v} \subset U_{\gets u}$ so $d(A,U_{\gets u})\leq E_2a^{-n}$.  So the total contribution to $(II)$ by $w,A$ as above with $d(A,U_{\gets v})\leq E_2 a^{-n}$ is
\begin{equation}
	\label{eq:volII1}
	(IIa) \preceq a^{-np}.
\end{equation}
So in the remainder of this proof we only need to consider $w \in \cC(v)$ and $A \in \cS_n(w)$ with $d(A,U_{\gets v}) > E_2 a^{-n}$.

In case $(I)$ above we partitioned $A \in \cS_n^0(v)$ according to the distance values $d_v(h_v B_A,h_v\Lambda_{e_v})$; this time we partition the set of children of $v$ according to both their distance $d_v(h_v \Lambda_{e_w}, h_v \Lambda_{e_v})$ and their size $\diam h_v \Lambda_{e_w}$ in the model space $X_v$.
We defined $\cL(v,q) \subset \cC(v)$ to be, depending on $q$, those children with large descendants, and will consider their contribution in $(IIc)$ below.
For $1 \leq j \leq m_v$ and $1 \leq k$ we partition the remaining children as follows:
\begin{multline*}
	\cC(v,q,j,k) = \Big\{w \in \cC(v)\setminus \cL(v,q) : d_v(h_v \Lambda_{e_w}, h_v \Lambda_{e_v}) \in (a^{-j},a^{-j+1}], \\ \text{ and } \diam h_v\Lambda_{e_w} \in (a^{-k},a^{-k+1}]\Big\},
\end{multline*}
and we define $\cC(v,q,j,k)$ similarly when $j=m_v+1$, replacing $(a^{-(m_v+1)},a^{-m_v}]$ by $(0,a^{-m_v}]$ in the appropriate place.

	There exists $k_\Delta$ so that if $j \geq k+k_\Delta$ then $\cC(v,q,j,k) = \emptyset$ because 
	Lemma~\ref{lem:reldist1} gives a lower bound on the relative distance between $\Lambda_{e_w}$ and $\Lambda_{e_v}$, and so there is a uniform lower bound on the relative distance of $h_v\Lambda_{e_w}$ and $h_v\Lambda_{e_v}$, since $h_v$ is $\eta$-quasisymmetric for uniform $\eta$ \cite[Lemma 3.2]{BK-02-S2-unif}.

Also, as already remarked, $\cC(v)$ is finite so only finitely many $\cC(v,q,j,k)$ are non-empty.

Having this notation, we bound the terms of $(II)$ not in $(IIa)$ as follows:
\begin{align}
   &  E_1^pa^{-np} \sum_{w \in \cC(v)}
  \sum_{\substack{A \in \cS_n(w) \\ d(U_{\gets v},A) >E_2 a^{-n}}} \rho_v^n(A)^p \prod_{u \in T_0(w)} \rho_u^n(A)^p
\notag	\\ & = E_1^pa^{-np} \sum_{w \in \cC(v)} \sum_{\substack{A \in \cS_n(w) \\ d(U_{\gets v},A) >E_2 a^{-n}}} 
		\left(\frac{\diam h_v(\Lambda_{e_w})}{\diam \Lambda_{e_w}} \frac{E_3 D_v}{f_v(\Lambda_{e_w})}\right)^p
		\prod_{u \in T_0(w)} \rho_u^n(A)^p
\notag	\\ & \preceq D_v^p \Bigg( 
	\underbrace{\sum_{k=1}^{\infty} \sum_{j=1}^{(m_v+1)\wedge (k+k_\Delta)} 
	\sum_{w \in \cC(v,q,j,k)} 
	\frac{a^{-kp} V_n(w)}{D_w^p f_v(\Lambda_{e_w})^p} 
	}_{(IIb)} 
		+
		\underbrace{\sum_{w \in \cL(v,q)} 
		\frac{(\diam h_v(\Lambda_{e_w}))^p V_n(w)}{D_w^p f_v(\Lambda_{e_w})^p}}_{(IIc)} \Bigg)
	\label{eq:volII2}
\end{align}

We now use that $v \notin T_{\delta'}$.
Dropping for the moment the constant $D_v^p$, we decompose $(IIb)$ as 
\begin{multline}
	(IIb) 
	\leq \sum_{k=1}^{\infty} \sum_{j=1}^{m_v \wedge (k+k_\Delta)}
	\sum_{w \in \cC(v,q,j,k)} 
	\frac{a^{-kp}V_n(w)}{D_w^p m_v^p a^{-jp}} 
	\\ +	
	\sum_{k=1\vee(m_v-k_\Delta+1)}^{\infty} 
	\sum_{w \in \cC(v,q,m_v+1,k)} 
	\frac{a^{-kp}V_n(w)}{D_w^p \cdot 1} ,
	\label{eq:volII3}
\end{multline} 
where we use that for $w \in \cC(v,q,j,k)$, $f_v(\Lambda_{e_w}) = 1$ if $j=m_v+1$ and $f_v(\Lambda_{e_w}) \geq m_v a^{j}$ if $j \leq m_v$.

In \eqref{eq:volII3}, each $w$ considered is not in $\cL(v,q)$, so $t_{w,n} < t_{v,n}-q=t-q$, so by the definition of $\hat{V}_{<(t-q)}$, \eqref{eq:volII3} is at most
\begin{equation}\label{eq:volII4}
	\begin{split} 
		& \Bigg( \sum_{k=1}^{\infty} \sum_{j=1}^{m_v\wedge (k+k_\Delta)}
	\sum_{w \in \cC(v,q,j,k)} 
	\frac{a^{-kp} }{m_v^p a^{-jp}} 
	+
	\sum_{k=1\vee(m_v-k_\Delta+1)}^{\infty} 
	\sum_{w \in \cC(v,q,m_v+1,k)} 
	a^{-kp} \Bigg) \hat{V}_{<(t-q)}
\end{split}
\end{equation}
Now by a volume estimate, for each $j, k$,
\begin{align}
	\sum_{w\in\cC(v,q,j,k)} a^{-kp}
	&= \sum_{w\in\cC(v,q,j,k)} a^{-kQ_v}a^{-k(p-Q_v)}
	\notag \\ & \leq \sum_{w \in \cC(v,q,j,k)} (\diam h_v(\Lambda_{e_w}))^{Q_v} a^{-k(p-Q_v)} \notag \\ & 
	 \preceq a^{-jQ_v} a^{-k(p-Q_v)}.  \label{eq:volII4s1}
\end{align}
So \eqref{eq:volII4} is at most $\hat{V}_{<(t-q)}$ times
\begin{align}
	& \sum_{k=1}^{\infty} \sum_{j=1}^{m_v\wedge(k+k_\Delta)}
	\frac{a^{-jQ_v} a^{-k(p-Q_v)}}{m_v^p a^{-jp}} 
	+
	 \sum_{k=1\vee(m_v-k_\Delta+1)}^{\infty} 
 	a^{-(m_v+1)Q_v} a^{-k(p-Q_v)} 
	\notag \\&
	\preceq
	\frac{1}{m_v^p} \sum_{k=1}^{m_v-k_\Delta} 
	a^{k(p-Q_v)} a^{-k(p-Q_v)} 
	+
	\frac{1}{m_v^p} \sum_{k=m_v-k_\Delta+1}^{\infty} 
	a^{m_v (p-Q_v)} a^{-k(p-Q_v)} 
	\notag \\&\qquad +
	 \sum_{k=m_v-k_\Delta+1}^{\infty} 
 	a^{-m_v Q_v} a^{-k(p-Q_v)} 
\notag 	\\ & =  
	\frac{1}{m_v^p} \sum_{k=1}^{m_v-k_\Delta} 1 
	+
	\frac{1}{m_v^p} \sum_{k=m_v-k_\Delta+1}^{\infty} 
	a^{-(k-m_v)(p-Q_v)} 
 	+	 \sum_{k=m_v-k_\Delta+1}^{\infty} 
 	a^{-m_v Q_v - k(p-Q_v)} 
\notag 	\\ & 
	\preceq 
	\frac{1}{m_v^{p-1}}
	+
	\frac{1}{m_v^p}  
	+
	a^{-m_v Q_v - m_v(p-Q_v)}
	 \preceq \frac{1}{m_v^{p-1}} \ . \label{eq:volII4b}
\end{align}
where the implied constants depend on $k_\Delta, p-Q_v$ and our other data.

It remains to bound $(IIc)$ from \eqref{eq:volII2}.  We now show
\begin{align}\label{eq:volII5}
  (IIc) & = \sum_{w\in\cL(v,q)} \frac{(\diam h_v(\Lambda_{e_w}))^p}{f_v(\Lambda_{e_{w}})^p} \cdot \frac{V_n(w)}{D_w^p}
  \preceq \sum_{w\in\cL(v,q)} \frac{1}{m_v^p} \cdot \frac{V_n(w)}{D_w^p} \ .
\end{align}
This is true because if $w \in \cL(v,q)$ we have $t_w\geq t_v-q-t_\Delta \geq t_v-2t_\Delta$, 
thus $\diam \Lambda_{e_w} \asymp \diam \Lambda_{e_v}$, and so by \eqref{eq:taulambdaHolder} $\diam h_v(\Lambda_{e_w}) \asymp 1$.
Moreover $d(\Lambda_{e_w}, \Lambda_{e_v}) \succeq \diam \Lambda_{e_w}$ by Lemma~\ref{lem:reldist1}.
By uniform relative distance distortion of quasisymmetric maps, we then get that $d_v(h_v\Lambda_{e_w},h_v\Lambda_{e_v})$ is bounded away from $0$ for a uniform constant.
In the case $d_v(h_v\Lambda_{e_w},h_v\Lambda_{e_v}) < a^{-m_v}$ then $m_v$ is bounded from above, so by \eqref{eq:mv-nice} $m_v\asymp 1$, thus $f_v(\Lambda_{e_w}) = 1 \asymp m_v$.
On the other hand if $d_v(h_v\Lambda_{e_w},h_v\Lambda_{e_v})\geq a^{-m_v}$ then $f_v(\Lambda_{e_w})=m_v d_v(h_v\Lambda_{e_w},h_v\Lambda_{e_v})\succeq m_v$.  In either case \eqref{eq:volII5} holds.

So in total, \eqref{eq:volII1}, \eqref{eq:volII2}, \eqref{eq:volII3}, \eqref{eq:volII4}, \eqref{eq:volII4b} and \eqref{eq:volII5} give
\begin{align*}
	(II) & \preceq (IIa) + D_v^p \big( (IIb)+(IIc) \big)  
	\\ & \preceq 
	a^{-np} + D_v^p \hat{V}_{<(t-q)} \cdot \frac{1}{m_v^{p-1}} 
	+ D_v^p \sum_{w\in\cL(v,q)} \frac{1}{m_v^p} \cdot \frac{V_n(w)}{D_w^p} \ . \qedhere
\end{align*}
\end{proof}

\subsection{Volume bound for $v \in T_{\delta'}$}
For the boundedly many vertices in $T_{\delta'}$, our bound of Proposition~\ref{prop:volume-small-v-case} need not hold, however the following weaker bound does hold by a similar argument.

Note that for all large enough $n$, for any $v \in T_{\delta'}$ we have $t_{v,n} \geq t_0'$ where $t_0'$ is the constant of \eqref{eq:deft0prime}, since $t_{v,n} \approx n$.  So we assume from now on that for all $v \in T_{\delta'}$ we have $\tau t \geq m_v > 1$ and so \eqref{eq:mv-nice} holds as well: $m_v \approx \tau t$ and $m_v \asymp \tau t$.
\begin{proposition}\label{prop:volume-large-v-case}
	There exists $C$ depending on $t_{\Delta}$ and the data of our construction so that for any $v \in T_{\delta'}$ and any $q \leq t_{\Delta}$, we have
\begin{align*}
	\frac{V_n(v)}{D_v^p}
	\leq C \hat{V}_{<(t-q)} + 
	{C} 
	 \sum_{w\in\cL(v,q)} \frac{V_n(w)}{D_w^p}
 \end{align*}
\end{proposition}

\begin{proof}
	We follow the proofs of Proposition~\ref{prop:volume-small-v-case} and Lemmas~\ref{lem:vol-small-bound-I} and \ref{lem:vol-small-bound-II} in Subsection~\ref{ssec:vol-small-vertices}, but consider the case $v \in T_{\delta'}$.

	Recall that by \eqref{eq:vol2} we can write $V_n(v)$ as two sums, $(I)$ where $A \in S_n^0(v)$, or $(II)$ where $\pi(A)$ is in $S_n(w)$ for some child $w \in \cC(v)$.
	These are bounded as follows; we defer the proofs until later.
	\begin{lemma}
		\label{lem:vol-large-bound-I}
	There exists $C$ so that for any $v \in T_{\delta'}$ we have
	\begin{equation*}
		(I) \preceq_{C} a^{-np} + D_v^p a^{-m_v(p-Q_v)}.
	\end{equation*}
	\end{lemma}
	\begin{lemma}
		\label{lem:vol-large-bound-II}
	There exists $C$ depending on $t_\Delta$ so that for any $v\in T_{\delta'}$ and $q \leq t_\Delta$ we have
	\begin{equation*}
	(II)  
	\preceq_{C} a^{-np}
	+ D_v^p \hat{V}_{<(t-q)} 	+ 
	D_v^p \sum_{w\in\cL(v,q)} \frac{V_n(w)}{D_w^p}.
	\end{equation*}
	\end{lemma}

	We now combine \eqref{eq:vol2}, Lemmas~\ref{lem:vol-large-bound-I} and \ref{lem:vol-large-bound-II} to find:
\begin{align*}
	\frac{V_n(v)}{D_v^p}
	& \preceq \frac{a^{-np}}{D_v^p}
	+ a^{-m_v(p-Q_v)} 
	+ \hat{V}_{<(t-q)} 
	+ \sum_{w \in \cL(v,q)} \frac{V_n(w)}{D_w^p}
	\\ & \asymp
	a^{-tp} + a^{-\tau t (p-Q_v)} 
	+ \hat{V}_{<(t-q)}
	+ \sum_{w\in\cL(v,q)} \frac{V_n(w)}{D_w^p}
	\\ & \preceq  \hat{V}_{<(t-q)} + 
	 \sum_{w\in\cL(v,q)} \frac{V_n(w)}{D_w^p} ,
\end{align*}
where we use that $m_v \approx \tau t$ and $\hat{V}_{<(t-q)} \geq 1$.
The proposition is proven.
\end{proof}

\subsubsection{Bound (I)}

\begin{proof}[Proof of Lemma~\ref{lem:vol-large-bound-I}]
	We follow the notation and proof of Lemma~\ref{lem:vol-small-bound-I}.
	The argument begins identically with \eqref{eq:volI1} and \eqref{eq:volI2}.
	
	Instead of decomposing $\cR(v)\subset \cS_n^0(v)$ to find the bound \eqref{eq:volI3}, we instead use the simpler fact that $\{h_v B_A : A \in \cS_n^0(v)\}$ is a bounded multiplicity collection of quasi-balls in the Ahlfors $Q_v$-regular space $X_v$.  Therefore, as $f_v(B_A)=1$, we have: 
\begin{align*}
 \sum_{A \in \cR(v)} \left(\frac{\diam h_v(B_A)}{\diam B_A \cdot f_v(B_A)}\right)^p 
& \preceq a^{np} \sum_{A \in \cR(v)} (\diam h_v B_A)^{Q_v+(p-Q_v)}
\\ & \leq a^{np} a^{-m_v(p-Q_v)} \sum_{A \in \cR(v)} (\diam h_v B_A)^{Q_v}
\\ & \preceq a^{np} a^{-m_v(p-Q_v)} .
\end{align*}
so instead of \eqref{eq:volI6} we have
\begin{align}
	(I) & 
	\preceq a^{-np} 
	+ a^{-np} D_v^p \left( a^{np} a^{-m_v(p-Q_v)} \right) 
	= a^{-np} + D_v^p a^{-m_v(p-Q_v)} . \qedhere \notag 
\end{align}
\end{proof} 

\subsubsection{Bound (II)}
\begin{proof}[Proof of Lemma~\ref{lem:vol-large-bound-II}]
We follow the argument and notation used in the proof of Lemma~\ref{lem:vol-small-bound-II}.

As before, $(IIa)$ satisfies the bound of \eqref{eq:volII1}.
We write the remaining terms as $D_v^p\big((IIb)+(IIc)\big)$ as in \eqref{eq:volII2}.
Since $f_v(\Lambda_{e_w})=1$, the bound \eqref{eq:volII3} for $(IIb)$ is replaced by the following
(again we drop for the moment the constant $D_v^p$):
\begin{equation}
		(IIb)  
=
		\sum_{k=1}^{\infty} \sum_{j=1}^{(m_v+1)\wedge(k+k_\Delta)}
	\sum_{w \in \cC(v,q,j,k)} 
	\frac{a^{-kp}V_n(w)}{D_w^p \cdot 1} .
	\label{eq:volII31}
\end{equation}

In \eqref{eq:volII31}, each $w$ considered is not in $\cL(v,q)$, so $t_{w,n}<t_{v,n}-q=t-q$, so by the definition of $\hat{V}_{<(t-q)}$, \eqref{eq:volII31} is at most
\begin{equation}\label{eq:volII41}
		\Bigg( 
		\sum_{k=1}^{\infty} \sum_{j=1}^{(m_v+1)\wedge(k+k_\Delta)}
	\sum_{w \in \cC(v,q,j,k)} 
	a^{-kp} \Bigg) \hat{V}_{<(t-q)}
\end{equation}
Now the volume estimate \eqref{eq:volII4s1} gives for each $j, k$, 
	\[ \sum_{w\in\cC(v,q,j,k)} a^{-kp}
	\preceq a^{-jQ_v} a^{-k(p-Q_v)}, \]
	so \eqref{eq:volII41} is at most
	\begin{align*}
	 	\sum_{k=1}^\infty \sum_{j=1}^{(m_v+1)\wedge(k+k_\Delta)} 
		a^{-jQ_v} a^{-k(p-Q_v)} \hat{V}_{<(t-q)}
		& \leq \sum_{k=1}^\infty \sum_{j=1}^\infty 
		a^{-jQ_v} a^{-k(p-Q_v)} \hat{V}_{<(t-q)}
		\\ & \preceq \hat{V}_{<(t-q)}.
	\end{align*}
So in total, \eqref{eq:volII1}, \eqref{eq:volII2}, \eqref{eq:volII31}, \eqref{eq:volII41}, and the above give
\begin{align*}
	(II) 
	& \preceq (IIa) + D_v^p\big( (IIb)+(IIc) \big)  
	\\ & \preceq a^{-np}
	+ D_v^p  
	\hat{V}_{<(t-q)} 
	+ D_v^p \cdot (IIc),
\end{align*}
	where, as $\diam h_v(\Lambda_{e_w}) \leq 1$ and $f_v(\Lambda_{e_w})=1$,
\begin{align*}
  (IIc) & = \sum_{w\in\cL(v,q)} \frac{(\diam h_v(\Lambda_{e_w}))^p}{f_v(\Lambda_{e_{w}})^p} \cdot \frac{V_n(w)}{D_w^p}
  \leq \sum_{w\in\cL(v,q)} \frac{V_n(w)}{D_w^p} . \qedhere
\end{align*}
\end{proof}

\subsection{Uniform volume bounds}

We can now complete the proof that $V_n(v_0)$ is bounded independently of $n$.
\begin{proof}[Proof of Theorem~\ref{thm:volbound}]
Our goal is to bound $V_n(v_0) \asymp V_n(v_0)/D_{v_0}^p$ independently of $n$.

	First we apply Proposition~\ref{prop:volume-large-v-case} with $v=v_0, t= t_{v_0,n}\leq n, q=0$ (additionally, $t \approx n$) to bound $V_n(v_0)$ by
\[
	V_n(v_0) \preceq
	\hat{V}_{< t_{v_0,n}}+ \sum_{w \in \cL(v_0,0)} \frac{V_n(w)}{D_w^p}
	\leq \hat{V}_{< n}+ \sum_{w \in \cL(v_0,0)} \frac{V_n(w)}{D_w^p}.
\]
For each $w$ on the right-hand side which is in $T_{\delta'}$, we apply Proposition~\ref{prop:volume-large-v-case} to it (with $v=w, t=t_{w,n}\leq n, q=0$) to get a bound 
\[
	\frac{V_n(w)}{D_w^p} \preceq \hat{V}_{<n} + \sum_{u \in \cL(w,0)} \frac{V_n(u)}{D_u^p}.
\]
There are boundedly many vertices in $T_{\delta'}$, so after doing this step to each such term, we have that
\[
	V_n(v_0) \preceq \hat{V}_{<n} + \sum_u \frac{V_n(u)}{D_u^p},
\]
where each $u$ in the sum is not in $T_{\delta'}$, but does have that $\Lambda_u$ has size comparable to $\delta'$.  Therefore there is a uniform bound on how many such $u$ appear, and by definition $V_n(u)/D_u^p \leq \hat{V}_{<(n+1)}$ since $t_{u,n} \leq n$.

In conclusion, we have found that 
\[
	V_n(v_0) \leq C \hat{V}_{<(n+1)}
\]
for some $C$ independent of $n$, and this is 
	bounded as by Lemma~\ref{lem:vhat-finite} and Proposition~\ref{prop:volume-v-hat-goes-to-0} $\sup_{t\in\Z} \hat{V}_n < \infty$.
\end{proof}

\section{Attainment of conformal dimension}\label{sec:attained}

In this section we characterise when the conformal dimension of a hyperbolic graph of groups with elementary edge groups is attained.
The key concept we use is porosity.
\begin{definition}
	\label{def:porous}
	A subset $Y$ of a metric space $X$ is \emph{porous} if there exists $c>0$ so that for any $y\in Y$ and $r \leq \diam(X)$ there exists $x \in X$ with $B(x,cr) \subset B(y,r) \setminus Y$.
\end{definition}
Under mild hypotheses, porosity is preserved by quasisymmetric homeomorphisms.
\begin{lemma}[{cf.\ \cite[Theorem 4.2]{Vai-88-porous}}]
	\label{lem:porous-qs-inv}
	If $X$ is a uniformly perfect metric space, and $Y \subset X$ is porous, and $f:X\to X'$ is a quasisymmetric homeomorphism,  then $f(Y) \subset X'$ is porous.
\end{lemma}
\begin{proof}
	Given $B' = B(y',r') \subset X'$ with $y' \in f(Y)$ and $r' \leq \diam X'$,
	since $f^{-1}$ is quasisymmetric there exists  $r>0$ so that $B=B(f^{-1}(y'),r)$ satisfies $B \subset f^{-1}(B') \subset \lambda B$,
	where $\lambda \geq 1$ is a constant depending only on $f$.
	Since $Y$ is porous, there exists $B(x,cr) \subset B \setminus Y$.
	Now $f(B(x,cr)) \subset f(B) \subset B'$, and by quasisymmetry there exists $r''>0$ with $B(f(x),r'') \subset f(B(x,cr)) \subset B(f(x),\lambda r'')$.
	Since $B(x,cr) \subset B \setminus Y$, $B(f(x),r'') \subset B'\setminus f(Y)$, so it remains to show that $r''/r' \geq c' > 0$ for a constant $c'$.

	In a uniformly perfect space, the radius of any ball is comparable to its diameter (indeed, this is an equivalent definition) up to some uniform constant $C$.
	So by \cite[Proposition 10.8]{Hein-01-lect-an-mtc-spc}, since $B \subset f^{-1}(B')$ and $\diam B \asymp \diam f^{-1}(B')$ we have $\diam f(B) \asymp \diam B' \asymp r'$.
	Thus again by \cite[Proposition 10.8]{Hein-01-lect-an-mtc-spc}, writing $\eta:[0,\infty)\to[0,\infty)$ for the distortion function of $f$,
	\begin{align*}
		\frac{r''}{r'} & \asymp \frac{\diam f(B(x,cr))}{\diam f(B)}	
		\geq \frac{1}{2\eta\left(\frac{\diam B}{\diam B(x,cr)}\right)}
		\geq \frac{1}{\eta(C^2/c)} >0. \qedhere
	\end{align*}
\end{proof}
We will use the following criteria for non-attainment of Ahlfors regular conformal dimension, likely well known to experts in the area.
\begin{proposition}
	\label{prop:porous-non-attain}
	Suppose there is a metric space $X$ with a subset $Y \subset X$ that is porous, so that the (Ahlfors regular) conformal dimensions of $Y$ and $X$ are equal (and finite).
	Then the conformal dimension of $X$ is not attained.
\end{proposition}
\begin{proof}
	Suppose otherwise, and that $f:X\to X'$ is a quasisymmetric map with $X'$ Ahlfors regular of dimension $\Confdim X$.
	Since $X'$ is Ahlfors regular it is uniformly perfect, and so is $X = f^{-1}(X')$.
	Then $f(Y) \subset X'$ is porous by Lemma~\ref{lem:porous-qs-inv} above,
	so its Assouad dimension satisfies $\dim_A f(Y) < \dim_A X' = \Confdim X$~\cite[Lemma 5.8]{David-Semmes-97-fract-fractals}.
	For any $Q > \dim_A f(Y)$, $f(Y)$ is quasisymmetric to an Ahlfors $Q$-regular space~\cite[Theorem 14.16]{Hein-01-lect-an-mtc-spc}, so choosing $Q \in (\dim_A f(Y), \Confdim X)$ we get that $\Confdim Y \leq Q < \Confdim X$, a contradiction.
\end{proof}
A useful tool for identifying porous subsets is the following.
\begin{proposition}
	\label{prop:qcvx-porous}
	Suppose $H$ is a quasiconvex subgroup of a hyperbolic group $G$.
	Then the following are equivalent:
	\begin{enumerate}
		\item the limit set $\Lambda H$ is porous in $\bdry G$,
		\item $\Lambda H \subset \bdry G$ is a proper subset, and
		\item $H$ is infinite index in $G$.
	\end{enumerate}
\end{proposition}
\begin{proof}
	The implication (1) $\Longrightarrow$ (2) is trivial.
	Likewise, (2) $\Longrightarrow$ (3) is straightforward: if $[G:H]<\infty$ then there is a bounded fundamental domain for the action of $H$ on a Cayley graph $X$ for $G$, and so for some constant $C$, every point of $X$ is within a distance $C$ of $H \subset X$, and thus $\Lambda H = \bdry X = \bdry G$.

	It remains to show (3) $\Longrightarrow$ (1).
	We fix a Cayley graph $X$ for $G$.
	Since $H$ acts freely on $G$, the quotient $H\backslash X$ is a regular graph of bounded degree, with a vertex for each right coset $Hg$.
	As $[G:H]=\infty$, $H\backslash X$ has infinite diameter, and so we can find a sequence of points $g_i\in G\subset X$, $i \in \N$, so that $d(H,g_i)\to\infty$ as $i\to\infty$.
	Suppose for each $i$ that $h_i \in H \subset X$ is a closest point in $H$ to $g_i$.
	Let $\gamma_i:[0,d(h_i,g_i)]\to X$ be a geodesic from $h_i$ to $g_i$.
	By the choice of $h_i$ for each $t \in [0,d(h_i,g_i)]$, $d(\gamma_i(t),H) \geq t$.
	Let $\beta_i = h_i^{-1}\gamma_i$, so that $\beta_i(0)=1$, and still for each $t$ in the domain of each $\beta_i$, $d(H, \beta_i(t)) \geq t$.
	We apply Arzel\`a--Ascoli to the sequence of maps $(\beta_i)$ to find a subsequence that converges uniformly on compact intervals to a map $\beta:[0,\infty)\to X$.
	This map $\beta$ will be a geodesic ray, and will inherit the property that $d(\beta(t),H) \geq t$ for all $t \in [0,\infty)$.

	Now to show porosity: fix a visual metric $\rho$ on $\bdry X = \bdry G$, with visual parameter $\epsilon>0$ and constant $C_1$, so that $\rho(x,y) \asymp_{C_1} e^{-\epsilon (x|y)_1}$.  Suppose $H$ is $C_2$-quasiconvex: any geodesic with endpoints on $H$ lies in $N_{C_2}H$; this will also be true for a geodesic ray from $1\in H$ to a point of $\Lambda H$.  Finally, write $\delta_X$ for the hyperbolicity constant of $X$.
	
	We want to find $c>0$ so that given $y \in \Lambda H$ and $r \leq \diam(\bdry X)$, there exists $x \in \bdry X$ with (i) for any $y' \in \Lambda H$, $\rho(x,y')\geq cr$, and (ii) for any $x' \in \bdry X$ with $\rho(x,x')\leq cr$, $\rho(x',y) \leq r$.
	
	We will set $c=e^{-\epsilon (A_1+A_2)}$, where $A_1$ and $A_2$ are parameters depending only $\epsilon, C_1, C_2, \delta_X$ found below.
	Given $y \in \Lambda H$ and $r \leq \diam \bdry X$, fix a geodesic ray $\alpha$ from $1$ representing $y$.
	Consider the point of $\alpha$ at distance $\frac{-1}{\epsilon}\log r + A_1$ from $1$, and let $h \in H$ be a point within $C_2$ from that point.
	Let $x \in \bdry X$ be the limit point of $h\beta$.

	We show that (i) holds.
	For $y'\in \Lambda H$, if $\rho(x,y') < e^{-\epsilon(A_1+A_2)} r$ then $(x|y')_1 \geq \frac{-1}{\epsilon}\log r + A_1+A_2-C_3$ for some $C_3=C_3(\epsilon,C_1)$, so the geodesics from $1$ to $x$ and to $y'$ stay $2\delta_X$-close for all times up to this value.
	But this is a contradiction for large $A_2$ since the geodesic from $1$ to $y'$ lies in $N_{C_2}H$, while at times $t \geq \frac{-1}{\epsilon}\log r +A_1$, the geodesic from $1$ to $x$ has distance at least $t-(-\epsilon^{-1}\log(r)+A_1)-C_4$ from $H$ for $C_4=C_4(\delta_X,\epsilon,C_1,C_2)$.  See Figure~\ref{fig:porous1}.
\begin{figure}
	\def\svgwidth{.5\textwidth}
\begingroup%
  \makeatletter%
  \providecommand\color[2][]{%
    \errmessage{(Inkscape) Color is used for the text in Inkscape, but the package 'color.sty' is not loaded}%
    \renewcommand\color[2][]{}%
  }%
  \providecommand\transparent[1]{%
    \errmessage{(Inkscape) Transparency is used (non-zero) for the text in Inkscape, but the package 'transparent.sty' is not loaded}%
    \renewcommand\transparent[1]{}%
  }%
  \providecommand\rotatebox[2]{#2}%
  \newcommand*\fsize{\dimexpr\f@size pt\relax}%
  \newcommand*\lineheight[1]{\fontsize{\fsize}{#1\fsize}\selectfont}%
  \ifx\svgwidth\undefined%
    \setlength{\unitlength}{247.24221057bp}%
    \ifx\svgscale\undefined%
      \relax%
    \else%
      \setlength{\unitlength}{\unitlength * \real{\svgscale}}%
    \fi%
  \else%
    \setlength{\unitlength}{\svgwidth}%
  \fi%
  \global\let\svgwidth\undefined%
  \global\let\svgscale\undefined%
  \makeatother%
  \begin{picture}(1,0.60327083)%
    \lineheight{1}%
    \setlength\tabcolsep{0pt}%
    \put(0,0){\includegraphics[width=\unitlength,page=1]{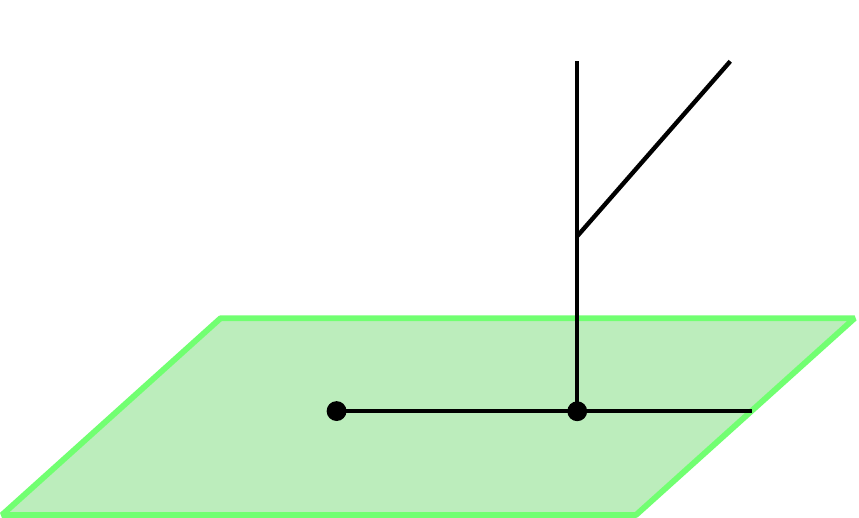}}%
    \put(0.32353403,0.10842177){\color[rgb]{0,0,0}\makebox(0,0)[lt]{\lineheight{1.25}\smash{\begin{tabular}[t]{l}$1$\end{tabular}}}}%
    \put(0.64367749,0.05870907){\color[rgb]{0,0,0}\makebox(0,0)[lt]{\lineheight{1.25}\smash{\begin{tabular}[t]{l}$h$\end{tabular}}}}%
    \put(0.90319699,0.10357719){\color[rgb]{0,0,0}\makebox(0,0)[lt]{\lineheight{1.25}\smash{\begin{tabular}[t]{l}$y$\end{tabular}}}}%
    \put(0.64852196,0.56199226){\color[rgb]{0,0,0}\makebox(0,0)[lt]{\lineheight{1.25}\smash{\begin{tabular}[t]{l}$x$\end{tabular}}}}%
    \put(0.83528365,0.56199226){\color[rgb]{0,0,0}\makebox(0,0)[lt]{\lineheight{1.25}\smash{\begin{tabular}[t]{l}$x'$\end{tabular}}}}%
    \put(0.11342978,0.24464903){\color[rgb]{0,0,0}\makebox(0,0)[lt]{\lineheight{1.25}\smash{\begin{tabular}[t]{l}$H$\end{tabular}}}}%
    \put(0,0){\includegraphics[width=\unitlength,page=2]{porous1.pdf}}%
    \put(0.87370349,0.25211729){\color[rgb]{0,0,0}\makebox(0,0)[lt]{\lineheight{1.25}\smash{\begin{tabular}[t]{l}$y'$\end{tabular}}}}%
  \end{picture}%
\endgroup%

	\caption{The configuration of the points $x$, $x'$ and $y$, with a potential location for $y'$}
	\label{fig:porous1}
\end{figure}
	
	We show that (ii) holds.
	If $\rho(x,x')\leq cr=e^{-\epsilon(A_1+A_2)}$ then $(x|x')_1 \geq \frac{-1}{\epsilon}\log r +A_1+A_2-C_3$.
	If $A_2$ is large enough, 
	the tree approximation to $1$, $y$, $x$ and $x'$
	must look like Figure~\ref{fig:porous1}.
	In particular, $(x'|y)_1$ equals $\frac{-1}{\epsilon}\log r+A_1$ up to an additive error $C_5$.
	Thus 
	$ \rho(x',y) \leq C_1 e^{-\epsilon(x'|y)} \leq C_1 e^{-\epsilon A_1}e^{\epsilon C_5}r$.
	Provided $A_1$ is chosen large enough depending on $C_1$, $\epsilon$, $C_5$, we have $\rho(x',y)\leq r$ as desired.
\end{proof}

As an aside, this implies that hyperbolic groups which attain their conformal dimension satisfy a kind of ``co-Hopfian'' property; compare the variations discussed in Kapovich--Lukyanenko~\cite{Kap-Luk-12-qi-cohopf} and Stark--Woodhouse~\cite{Sta-Wood-hyp-cohopf}.  (The second author thanks Woodhouse for asking him this question.)
\begin{corollary}
	\label{cor:qcvx-cohopf}
	If $G$ is a hyperbolic group, and $\bdry G$ attains its (Ahlfors regular) conformal dimension, then no finite index subgroup of $G$ is isomorphic to a quasiconvex infinite index subgroup of $G$.
\end{corollary}
\begin{proof}
	Suppose $H_1, H_2 \leq G$ are isomorphic (indeed, it suffices that they are quasi-isometric) with $[G:H_1]<\infty$ and $[G:H_2]=\infty$.
	By Proposition~\ref{prop:qcvx-porous}, $\Lambda H_2$ is porous in $\bdry G$.
	But $\Lambda H_2$ and $\bdry G$ are quasisymmetric, and hence each attains their conformal dimension, which contradicts Proposition~\ref{prop:porous-non-attain}.
\end{proof}
We return to our main goal, of characterising the attainment of conformal dimension for a hyperbolic graph of groups with finite or 2-ended edge groups.
\begin{proof}[Proof of Theorem~\ref{thm:main-attained}]
	Suppose $G$ is a hyperbolic group so that $\Confdim \bdry G$ is attained, and with a graph of groups decomposition over finite or $2$-ended subgroups.

	If $\Confdim \bdry G = 0$ then $G$ is virtually free by Stallings--Dunwoody, and as the conformal dimension is attained $G$ is $2$-ended, see e.g.~\cite[Theorem 3.4.6]{MacTys10-confdim}.
	If $\Confdim \bdry G = 1$ is attained, then $G$ is virtually a cocompact Fuchsian group by, e.g., a result of Bonk--Kleiner~\cite[Theorem 1.1]{Bonk-Kleiner-02-rigidity-QM-actions}.

	We are left with the case that $\Confdim \bdry G >1$ is attained, and so by Theorem~\ref{thm:main} is equal to $\Confdim \bdry G_i$ for some vertex group $G_i$.
	Let $T$ be the Bass--Serre tree for the given graph of groups decomposition $\cG$.  
	Each vertex of $T$ is stabilized by a conjugate of a vertex group.
	If $T$ has infinite diameter, then there are infinitely many such vertices stabilized by conjugates of $G_i$, each corresponding to a left coset of $G_i$, so $[G:G_i]=\infty$, but this contradicts Proposition~\ref{prop:porous-non-attain}.  So $T$ has finite diameter.

	If there were any loops in $\cG$ then $T$ would have infinite diameter, so $\cG$ must be a tree.
	Consider a leaf of $\cG$ where the vertex group is $G_v$ and the adjacent edge group $G_e$.
	Let $x, y \in T$ be vertices stabilized by $G_i$ and $G_v$ respectively, and $\gamma\subset T$ the simple path connecting them.

	If the injection $i_e:G_e\to G_v$ has proper image, then the index $[G_v:i(G_e)] \geq 2$, so $y$ has degree $\geq 2$, and there is an element $g_1$ of $G_v \leq G$ which fixes $y$ but moves the rest of $\gamma$.
	Since $\Confdim \bdry G_i >1$, the edge groups adjacent to $G_i$ have infinite index in $G_i$, so again there is an element $g_2$ of $G_i \subset G$ which fixes $x$ but moves the rest of $\gamma$.
	Alternating $g_1$ and $g_2$, one shows that $T$ contains an unbounded line and hence has infinite diameter, a contradiction.

	So the injection $i_e:G_e\to G_v$ is an isomorphism, and we can remove $v$ and $e$ from $\cG$ without changing $G$.
	We continue to do this process, removing leafs, until only $G_i$ is left, and thus $G = G_i$.
\end{proof}

\bibliography{biblio}

\end{document}